\documentclass[12pt]{amsart}
\usepackage{amssymb}
\usepackage{mathtools}
\usepackage{txfonts}
\usepackage{eucal}
\usepackage{extarrows}

\usepackage[all]{xy}

\usepackage{tikz}

\usepackage{graphicx}
\usepackage{float}

\usepackage{xspace}

\usepackage[a4paper,body={16.3cm,22.8cm},centering]{geometry}
\usepackage[colorlinks,final,backref=page,hyperindex,pdftex]{hyperref}

\newcommand{\nc}{\newcommand}
\newcommand{\delete}[1]{}

\nc{\mlabel}[1]{\label{#1}}  
\nc{\mcite}[1]{\cite{#1}}  
\nc{\mref}[1]{\ref{#1}}  
\nc{\mbibitem}[1]{\bibitem{#1}} 

\delete{
\nc{\mlabel}[1]{\label{#1}  
{\hfill \hspace{1cm}{\small\tt{{\ }\hfill(#1)}}}}
\nc{\mcite}[1]{\cite{#1}{\small{\tt{{\ }(#1)}}}}  
\nc{\mref}[1]{\ref{#1}{{\tt{{\ }(#1)}}}}  
\nc{\mbibitem}[1]{\bibitem[\bf #1]{#1}} 
}

\newtheorem{theorem}{Theorem}[section]
\newtheorem{prop}[theorem]{Proposition}
\newtheorem{lemma}[theorem]{Lemma}
\newtheorem{coro}[theorem]{Corollary}

\theoremstyle{definition}
\newtheorem{defn}[theorem]{Definition}
\newtheorem{remark}[theorem]{Remark}
\newtheorem{exam}[theorem]{Example}
\newtheorem{prop-def}{Proposition-Definition}[section]

\newcommand\cal[1]{\mathcal{#1}}

\newcommand\alphlist{a,b,c,d,e,f,g,h,i,j,k,l,m,n,o,p,q,r,s,t,u,v,w,x,y,z}
\newcommand\Alphlist{A,B,C,D,E,F,G,H,I,J,K,L,M,N,O,P,Q,R,S,T,U,V,W,X,Y,Z}
\newcommand\getcmds[3]{\expandafter\newcommand\csname #2#1\endcsname{#3{#1}}}
\makeatletter
\@for\x:=\alphlist\do{\expandafter\getcmds\expandafter{\x}{frak}{\mathfrak}}
\@for\x:=\Alphlist\do{\expandafter\getcmds\expandafter{\x}{frak}{\mathfrak}}
\makeatother

\nc{\bfk}{{\bf k}}
\font\cyr=wncyr10

\newfont{\scyr}{wncyr10 scaled 550}
\nc{\sha}{\mbox{\cyr X}}
\nc{\ssha}{\mbox{\bf \scyr X}}

\nc{\id}{\mathrm{id}}
\nc{\Id}{\mathrm{Id}}
\nc{\lbar}[1]{\overline{#1}}
\nc{\ot}{\otimes}
\nc{\dep}{\mathrm{dep}}

\nc{\tred}[1]{\textcolor{red}{#1}} \nc{\tgreen}[1]{\textcolor{green}{#1}}
\nc{\tblue}[1]{\textcolor{blue}{#1}} \nc{\tpurple}[1]{\textcolor{purple}{#1}}

\nc{\li}[1]{\tpurple{\underline{Li:}#1 }}
\nc{\liadd}[1]{\tpurple{#1}}
\nc{\xing}[1]{\tblue{\underline{Xing:}#1 }}
\nc{\dominique}[1]{\tblue{\underline{Dominique: }#1 }}
\nc{\yuan}[1]{\tred{\underline{Yuan:}#1 }}
\nc{\markus}[1]{\tred{\underline{Markus:} #1}}

\usetikzlibrary{calc}
\newlength\xch
\newsavebox\dbox
\sbox\dbox{\tikz{\fill (0,0) circle (0.05cm);}}
\newif\ifqdd
\newif\ifzdd
\newcommand\cddf[3]{%
\coordinate (#2) at ($(#1)+(#3)$);
\draw (#1)--(#2);
\ifqdd\node at (#1) {\usebox\dbox};\fi
\ifzdd\node at (#2) {\usebox\dbox};\fi}
\newcommand\cdx[4][1]{\cddf{#2}{#3}{#4:#1*\xch}}
\newcommand\cdl[2][1]{\cdx[#1]{#2}{#2l}{135}}
\newcommand\cdr[2][1]{\cdx[#1]{#2}{#2r}{45}}
\newcommand\cdlr[2][1]{%
\foreach \i in {#2} {\cdl[#1]{\i}\cdr[#1]{\i}}}
\newcommand\cda[2][1]{\cdx[#1]{#2}{#2a}{90}}
\newcommand\cdb[2][1]{\cdx[#1]{#2}{#2b}{-90}}
\let\treeoo\treeo%
\newcommand\zhongdian[2]{\node at ($(#1)!0.5!(#2)$) {\usebox\dbox};}
\newcommand\zhd[1]{\foreach \i/\j in {#1} {\zhongdian{\i}{\i\j}}}

\newcommand\ocdx[6][1]{%
\node[draw,circle,minimum size=2pt,label={#6:$#5$}]
(#3) at ($(#2)+(#4:#1*\xch)$) {};
\draw (#2)--(#3);}
\newcommand\ocdl[3][1]{\ocdx[#1]{#2}{#2l}{135}{#3}{above}}
\newcommand\ocdr[3][1]{\ocdx[#1]{#2}{#2r}{45}{#3}{above}}

\newcommand\scopeclip[1]{\begin{scope}
\clip(-1.1,-0.5)rectangle(1.1,1);#1\end{scope}}
\newcommand\XX[2][]{%
\tikz[line width=0.15ex,x=0.5cm,y=0.5cm,baseline,inner sep=1.5pt,
every node/.style={font=\scriptsize},#1]{
\scopeclip{\draw (135:1.5)--(0,0)--(45:1.5) (0,-0.5)--(0,0);}#2}}
\newcommand\xx[3]{%
\scopeclip{\draw(#1/10,#2/10)--+(#3*45:2.5);}}
\newcommand\xxl[2]{\xx{#1}{#2}3}
\newcommand\xxr[2]{\xx{#1}{#2}1}
\newcommand\xxlr[2]{\xxl{#1}{#2}\xxr{#1}{#2}}
\newcommand\xxh[6]{
\draw(#1/10,#2/10)+(0.5*#3*45+0.5*#4*45:#6) node[above] {$#5$};}
\newcommand\xxhu[4][0.15]{\xxh{#2}{#3}13{#4}{#1}}
\newcommand\stree[1]{\XX{\xxhu[0.25]00{#1}}}

\nc{\dnx}{\Delta_n A} \nc{\dx}{\Delta A} \nc{\dgp}{{\rm deg_{P}}}
\nc{\dgt}{{\rm deg_{T}}} \nc{\dg}{{\rm deg}} \nc{\ida}{ID($A$)} \nc{\tu}{\tilde{u}} \nc{\tv}{\tilde{v}}
\nc{\nr}{\calr_n} \nc{\nz}{\calz_n} \nc{\fun}{\cala_{n,d}}
 \nc{\fbase}{\calb} \nc{\LF}{\mathrm{RF}} \nc{\FFA}{\mathrm{LF}} \nc{\irr}{\mathrm{Irr}}
 \nc{\result}{\bfk\mathrm{Irr}(S_n)}  \nc{\I}{I_{\mathrm{ID},n}^0}
 \nc{\nrs}{\calr_n^\star} \nc{\ii}{\mathrm{I}} \nc{\iii}{\mathrm{II}}
\nc{\intl}{{\rm int}}\nc{\ws}[1]{{#1}}\nc{\deleted}[1]{\delete{#1}}\nc{\plas}{placements\xspace}

\nc{\bim}[1]{#1}  \nc{\shaop}{\sha_{\Omega}^{+}}  \nc{\shao}{\sha_{\Omega}}
\nc{\bbim}[2]{#1 #2} \nc{\bbbim}[2]{#1,\, #2} \nc{\RBF}{{\rm RBF}}
\nc{\frb}{F_{\RB}} \nc{\shaf}{\ssha_{\tiny{\Omega}}} \nc{\sham}{\diamond_{\tiny{\Omega}}}
\nc{\lf}{\lfloor} \nc{\rf}{\rfloor} \nc{\shan}{\ssha_{\lambda}}
\nc{\rlex}{{\rm {lex}}} \nc{\bb}{\Box} \nc{\ra}{\rightarrow}
\nc{\e}{{\rm {e}}}
\nc{\DDF}{\mathrm{DD}(X,\,\Omega)}\nc{\DTF}{\mathrm{DT}(X,\,\Omega)} \nc{\DT}{\mathrm{DT}'(\Omega,\,V)}
\nc{\bra}{\mathrm{bra}} \nc{\bre}{\mathrm{bre}}
\nc{\dec}{\mathrm{dec}} \nc{\diamondw}{\diamond_{w}}
\nc{\type}{\mathrm{type}}

\nc\caF[1]{\cal{F}_{#1}(X,\,\Omega)}
\nc\calt{\cal{T}(X,\,\Omega)} \nc\caltn{\cal{T}_n(X,\,\Omega)}
\nc\caltbin{\cal{T}_b(X,\,\Omega)}
\nc\calta{\cal{T}_0(X,\,\Omega)}
\nc\caltb{\cal{T}_1(X,\,\Omega)}
\nc\caltc{\cal{T}_2(X,\,\Omega)}
\nc\caltd{\cal{T}_3(X,\,\Omega)}
\nc\caltm{\cal{T}_m(X,\,\Omega)}
\nc\calf{\cal{F}(X,\,\Omega)}
\nc\fram{\frak{M}(\Omega,\, X)}
\nc\shaw{\sha^{NC}_w(\Omega,\, X)}
\nc\dw{\diamond_w} \nc\dl{\diamond_\ell}
\nc\shal{\sha^{NC}_\ell(X,\, \Omega)} \nc\shav{\sha^{NC}_w(\Omega,\, V)} \nc\shat{\sha^{NC,1}_w(\Omega,\, T^{+}(V))}
\nc{\cfo}{\cal{F}(X,\,\Omega)}

\nc{\lar}{\varinjlim}
\nc\XO{(X,\,\Omega)}
\def\cxo#1#2;{\cal{#1}#2\XO}
\def\cxob#1#2;{\cal{#1}#2_b\XO}
\nc\lrf[2]{B_{#2}^+(#1)}
\nc{\fd}{\mathrm{\text{typed angularly decorated planar rooted trees}}}
\nc{\rb}{\mathrm{RBFWs}} \nc{\dfw}{\mathrm{DFW{(X)}}} \nc{\tfw}{\mathrm{TFW{(X)}}}
\nc{\tfv}{\mathrm{TFW{(V)}}} \nc{\rbf}{\mathrm{RBF}}

\def\Ve#1,#2,#3;{\vee_{#1,\,(#2,\,#3)}}
\def\bigv#1;#2;#3;{\bigvee\nolimits_{#1}^{#2;\,#3}}

\begin{document}

\title[Free Rota-Baxter family algebras and Free (tri)dendriform family algebras]{Free Rota-Baxter family algebras and Free (tri)dendriform family algebras}
%
%

\author{Yuanyuan Zhang} \address{School of Mathematics and Statistics,
Lanzhou University, Lanzhou, 730000, P. R. China}
\email{zhangyy17@lzu.edu.cn}
\author{Xing Gao}
\address{School of Mathematics and Statistics,
Key Laboratory of Applied Mathematics and Complex Systems,
Lanzhou University, Lanzhou, 730000, P.R. China}
\email{gaoxing@lzu.edu.cn}
\author{Dominique Manchon}
\address{Laboratoire de Math\'ematiques Blaise Pascal,
CNRS--Universit\'e Clermont-Auvergne,
3 place Vasar\'ely, CS 60026,
F63178 Aubi\`ere, France}
\email{Dominique.Manchon@uca.fr}

\date{\today}

\begin{abstract}
 In this paper, we first construct the free Rota-Baxter family algebra generated by some set $X$ in terms of typed angularly $X$-decorated planar rooted trees.
As an application, we obtain a new construction of the free Rota-Baxter algebra only in terms of
angularly decorated planar rooted trees (not forests), which is quite different from the known
construction via angularly decorated planar rooted forests by K. Ebrahimi-Fard and L. Guo.
We then embed the free dendriform (resp. tridendriform) family algebra into the free Rota-Baxter family algebra of weight zero (resp. one). Finally, we prove that the free Rota-Baxter family algebra is the universal enveloping algebra of the free (tri)dendriform family algebra.
\end{abstract}

\subjclass[2010]{
16W99, 
16S10, 
13P10, 
08B20, 
}

\keywords{Rota-Baxter family algebra, (tri)dendriform family algebra, typed decorated planar rooted trees, free Rota-Baxter family algebras, free (tri)dendriform family algebras.}

\maketitle

\tableofcontents

\setcounter{section}{0}

\allowdisplaybreaks
\section{Introduction}
Algebraic structures may appear in ``family versions", where the operations at hand are replaced by a family of operations indexed by some set $\Omega$, in general endowed with a semigroup structure. The first example of this situation appeared in 2007 in a paper by K. Ebrahimi-Fard, J. Gracia-Bondia and F. Patras \cite[Proposition~9.1]{FBP} (see also \cite[Theorem 3.7.2]{DK}) about algebraic aspects of renormalization in Quantum Field Theory, where a ``Rota-Baxter family" appears: this terminology was suggested to the authors by Li Guo (see Footnote following Proposition 9.2 therein), who further duscussed the underlying structure under the name \textsl{Rota-Baxter Family algebra} in \cite{Guo09}. Various other kinds of family algebraic structures have been recently defined \cite{Foi20, ZG, ZGM, ZM}.
\begin{defn}\mlabel{def:pp}
Let $\Omega$ be a semigroup, let $\bfk$ be a unital commutative ring and let $\lambda\in \bfk$ be given.
A {\bf Rota-Baxter family} of weight $\lambda$ on a $\bfk$-algebra $R$ is a collection of linear operators $(P_\omega)_{\omega\in\Omega}$ on $R$ such that
\begin{equation}
P_{\alpha}(a)P_{\beta}(b)=P_{\alpha\beta}\left( P_{\alpha}(a)b  + a P_{\beta}(b) + \lambda ab \right),\, \text{ for }\, a, b \in R\,\text{ and }\, \alpha,\, \beta \in \Omega.
\mlabel{eq:RBF}
\end{equation}
The pair $\big(R,\, (P_\omega)_{\omega\in\Omega}\big)$ is called a {\bf Rota-Baxter family algebra} of weight $\lambda$. The example considered in \cite{FBP, Guo09, DK} is the algebra of Laurent series $R=\bfk[z^{-1},z]]$. It is a Rota-Baxter family algebra of weight $-1$, with $\Omega=(\mathbb Z,+)$, where the operator $P_\omega$ is the projection onto the subspace $R_{<\omega}$ generated by $\{z^k,\,k<\omega\}$ parallel to the supplementary subspace $R_{\ge \omega}$ generated by $\{z^k,\,k\ge\omega\}$.\\
\end{defn}
Rota-Baxter algebras (in the ordinary sense) are Rota-Baxter family algebras indexed by the trivial semigroup reduced to one element. They find their origin in the work of the American mathematician Glen E. Baxter~\mcite{Bax} in the realm of probability theory. The free Rota--Baxter algebras were constructed in~\cite{GK2,Guo12}, as well as in~\cite{EFG} and~\mcite{GK3}. The authors in~\cite{ZG} constructed free commutative Rota-Baxter family algebras and free noncommutative Rota-Baxter family algebras via the method of Gr\"{o}bner-Shirshov bases.\\

In this paper, we use planar rooted trees, both angularly decorated by the set $X$ of generators and typed by the semigroup $\Omega$ (see Definition \ref{defn:tdtree1} below), to provide an alternative realization of free Rota-Baxter family algebras (Theorem \ref{thm:free}). The model we propose here is inspired by our recent description of free dendriform and tridendriform family algebras \cite{ZG, ZGM}, which we now briefly outline: J.-L. Loday proposed the concept of dendriform algebra~\cite{Lod93} with motivation from algebraic K-theory. J.-L.~Loday and M.~Ronco later introduced the concept of a tridendriform algebra~\cite{LoRo98} in the study of polytopes and Koszul duality. Dendriform and tridendriform family algebras (see Definitions \ref{defn:dend} and \ref{defn:tridend} below) were introduced in~\cite{ZG}. Free tridendriform family algebra were constructed in ~\cite{ZGM} in terms of typed angularly decorated Schr\"oder trees, also known as planar reduced trees. They are exactly the same trees as above, except that each vertex has at least two incoming edges. Free dendriform family algebras were similarly constructed in terms of typed decorated planar binary trees.\\

It is well known that a Rota--Baxter algebra of weight zero (resp.~weight $\lambda\neq 0$) possesses a dendriform (resp.~tridendriform) algebra structure~\mcite{Agu99, EbFa}. This also holds for the family versions \cite{ZG}. Universal enveloping Rota-Baxter algebras of (tri)dendriform algebras have been studied in ~\cite{EFG08}. Whenever $\bfk$ is a field, we prove, using the natural inclusion of binary trees and Schr\"oder trees into planar rooted trees, that the free Rota-Baxter family algebra of weight $\lambda= 0$ (resp. $\lambda\neq 0$) generated by a set $X$ is the universal enveloping algebra of the free dendriform (resp. tridendriform) family algebra generated by $X$ (Theorem \ref{thm:ddf} and Theorem \ref{thm:dtf}).\\
\delete{
This enables us to consider the relationship between free Rota-Baxter family algebras and free (tri)dendriform family algebras. Ebrahimi-Fard and Guo considered the universal enveloping algebras of (tri)dendriform algebras in terms of words in,
which inspires us to study the universal enveloping Rota-Baxter family algebra of the free (tri)dendriform family algebra in terms of angularly decorated planar rooted trees .\\
}

The paper is organized as follows: In Section~\mref{sec:FRBF}, we first define a multiplication on the linear span of the set $\calt$ of typed angularly decorated planar rooted trees. Second, we introduce the set  $\frak{X}_\infty$ of Rota-Baxter family words, and contruct a bijection $\psi$ from $\frak{X}_\infty$ onto $\calt$ (Proposition~\mref{prop:iso}). Finally, using the bijection $\psi$ and the results of \cite{ZG}, we prove that the linear span $\bfk\calt$ is the free Rota-Baxter family algebra generated by $X$ (Theorem~\mref{thm:free}). In Section~\mref{sec:emb}, we show that the natural inclusion of typed decorated planar binary trees into typed angularly decorated planar rooted trees makes the free dendriform family algebra on $X$ a canonical dendriform family subalgebra of the free Rota-Baxter family algebra $\bfk\calt$ of weight $0$ (Theorem ~\mref{thm:sub1}). Similarly, the natural inclusion of typed decorated Schr\"oder trees makes the free tridendriform family algebra on $X$ a canonical tridendriform family subalgebra of the free Rota-Baxter family algebra $\bfk\calt$ of weight $1$ (Theorem ~\mref{thm:main}). At the same time, we prove that the free Rota-Baxter family algebra of weight zero (resp. one) is the universal enveloping algebra of the free dendriform (resp.~tridendriform) family algebra Theorem~\mref{thm:ddf} (resp. Theorem~\mref{thm:dtf}). Whenever $\bfk$ is a field, the results outlined above for weight one are also valid for weight $\lambda$ modulo multiplicating the embedding map by $\lambda^{-1}$.\\

{\bf Notation.}
 Throughout this paper, let $\bfk$ be a unitary commutative ring which will be the base ring of all modules, algebras, as well as linear maps, unless otherwise specified.
 Algebras are unitary associative algebras but not necessary commutative. Let RBFAs,
 DFAs and TFAs be short for Rota-Baxter family algebras, dendriform family algebras and tridendriform family algebras, respectively.
\section{Free Rota-Baxter family algebras}
\mlabel{sec:FRBF}
In this section, we construct free Rota-Baxter family algebras in terms of typed angularly decorated planar rooted trees.

\subsection{Typed angularly decorated planar rooted trees}
\mlabel{sec:RBF}
Typed decorated rooted trees~\cite{BHZ,Foi18} are rooted trees with vertex decoration and edge decoration.
For a rooted tree $T$, denote by $V(T)$ (resp. $E(T)$) the set of its vertices (resp. edges). In this subsection, we introduce typed angularly decorated planar rooted trees.
\begin{defn}\mlabel{defn:tdtree}\cite{BHZ}
Let $X$ and $\Omega$ be two sets. An {\bf $X$-decorated $\Omega$-typed (abbreviated typed decorated) rooted tree}
is a triple $T = (T,\dec, \type)$, where
\begin{enumerate}
\item $T$ is a rooted tree.
\item $\dec: V(T)\ra X$ is a map.
\item $\type: E(T)\ra \Omega$ is a map.
\end{enumerate}
\end{defn}

In~\cite{ZGM}, typed decorated planar binary trees were used to construct free dendriform family algebras. Along these lines,
in the present paper, we use typed decorated planar rooted trees to construct free Rota-Baxter family algebras
which have close relationship with free dendriform family algebras. So we adopt the graphical representation of planar rooted trees as the one of planar binary trees in~\cite{ZGM},
which is slightly different from the graphical representation of rooted trees used in \cite{BHZ,Foi18}: the root and the leaves are now edges rather than vertices.
See examples below. So the set $E(T)$ must be replaced by the set $IE(T)$ of internal edges, i.e. edges which are neither leaves nor the root. However, a planar rooted tree has at least one vertex. Each vertex $v$ yields a (possibly empty) set of angles $A(v)$, an angle being a pair $(e,e')$ of adjacent incoming edges for $v$. Let $A(T)=\bigsqcup_{v\in V(T)}A(v)$ be the set of angles of the planar rooted tree $T$. We propose the following concept:
\begin{defn}\mlabel{defn:tdtree1}
Let $X$  and $\Omega$ be two sets.  An {\bf $X$-angularly decorated $\Omega$-typed (abbreviated typed angularly decorated) planar rooted tree}
is a triple $T = (T,\dec, \type)$, where
\begin{enumerate}
\item $T$ is a planar rooted tree.
\item $\dec: A(T)\ra  X$ is a map.
\item $\type: IE(T)\ra \Omega$ is a map.
\end{enumerate}
\end{defn}

\delete{
Rota-Baxter family algebras, proposed by Guo, are a generalization of Rota-Baxter algebras and
arise naturally in renormalization of quantum field theory~\cite{DK,FBP}.

\begin{defn}\mlabel{def:pp}
Let $\Omega$ be a semigroup and $\lambda\in \bfk$ be given.
A {\bf Rota-Baxter family} of weight $\lambda$ on an algebra $R$ is a collection of linear operators $(P_\omega)_{\omega\in\Omega}$ on $R$ such that
\begin{equation}
P_{\alpha}(a)P_{\beta}(b)=P_{\alpha\beta}\left( P_{\alpha}(a)b  + a P_{\beta}(b) + \lambda ab \right),\, \text{ for }\, a, b \in R\,\text{ and }\, \alpha,\, \beta \in \Omega.
\mlabel{eq:RBF}
\end{equation}
Then the pair $(R,\, (P_\omega)_{\omega\in\Omega})$ is called a {\bf Rota-Baxter family algebra} of weight $\lambda$.
\end{defn}

When the semigroup $\Omega$ is trivial, a Rota-Baxter family algebra degenerates to a Rota-Baxter algebra,
which is originated by the American mathematician Glen E. Baxter~\mcite{Bax} in the realm of probability theory.}

If a semigroup $\Omega$ has no identity element, we consider the monoid $\Omega^{1}:=\Omega\sqcup \{1\}$
obtained from $\Omega$ by adjoining an identity:
 $$1\omega:=\omega1:=\omega,\,\text{ for }\,\omega\in\Omega\,\text{ and }\, 11:=1.$$

For $n\geq 0$, let $\caltn$ denote the set of  $X$-angularly decorated $\Omega^{1}$-typed planar rooted trees with $n+1$ leaves such that leaves are decorated by the identity 1 in $\Omega^{1}$ and internal edges are decorated by elements of $\Omega.$ Note that the root is not decorated. Denote by
$$\calt:= \bigsqcup_{n\geq 0}\caltn\,\text{ and }\, \bfk\calt :=\bigoplus_{n\geq 0}\bfk\caltn.$$
Here are some examples in $\calt$.
\begin{align*}
\calta=& \ \{\treeoo{\cda[1.5] o\zhd{o/a}
\node[right] at ($(o)!0.7!(oa)$) {};
},\,
\treeoo{\cda[1.0] o\zhd{o/a}\cdb[1.0]o\zhd{o/b}
\node[right] at ($(o)!0.7!(oa)$) {};
\node[right] at ($(o)!0.2!(ob)$) {$\omega_1$};
},\,
\treeoo{\cda[1.0] o\zhd{o/a}\cdb[1.0]o\zhd{o/b}
\node[right] at ($(o)!0.7!(oa)$) {};
\node[right] at ($(o)!0.2!(ob)$) {$\omega_1$};
\cdb{ob}\zhd{ob/b}
\node[right] at ($(ob)!0.1!(obb)$) {$\omega_2$};
},\,
\cdots\big|\omega_1,\omega_2,\ldots\in\Omega\},\,\,\\
\caltb =&\left\{\stree x,\,\treeoo{\cdb[1.5] o\zhd{o/b}\cdlr o
\node[above=1pt] at (o) {$x$};
\node[right] at ($(o)!0.3!(ob)$) {$\omega$};
},\,
\treeoo{\cdb[1.5] o\zhd{o/b}\cdlr[1.5] o \zhd{o/l}
\node[above=1pt] at (o) {$x$};
\node[right] at ($(o)!0.3!(ob)$) {$\omega$};
\node[left] at ($(o)!0.1!(ol)$) {$\alpha$};
},\,
\treeoo{\cdb[1.5] o\zhd{o/b}\cdlr[1.5] o \zhd{o/l}\zhd{o/r}
\node[above=1pt] at (o) {$x$};
\node[right] at ($(o)!0.3!(ob)$) {$\omega$};
\node[left] at ($(o)!0.1!(ol)$) {$\alpha$};
\node[right] at ($(o)!0.1!(or)$) {$\beta$};
},\,
\treeoo{\cdb[1.5] o\zhd{o/b}\cdlr[1.5] o 
\node at ($(o)!0.33!(ol)$) {\usebox\dbox};
\node at ($(o)!0.66!(ol)$) {\usebox\dbox};
\node[above=1pt] at (o) {$x$};
\node[right] at ($(o)!0.3!(ob)$) {$\omega$};
\node[left] at ($(o)!0.01!(ol)$) {$\alpha$};
\node[left] at ($(o)!0.3!(ol)$) {$\beta$};
},\,
\treeoo{\cdb[1.5] o\zhd{o/b}\cdlr o
\node[above=1pt] at (o) {$x$};
\node[right] at ($(o)!0.3!(ob)$) {$\omega$};
\cdb{ob}\zhd{ob/b}
\node[right] at ($(ob)!0.1!(obb)$) {$\alpha$};
},
\cdots\big| x\in X,\alpha,\beta,\omega,\cdots\in\Omega \right\},\\
 \caltc=& \ \left\{
\treeoo{\cdb o\zhd{o/b}\cdlr{o,ol}[1.5]\zhd{o/r}
\node[above=1pt] at (ol) {$x$};
\node[above=1pt] at (o) {$y$};
\node[right] at ($(o)!0.3!(ob)$) {$\omega$};
\node[left=1pt] at ($(o)!0.3!(ol)$) {$\beta$};
\node[right] at ($(o)!0.1!(or)$) {$\alpha$};
},
\treeoo{\cdb o\cdlr{o,or}
\node[above=1pt] at (or) {$y$};
\node[above=1pt] at (o) {$x$};
\node[right] at ($(o)!0.3!(or)$) {$\alpha$};
},
\treeoo{\cdb o\zhd{o/b}\cdlr{o,or}
\node[above=1pt] at (or) {$y$};
\node[above=1pt] at (o) {$x$};
\node[right] at ($(o)!0.3!(ob)$) {$\omega$};
\node[right] at ($(o)!0.3!(or)$) {$\alpha$};
},
\treeoo{\cdb o\cda o
\cdx[1.5]{o}{ol}{150}\cdx[1.5]{o}{or}{30}\zhd{o/r}
\node[left] at ($(o)!0.5!(oa)$) {$x$};
\node[right] at ($(o)!0.5!(oa)$) {$y$};
\node[right] at ($(o)!0.1!(or)$) {$\alpha$};
},
\treeoo{\cdb o\cda o\zhd{o/b}
\cdx{o}{ol}{150}\cdx{o}{or}{30}
\node[left] at ($(o)!0.5!(oa)$) {$x$};
\node[right] at ($(o)!0.5!(oa)$) {$y$};
\node[right] at ($(o)!0.1!(ob)$) {$\omega$};
},\,
\treeoo{\cdb o\cda o\zhd{o/b}
\cdx{o}{ol}{150}\cdx{o}{or}{30}
\node[left] at ($(o)!0.5!(oa)$) {$x$};
\node[right] at ($(o)!0.5!(oa)$) {$y$};
\node[right] at ($(o)!0.1!(ob)$) {$\omega$};
\cdb{ob}\zhd{ob/b}
\node[right] at ($(ob)!0.1!(obb)$) {$\alpha$};
},
\cdots
\big| x,y\in X, \alpha,\beta,\omega,\cdots\in\Omega\right\},\\
\caltd=&  \left\{
\XX[scale=1.8]{\xxr{-4}4\xxr{-7.5}{7.5}\zhongdian{0,0}{0.75,0.75}
\node at (-0.3,0) {$\alpha$};
\node at (-0.72,0.4) {$\beta$};
\node at (0.3,0.03) {$\omega$};
\xxhu00x \xxhu[0.1]{-4}4{\,y} \xxhu[0.1]{-7.5}{7.5}{z}
},
\XX[scale=1.6]{\xxr{-4}4\xx{-4}42
\node at (-0.3,0) {$\alpha$};
\xxhu00x \xxh{-4}423{y\ \,}{0.3} \xxh{-4}412{\ \, z}{0.3}
},
\XX[scale=1.6]{\xx00{1.6}\xx00{2.4}
\xxh001{1.6}{\ \ \,z}{0.6}
\xxh00{1.6}{2.4}{y}{0.5}
\xxh00{2.4}3{x\ \ }{0.6}
},
\XX[scale=1.6]{\xxr{-6}6\xxl66\zhongdian{0,0}{0,-1}
\draw(0,0)--(0,-1);
\node at (-0.5,0.15) {$\alpha$};
\node at (0.45,0.15) {$\beta$};
\node at (-0.25,-0.25) {$\omega$};
\xxhu00{x} \xxhu66{z} \xxhu{-6}6{y}
},
\ldots\Bigg|\,x,y,z\in X,\alpha,\beta,\omega,\cdots\in\Omega
\right\}.
\end{align*}
Graphically, an element $T\in \cxo T;$ is of the form:
\begin{equation*}
T = \treeoo{\cdb o\ocdx[2]{o}{a1}{160}{T_1}{left}
\ocdx[2]{o}{a2}{120}{T_2}{above}
\ocdx[2]{o}{a3}{60}{T_n}{above}
\ocdx[2]{o}{a4}{20}{T_{n+1}}{right}
\node at (140:\xch) {$x_1$};
\node at (90:\xch) {$\cdots$};
\node at (40:\xch) {$x_n$};
\node[below] at ($(o)!0.7!(a1)$) {$\alpha_1$};
\node[below] at ($(o)!0.7!(a4)$) {$\alpha_{n+1}$};
\node[right] at ($(o)!0.85!(a2)$) {$\alpha_2$};
\node[left] at ($(o)!0.85!(a3)$) {$\alpha_n\!$};
},\,\text{ with } n\ge 0, \text{ where }\, x_1,\cdots,x_n\in X\,\text{ and }\,\alpha_1,\cdots,\alpha_{n+1}\in\Omega^1,
\end{equation*}
i.e., $\alpha_j\in\Omega$ if $T_j\neq |;$ $\alpha_j=1$ if $T_j=|.$
For each $\omega\in\Omega,$ define a linear operator
$$B^{+}_\omega:\bfk\calt\ra\bfk\calt,$$
where $B^{+}_\omega(T)$ is the typed angularly decorated planar rooted tree obtained by adding a new root and
an new internal edge decorated by $\omega$ connecting the new root and the root of $T$.
For example,

\[B^+_\omega\Big(\treeoo{\cda[1.5] o\zhd{o/a}
\node[right] at ($(o)!0.7!(oa)$) {};
}\Big)=
\treeoo{\cda[1.0] o\zhd{o/a}\cdb[1.0]o\zhd{o/b}
\node[right] at ($(o)!0.7!(oa)$) {};
\node[right] at ($(o)!0.15!(ob)$) {$\omega$};
},\quad
B^{+}_\omega\Big(\stree x\Big)=\treeoo{\cdb[1.5] o\zhd{o/b}\cdlr o
\node[above=1pt] at (o) {$x$};
\node[left] at ($(o)!0.3!(ob)$) {$\omega$};
},\quad
B^{+}_\omega\Big(\XX[scale=1.6] {\xxr{-6}6\xxl66
\node at (-0.5,0.15) {$\alpha$};
\node at (0.45,0.15) {$\beta$};
\xxhu00{x} \xxhu66{z} \xxhu{-6}6{y}
}\Big)=\XX[scale=1.6]{\xxr{-6}6\xxl66\zhongdian{0,0}{0,-1}
\draw(0,0)--(0,-1);
\node at (-0.5,0.15) {$\alpha$};
\node at (0.45,0.15) {$\beta$};
\node at (-0.25,-0.25) {$\omega$};
\xxhu00{x} \xxhu66{z} \xxhu{-6}6{y}
}.
\]

\begin{remark}
If $\Omega$ is a singleton set, then typed angularly decorated planar rooted trees
are reduced to angularly decorated planar rooted trees studied by Guo~\cite{Guo09}, who used them to obtain angularly decorated planar rooted forests and to construct free Rota-Baxter algebras.
\end{remark}

The {\bf depth} $\dep{(T)}$ of a rooted tree $T$ is the maximal length of linear chains from the root to the leaves of the tree.
For example,
\[
\dep\big(\treeoo{\cda[1.5] o\zhd{o/a}
\node[right] at ($(o)!0.7!(oa)$) {};
}\Big) = \dep\Big(\stree x\Big) = 1\, \text{ and } \, \dep\Big(\treeoo{\cdb[1.5] o\zhd{o/b}\cda o\cdb o
\cdx{o}{ol}{150}\cdx{o}{or}{30}
\node[left] at ($(o)!0.5!(oa)$) {$x$};
\node[right] at ($(o)!0.5!(oa)$) {$y$};
\node[left] at ($(o)!0.3!(ob)$) {$\alpha$};
}\Big) = 2.
\]
For later use, we add the "zero-vertex tree" $\vert$ to the picture, and set $\dep(\vert)=0$. Note that the operators $B^+_\omega$ are not defined on $\vert$.
\begin{remark}
For any $T\in \calt\sqcup\{\vert\}$, if $\dep(T) = 0$, then $T = |$. Define the number of branches $\bra(T)$ of $T$ to be $0$ in this case.
Otherwise, $\dep(T)\geq 1$ and $T$ is of the form
$$T = \treeoo{\cdb o\ocdx[2]{o}{a1}{160}{T_1}{left}
\ocdx[2]{o}{a2}{120}{T_2}{above}
\ocdx[2]{o}{a3}{60}{T_n}{above}
\ocdx[2]{o}{a4}{20}{T_{n+1}}{right}
\node at (140:\xch) {$x_1$};
\node at (90:\xch) {$\cdots$};
\node at (40:\xch) {$x_n$};
\node[below] at ($(o)!0.7!(a1)$) {$\alpha_1$};
\node[below] at ($(o)!0.7!(a4)$) {$\alpha_{n+1}$};
\node[right] at ($(o)!0.85!(a2)$) {$\alpha_2$};
\node[left] at ($(o)!0.85!(a3)$) {$\alpha_n\!$};
}\,\text{ with }\, n\geq 0. $$
Here any branch $T_j\in \calt\sqcup\{\vert\}, j=1,\ldots, n+1$ is of depth at most one less than the depth of $T$, and equal to zero if and only if $T_j=\vert$. We define $\bra(T):=n+1$.
For example,
\begin{align*}
 \bra\Big(\treeoo{\cda[1.0] o\zhd{o/a}\cdb[1.0]o\zhd{o/b}
\node[right] at ($(o)!0.7!(oa)$) {};
\node[right] at ($(o)!0.15!(ob)$) {$\omega$};
}\Big)=1,\,\bra\Big(\stree x\Big) = 2
 \, \text{ and }\,
\bra\Big(\treeoo{\cdb o\cda o
\cdx{o}{ol}{150}\cdx{o}{or}{30}
\node[left] at ($(o)!0.5!(oa)$) {$x$};
\node[right] at ($(o)!0.5!(oa)$) {$y$};
}\Big) = 3.
\end{align*}
\mlabel{re:2form}
\end{remark}

\subsection{A multiplication on $\bfk\calt$}
This subsection is devoted to define a unital multiplication $\diamond$ on $\bfk\calt$,
which is the crucial ingredient to construct a free Rota-Baxter family algebra.\\

Let $X$ be a set and let $\Omega$ be a semigroup. Let $T,T'\in\calt$.
We define $T\diamond T'$ by induction on $\dep(T)+\dep(T')\geq 2$.
For the initial step $\dep(T)+\dep(T')=2$, we have $\dep(T)=\dep(T')=1$ and $T, T'$ are of the form
$$T =\treeoo{\cdb o\cdx[1.5]{o}{a1}{160}
\cdx[1.5]{o}{a2}{120}
\cdx[1.5]{o}{a3}{60}
\cdx[1.5]{o}{a4}{20}
\node at (140:1.2*\xch) {$x_1$};
\node at (90:1.2*\xch) {$\cdots$};
\node at (40:1.2*\xch) {$x_m$};
}\,\text{ and }\,T'=\treeoo{\cdb o\cdx[1.5]{o}{a1}{160}
\cdx[1.5]{o}{a2}{120}
\cdx[1.5]{o}{a3}{60}
\cdx[1.5]{o}{a4}{20}
\node at (140:1.2*\xch) {$y_1$};
\node at (90:1.2*\xch) {$\cdots$};
\node at (40:1.2*\xch) {$y_n$};
},\,\text{ with }\,m,n\geq 0.$$
Define
\begin{equation}
T\diamond T':=\treeoo{\cdb o\cdx[1.5]{o}{a1}{160}
\cdx[1.5]{o}{a2}{120}
\cdx[1.5]{o}{a3}{60}
\cdx[1.5]{o}{a4}{20}
\node at (140:1.2*\xch) {$x_1$};
\node at (90:1.2*\xch) {$\cdots$};
\node at (40:1.2*\xch) {$x_m$};
}\diamond\treeoo{\cdb o\cdx[1.5]{o}{a1}{160}
\cdx[1.5]{o}{a2}{120}
\cdx[1.5]{o}{a3}{60}
\cdx[1.5]{o}{a4}{20}
\node at (140:1.2*\xch) {$y_1$};
\node at (90:1.2*\xch) {$\cdots$};
\node at (40:1.2*\xch) {$y_n$};
}:=\treeoo{\cdb o
\cdx[2.2]{o}{a1}{170}
\cdx[2.2]{o}{a2}{145}
\cdx[2.2]{o}{a3}{115}
\cdx[2.2]{o}{a4}{90}
\cdx[2.2]{o}{a5}{65}
\cdx[2.2]{o}{a6}{35}
\cdx[2.2]{o}{a7}{10}
\node at ($(a1)!0.5!(a2)$) {$x_1$};
\node[rotate=40] at ($(a2)!0.5!(a3)$) {$\cdots$};
\node at ($(a3)!0.5!(a4)$) {$x_m$};
\node at ($(a4)!0.5!(a5)$) {$y_1$};
\node[rotate=-40] at ($(a5)!0.5!(a6)$) {$\cdots$};
\node at ($(a6)!0.5!(a7)$) {$y_n$};
}.
\mlabel{eq:def1}
\end{equation}
For the induction step $\dep(T)+\dep(T')\geq 3$,  the trees $T$ and $T'$ are of the form
$$T=\treeoo{\cdb o\ocdx[2]{o}{a1}{160}{T_1}{left}
\ocdx[2]{o}{a2}{120}{T_2}{above}
\ocdx[2]{o}{a3}{60}{T_m}{above}
\ocdx[2]{o}{a4}{20}{T_{m+1}}{right}
\node at (140:\xch) {$x_1$};
\node at (90:\xch) {$\cdots$};
\node at (40:\xch) {$x_m$};
\node[below] at ($(o)!0.7!(a1)$) {$\alpha_1$};
\node[below] at ($(o)!0.7!(a4)$) {$\alpha_{m+1}$};
\node[right] at ($(o)!0.85!(a2)$) {$\alpha_2$};
\node[left] at ($(o)!0.85!(a3)$) {$\alpha_m\!$};
}\,\text{ and }\, T'=\treeoo{\cdb o\ocdx[2]{o}{a1}{160}{T'_1}{left}
\ocdx[2]{o}{a2}{120}{T'_2}{above}
\ocdx[2]{o}{a3}{60}{T'_n}{above}
\ocdx[2]{o}{a4}{20}{T'_{n+1}}{right}
\node at (140:\xch) {$y_1$};
\node at (90:\xch) {$\cdots$};
\node at (40:\xch) {$y_n$};
\node[below] at ($(o)!0.7!(a1)$) {$\beta_1$};
\node[below] at ($(o)!0.7!(a4)$) {$\beta_{n+1}$};
\node[right] at ($(o)!0.85!(a2)$) {$\beta_2$};
\node[left] at ($(o)!0.85!(a3)$) {$\beta_n\!$};
}\,\text{ with some }\, T_i\neq |\,\text{ or some}\, T'_j\neq |.$$ There are four cases to consider.

\noindent{\bf Case 1:} $T_{m+1}=|=T'_1$. Define

\begin{equation}
 T\diamond T':= \treeoo{\cdb o\ocdx[2]{o}{a1}{160}{T_1}{left}
\ocdx[2]{o}{a2}{120}{T_2}{above}
\ocdx[2]{o}{a3}{60}{T_m}{above}
\cdx[2]{o}{a4}{20}
\node at (140:\xch) {$x_1$};
\node at (90:\xch) {$\cdots$};
\node at (40:\xch) {$x_m$};
\node[below] at ($(o)!0.7!(a1)$) {$\alpha_1$};
\node[right] at ($(o)!0.85!(a2)$) {$\alpha_2$};
\node[left] at ($(o)!0.85!(a3)$) {$\alpha_m\!$};
}\diamond \treeoo{\cdb o
\cdx[2]{o}{a1}{160}
\ocdx[2]{o}{a2}{120}{T'_2}{above}
\ocdx[2]{o}{a3}{60}{T'_n}{above}
\ocdx[2]{o}{a4}{20}{T'_{n+1}}{right}
\node at (140:\xch) {$y_1$};
\node at (90:\xch) {$\cdots$};
\node at (40:\xch) {$y_n$};
\node[below] at ($(o)!0.7!(a4)$) {$\beta_{n+1}$};
\node[right] at ($(o)!0.85!(a2)$) {$\beta_2$};
\node[left] at ($(o)!0.85!(a3)$) {$\beta_n\!$};
}
:=\treeoo{\cdb o
\ocdx[2.5]{o}{a1}{170}{T_1}{170}
\ocdx[2.5]{o}{a2}{145}{T_2}{145}
\ocdx[2.5]{o}{a3}{115}{T_m}{115}
\cdx[2.5]{o}{a4}{90}{-90}
\ocdx[2.5]{o}{a5}{65}{T'_2}{65}
\ocdx[2.5]{o}{a6}{35}{T'_n}{35}
\ocdx[2.5]{o}{a7}{10}{T'_{n+1}}{10}
\node at (157:1.3*\xch) {$x_1$};
\node[rotate=40] at (129:2.5*\xch) {$\cdots$};
\node at (102:1.5*\xch) {$x_m$};
\node at (75:1.5*\xch) {$y_1$};
\node[rotate=-40] at (49:2.5*\xch) {$\cdots$};
\node at (20:1.5*\xch) {$y_n$};
\node[below] at ($(o)!0.7!(a1)$) {$\alpha_1$};
\node[below] at ($(o)!0.8!(a7)$) {\tiny$\beta_{n+1}$};
\node[below] at ($(o)!0.9!(a2)$) {\tiny$\alpha_2$};
\node[left=-1pt] at ($(o)!0.65!(a3)$) {\tiny$\alpha_m$};
\node[right=1pt] at ($(o)!0.65!(a5)$) {\tiny$\beta_2$};
\node[right=1pt] at ($(o)!0.65!(a6)$) {\tiny$\beta_n$};
}.
\mlabel{eq:def3}
\end{equation}

\noindent{\bf Case 2:} $T_{m+1}\neq|=T'_1$. Define
\begin{equation}
T\diamond T':=\treeoo{\cdb o\ocdx[2]{o}{a1}{160}{T_1}{left}
\ocdx[2]{o}{a2}{120}{T_2}{above}
\ocdx[2]{o}{a3}{60}{T_m}{above}
\ocdx[2]{o}{a4}{20}{T_{m+1}}{right}
\node at (140:\xch) {$x_1$};
\node at (90:\xch) {$\cdots$};
\node at (40:\xch) {$x_m$};
\node[below] at ($(o)!0.7!(a1)$) {$\alpha_1$};
\node[below] at ($(o)!0.7!(a4)$) {$\alpha_{m+1}$};
\node[right] at ($(o)!0.85!(a2)$) {$\alpha_2$};
\node[left] at ($(o)!0.85!(a3)$) {$\alpha_m\!$};
}\diamond \treeoo{\cdb o
\cdx[2]{o}{a1}{160}
\ocdx[2]{o}{a2}{120}{T'_2}{above}
\ocdx[2]{o}{a3}{60}{T'_n}{above}
\ocdx[2]{o}{a4}{20}{T'_{n+1}}{right}
\node at (140:\xch) {$y_1$};
\node at (90:\xch) {$\cdots$};
\node at (40:\xch) {$y_n$};
\node[below] at ($(o)!0.7!(a4)$) {$\beta_{n+1}$};
\node[right] at ($(o)!0.85!(a2)$) {$\beta_2$};
\node[left] at ($(o)!0.85!(a3)$) {$\beta_n\!$};
}:=\treeoo{\cdb o
\ocdx[2.5]{o}{a1}{170}{T_1}{170}
\ocdx[2.5]{o}{a2}{145}{T_2}{145}
\ocdx[2.5]{o}{a3}{115}{T_m}{115}
\ocdx[2.5]{o}{a4}{90}{T_{m+1}}{90}
\ocdx[2.5]{o}{a5}{65}{T'_2}{65}
\ocdx[2.5]{o}{a6}{35}{T'_n}{35}
\ocdx[2.5]{o}{a7}{10}{T'_{n+1}}{10}
\node at (157:1.3*\xch) {$x_1$};
\node[rotate=40] at (129:2.5*\xch) {$\cdots$};
\node at (102:1.5*\xch) {$x_m$};
\node at (75:1.5*\xch) {$y_1$};
\node[rotate=-40] at (49:2.5*\xch) {$\cdots$};
\node at (20:1.5*\xch) {$y_n$};
\node[below] at ($(o)!0.7!(a1)$) {$\alpha_1$};
\node[below] at ($(o)!0.8!(a7)$) {\tiny$\beta_{n+1}$};
\node[below] at ($(o)!0.9!(a2)$) {\tiny$\alpha_2$};
\node[left=-1pt] at ($(o)!0.65!(a3)$) {\tiny$\alpha_m$};
\node[right=1pt] at ($(o)!0.65!(a5)$) {\tiny$\beta_2$};
\node[right=1pt] at ($(o)!0.65!(a6)$) {\tiny$\beta_n$};
\node[right] at ($(o)!0.8!(a4)$) {\tiny$\alpha_{m+1}$};
}.
\mlabel{eq:def4}
\end{equation}

\noindent{\bf Case 3:} $T_{m+1}=|\neq T'_1$. Define
\begin{equation}
 T\diamond T':=\treeoo{\cdb o\ocdx[2]{o}{a1}{160}{T_1}{left}
\ocdx[2]{o}{a2}{120}{T_2}{above}
\ocdx[2]{o}{a3}{60}{T_m}{above}
\cdx[2]{o}{a4}{20}
\node at (140:\xch) {$x_1$};
\node at (90:\xch) {$\cdots$};
\node at (40:\xch) {$x_m$};
\node[below] at ($(o)!0.7!(a1)$) {$\alpha_1$};
\node[right] at ($(o)!0.85!(a2)$) {$\alpha_2$};
\node[left] at ($(o)!0.85!(a3)$) {$\alpha_m\!$};
}\diamond \treeoo{\cdb o\ocdx[2]{o}{a1}{160}{T'_1}{left}
\ocdx[2]{o}{a2}{120}{T'_2}{above}
\ocdx[2]{o}{a3}{60}{T'_n}{above}
\ocdx[2]{o}{a4}{20}{T'_{n+1}}{right}
\node at (140:\xch) {$y_1$};
\node at (90:\xch) {$\cdots$};
\node at (40:\xch) {$y_n$};
\node[below] at ($(o)!0.7!(a1)$) {$\beta_1$};
\node[below] at ($(o)!0.7!(a4)$) {$\beta_{n+1}$};
\node[right] at ($(o)!0.85!(a2)$) {$\beta_2$};
\node[left] at ($(o)!0.85!(a3)$) {$\beta_n\!$};
}:=\treeoo{\cdb o
\ocdx[2.5]{o}{a1}{170}{T_1}{170}
\ocdx[2.5]{o}{a2}{145}{T_2}{145}
\ocdx[2.5]{o}{a3}{115}{T_m}{115}
\ocdx[2.5]{o}{a4}{90}{T'_1}{90}
\ocdx[2.5]{o}{a5}{65}{T'_2}{65}
\ocdx[2.5]{o}{a6}{35}{T'_n}{35}
\ocdx[2.5]{o}{a7}{10}{T'_{n+1}}{10}
\node at (157:1.3*\xch) {$x_1$};
\node[rotate=40] at (129:2.5*\xch) {$\cdots$};
\node at (102:1.5*\xch) {$x_m$};
\node at (75:1.5*\xch) {$y_1$};
\node[rotate=-40] at (49:2.5*\xch) {$\cdots$};
\node at (20:1.5*\xch) {$y_n$};
\node[below] at ($(o)!0.7!(a1)$) {$\alpha_1$};
\node[below] at ($(o)!0.8!(a7)$) {\tiny$\beta_{n+1}$};
\node[below] at ($(o)!0.9!(a2)$) {\tiny$\alpha_2$};
\node[left=-1pt] at ($(o)!0.65!(a3)$) {\tiny$\alpha_m$};
\node[right=1pt] at ($(o)!0.65!(a5)$) {\tiny$\beta_2$};
\node[right=1pt] at ($(o)!0.65!(a6)$) {\tiny$\beta_n$};
\node[right] at ($(o)!0.8!(a4)$) {\tiny$\beta_{1}$};
} .
\mlabel{eq:def5}
\end{equation}

\noindent{\bf Case 4:} $T_{m+1}\neq|\neq T'_1$. Define
{\small{\begin{align}
 T\diamond T'
:=& \ \treeoo{\cdb o\ocdx[2]{o}{a1}{160}{T_1}{left}
\ocdx[2]{o}{a2}{120}{T_2}{above}
\ocdx[2]{o}{a3}{60}{T_m}{above}
\ocdx[2]{o}{a4}{20}{T_{m+1}}{right}
\node at (140:\xch) {$x_1$};
\node at (90:\xch) {$\cdots$};
\node at (40:\xch) {$x_m$};
\node[below] at ($(o)!0.7!(a1)$) {$\alpha_1$};
\node[below] at ($(o)!0.7!(a4)$) {$\alpha_{m+1}$};
\node[right] at ($(o)!0.85!(a2)$) {$\alpha_2$};
\node[left] at ($(o)!0.85!(a3)$) {$\alpha_m\!$};
}\diamond \treeoo{\cdb o\ocdx[2]{o}{a1}{160}{T'_1}{left}
\ocdx[2]{o}{a2}{120}{T'_2}{above}
\ocdx[2]{o}{a3}{60}{T'_n}{above}
\ocdx[2]{o}{a4}{20}{T'_{n+1}}{right}
\node at (140:\xch) {$y_1$};
\node at (90:\xch) {$\cdots$};
\node at (40:\xch) {$y_n$};
\node[below] at ($(o)!0.7!(a1)$) {$\beta_1$};
\node[below] at ($(o)!0.7!(a4)$) {$\beta_{n+1}$};
\node[right] at ($(o)!0.85!(a2)$) {$\beta_2$};
\node[left] at ($(o)!0.85!(a3)$) {$\beta_n\!$};
}\nonumber\\
:&=\Biggl(\treeoo{\cdb o\ocdx[2]{o}{a1}{160}{T_1}{left}
\ocdx[2]{o}{a2}{120}{T_2}{above}
\ocdx[2]{o}{a3}{60}{T_m}{above}
\cdx[2]{o}{a4}{20}
\node at (140:\xch) {$x_1$};
\node at (90:\xch) {$\cdots$};
\node at (40:\xch) {$x_m$};
\node[below] at ($(o)!0.7!(a1)$) {$\alpha_1$};
\node[right] at ($(o)!0.85!(a2)$) {$\alpha_2$};
\node[left] at ($(o)!0.85!(a3)$) {$\alpha_m\!$};
}\diamond \Big(B^{+}_{\alpha_{m+1}}(T_{m+1})\diamond B^{+}_{\beta_1}(T'_1)\Big)\Biggl)\diamond
\treeoo{\cdb o\cdx[2]{o}{a1}{160}
\ocdx[2]{o}{a2}{120}{T'_2}{above}
\ocdx[2]{o}{a3}{60}{T'_n}{above}
\ocdx[2]{o}{a4}{20}{T'_{n+1}}{right}
\node at (140:\xch) {$y_1$};
\node at (90:\xch) {$\cdots$};
\node at (40:\xch) {$y_n$};
\node[below] at ($(o)!0.7!(a4)$) {$\beta_{n+1}$};
\node[right] at ($(o)!0.85!(a2)$) {$\beta_2$};
\node[left] at ($(o)!0.85!(a3)$) {$\beta_n\!$};
}\nonumber\\
:=& \ \Biggl(\treeoo{\cdb o\ocdx[2]{o}{a1}{160}{T_1}{left}
\ocdx[2]{o}{a2}{120}{T_2}{above}
\ocdx[2]{o}{a3}{60}{T_m}{above}
\cdx[2]{o}{a4}{20}
\node at (140:\xch) {$x_1$};
\node at (90:\xch) {$\cdots$};
\node at (40:\xch) {$x_m$};
\node[below] at ($(o)!0.7!(a1)$) {$\alpha_1$};
\node[right] at ($(o)!0.85!(a2)$) {$\alpha_2$};
\node[left] at ($(o)!0.85!(a3)$) {$\alpha_m\!$};
}\diamond B^{+}_{\alpha_{m+1}\beta_1}\Big(B^{+}_{\alpha_{m+1}}(T_{m+1})\diamond T'_1
+T_{m+1}\diamond B^{+}_{\beta_1}(T'_1)+\lambda T_{m+1}\diamond T'_1\Big)\Biggl)\diamond\treeoo{\cdb o
\cdx[2]{o}{a1}{160}
\ocdx[2]{o}{a2}{120}{T'_2}{above}
\ocdx[2]{o}{a3}{60}{T'_n}{above}
\ocdx[2]{o}{a4}{20}{T'_{n+1}}{right}
\node at (140:\xch) {$y_1$};
\node at (90:\xch) {$\cdots$};
\node at (40:\xch) {$y_n$};
\node[below] at ($(o)!0.7!(a4)$) {$\beta_{n+1}$};
\node[right] at ($(o)!0.85!(a2)$) {$\beta_2$};
\node[left] at ($(o)!0.85!(a3)$) {$\beta_n\!$};
}.
\mlabel{eq:def6}
\end{align}}}
\noindent Here the first $\diamond$ is defined by Case 3, the second, third and fourth $\diamond$ are defined by induction and the last $\diamond$ is defined by Case 2.
This completes the inductive definition of the multiplication $\diamond$ on $\calt$.
Extending by linearity, we can expand the $\diamond$ to $\bfk\calt$. Let us expose an example for better understanding.

\begin{exam}
Let $X$ be a set and let $\Omega$ be a semigroup. For $x,y\in X$ and $\alpha,\beta\in\Omega,$ we have
\begin{align*}
\stree x\diamond \stree y=& \ \treeoo{\cdb o\cda o
\cdx{o}{ol}{150}\cdx{o}{or}{30}
\node[left] at ($(o)!0.5!(oa)$) {$x$};
\node[right] at ($(o)!0.5!(oa)$) {$y$};
}\quad(\text{by Eq.~(\mref{eq:def1})}),\\
\treeoo{\cdb[1.5] o\zhd{o/b}\cdlr o
\node[above=1pt] at (o) {$x$};
\node[left] at ($(o)!0.3!(ob)$) {$\alpha$};
}\diamond \stree y=& \ \treeoo{\cdb o\cdlr{o,ol}
\node[above=1pt] at (ol) {$x$};
\node[above=1pt] at (o) {$y$};
\node[left=1pt] at ($(o)!0.3!(ol)$) {$\alpha$};
},\quad
\stree x\diamond \treeoo{\cdb[1.5] o\zhd{o/b}\cdlr o
\node[above=1pt] at (o) {$y$};
\node[left] at ($(o)!0.3!(ob)$) {$\alpha$};
}= \treeoo{\cdb o\cdlr{o,or}
\node[above=1pt] at (or) {$y$};
\node[above=1pt] at (o) {$x$};
\node[right] at ($(o)!0.3!(or)$) {$\alpha$};
}\quad(\text{by Eq.~(\mref{eq:def4}) and Eq.~(\mref{eq:def5})}),\\
\treeoo{\cdb[1.5] o\zhd{o/b}\cdlr o
\node[above=1pt] at (o) {$x$};
\node[left] at ($(o)!0.3!(ob)$) {$\alpha$};
}\diamond \treeoo{\cdb[1.5] o\zhd{o/b}\cdlr o
\node[above=1pt] at (o) {$y$};
\node[left] at ($(o)!0.3!(ob)$) {$\beta$};
}=& \ B^{+}_\alpha\Big(\stree x\Big)\diamond B^{+}_\beta\Big(\stree y\Big)\\
=& \ B^{+}_{\alpha\beta}\Big(B^{+}_\alpha(\stree x)\diamond \stree y+\stree x\diamond B^{+}_\beta\Big(\stree y\Big)+\lambda \stree x\diamond \stree y\Big)\\
& \ \hspace{3cm}\text{(by Eq.~(\mref{eq:def6}))}\\
=& \ B^{+}_{\alpha\beta}\Big(\treeoo{\cdb[1.5] o\zhd{o/b}\cdlr o
\node[above=1pt] at (o) {$x$};
\node[left] at ($(o)!0.3!(ob)$) {$\alpha$};
}\diamond \stree y+\stree x\diamond \treeoo{\cdb[1.5] o\zhd{o/b}\cdlr o
\node[above=1pt] at (o) {$y$};
\node[left] at ($(o)!0.3!(ob)$) {$\beta$};
}+\lambda \stree x\diamond \stree y\Big)\\
=& \ B^{+}_{\alpha\beta}\Big(\treeoo{\cdb o\cdlr{o,ol}
\node[above=1pt] at (ol) {$x$};
\node[above=1pt] at (o) {$y$};
\node[left=1pt] at ($(o)!0.3!(ol)$) {$\alpha$};
}+\treeoo{\cdb o\cdlr{o,or}
\node[above=1pt] at (or) {$y$};
\node[above=1pt] at (o) {$x$};
\node[right] at ($(o)!0.3!(or)$) {$\beta$};
}+\lambda \treeoo{\cdb o\cda o
\cdx{o}{ol}{150}\cdx{o}{or}{30}
\node[left] at ($(o)!0.5!(oa)$) {$x$};
\node[right] at ($(o)!0.5!(oa)$) {$y$};
}\Big)\quad(\text{by Eqs.~(\mref{eq:def4}),~(\mref{eq:def5}) and ~(\mref{eq:def1})})\\
=& \ \treeoo{\cdb[1.5] o\zhd{o/b}\cdlr{o,ol}
\node[above=1pt] at (ol) {$x$};
\node[above=1pt] at (o) {$y$};
\node[left=1pt] at ($(o)!0.3!(ol)$) {$\alpha$};
\node[left] at ($(o)!0.3!(ob)$) {$\alpha\beta$};
}
+\treeoo{\cdb[1.5] o\zhd{o/b}\cdlr{o,or}
\node[above=1pt] at (or) {$y$};
\node[above=1pt] at (o) {$x$};
\node[right] at ($(o)!0.3!(or)$) {$\beta$};
\node[left] at ($(o)!0.3!(ob)$) {$\alpha\beta$};
}+\lambda \,\treeoo{\cdb[1.5] o\zhd{o/b}\cda o
\cdx{o}{ol}{150}\cdx{o}{or}{30}
\node[left] at ($(o)!0.5!(oa)$) {$x$};
\node[right] at ($(o)!0.5!(oa)$) {$y$};
\node[left] at ($(o)!0.3!(ob)$) {$\alpha\beta$};
}.
\end{align*}
\end{exam}

\begin{remark}
The unit for the product $\diamond$ is given by the tree $\treeoo{\cda[1.5] o\zhd{o/a}
\node[right] at ($(o)!0.7!(oa)$) {};
}$.
\end{remark}
\subsection{Construction of free RBFA}
In this subsection, we prove that $\bfk\calt$ is the free Rota-Baxter family algebra of weight $\lambda$ generated by $X$,
via a bijection between $\calt$ and the set of Rota-Baxter family bracketed words.
The latter is a subset of bracketed words~\mcite{BoCh,Guo09} and is used to construct free Rota-Baxter family algebras~\cite{ZG}.\\

Denote by $M(X)$ the free monoid generated by $X$. Let $\Omega$ be a set.
For any set $Y$ and $\omega\in \Omega$, let $\lfloor Y \rfloor_{\omega}$
denote the set $\big\{\lfloor y\rfloor_{\omega}\mid y\in Y\big\}$.
So $\lfloor Y \rfloor_{\omega}$ is a disjoint copy of $Y$.
Assume the sets $\lfloor Y \rfloor_{\omega}$ to be disjoint with each other when $\omega$ varies in $\Omega$. We use induction to define a direct system $\big\{\frak{M_{n}}:=\frak{M_{n}}(\Omega, X), i_{n,\, n+1}: \frak{M_{n}} \rightarrow \frak{M}_{n+1}\big\}_{n\geq 1}$ of free monoids. First, define
\begin{equation*}
\frak{M}_{1}:=M(X)\,\text{ and }\, \mathfrak{M}_{2}:= M\bigl(X\sqcup (\sqcup_{\omega\, \in \Omega}\lfloor \frak{M}_{1} \rfloor_{\omega})\bigr),
\mlabel{eq:inclu1}
\end{equation*}
with $i_{1,\,2}$ being the inclusion
\begin{equation*}
i_{1,\,2}: \frak{M}_{1}=M(X)\hookrightarrow  \frak{M}_{2}=M\bigl(X\sqcup (\sqcup_{\omega\,\in \Omega}\lfloor \frak{M}_{1}\rfloor_{\omega})\bigr).
\end{equation*}
Inductively assume that $\frak{M}_{n-1}$ has been defined for $n \geq 3,$ with the inclusion
\begin{equation}
i_{n-2,\,n-1}: \frak{M}_{n-2}\rightarrow \frak{M}_{n-1}.
\mlabel{eq:inclu2}
\end{equation}
Then, define
\begin{equation*}
\frak{M}_{n}:=M\bigl(X\sqcup(\sqcup_{\omega\,\in \Omega}\lfloor \frak{M}_{n-1}\rfloor_{\omega})\bigr).
\end{equation*}
The inclusion in Eq.~(\mref{eq:inclu2}) induces the inclusion
$$\lfloor \frak{M}_{n-2}\rfloor_{\omega}\rightarrow \lfloor \frak{M}_{n-1}\rfloor_{\omega},\,\text{ for each }\,\omega\in \Omega,$$
and generates an inclusion  of free monoids
$$i_{n-1,\,n}: \frak{M}_{n-1}=M\bigl(X\sqcup(\sqcup_{\omega \,\in \Omega}\,\lfloor \frak{M}_{n-2}\rfloor_{\omega})\bigr)
\hookrightarrow M\bigl(X\sqcup(\sqcup_{\omega \,\in \Omega}\,\lfloor \frak{M}_{n-1}\rfloor_{\omega})\bigr)=\frak{M}_{n}.$$
This completes the inductive
construction of the direct systems. Finally, define the direct limit of monoids
$$ \frak{M}(\Omega,\, X) :=\lim_{\longrightarrow}\frak{M}_{n}=\bigcup_{n\geq 0}\frak{M}_{n}$$
with identity $1$. An element in $ \frak{M}(\Omega,\, X)$ is called a {\bf bracketed word}
and is said to be of {\bf depth} $n$ if it is in $\frak{M}_{n}\backslash \frak{M}_{n-1}$.
For each $u\in \mathfrak{M}(\Omega,\,X)$ with $u\neq 1$, we may write $u$ as a product $u_1\cdots u_k$ uniquely for some $k$ with $u_i\in X\sqcup (\sqcup_{\omega\in \Omega}\lf \mathfrak{M}(\Omega,\,X)\rf_\omega)$ for $1\leq i\leq k.$ We call $k$ the {\bf breadth} of $u$ and denote it by $|u|.$ If $u=1,$ we define $|u|:=0.$
For each $\omega\in\Omega$, we introduce the linear operator
$$P_\omega:\fram\ra\fram, \quad u \mapsto \lf u\rf_\omega.$$
Let $Y,Z$ be the subsets of $\fram$. First define  the {\bf alternating products} of $Y$ and $Z$ by
\begin{align*}
\Lambda(Y,Z):=& \ \Big(\bigcup_{\substack{r\geq 1,\\ \omega\in\Omega}}(Y\sqcup_{\omega\in\Omega}\lf Z\rf_\omega)^{r}\Big)\,\bigcup\,\Big(\bigcup_{\substack{r\geq 0,\\ \omega\in\Omega}}(Y\sqcup_{\omega\in\Omega}\lf Z\rf_\omega)^{r}Y\Big)\,\bigcup\,\Big(\bigcup_{\substack{r\geq 1, \\ \omega\in\Omega}}(\sqcup_{\omega\in\Omega}\lf Z\rf_\omega Y)^{r}\Big)\\
& \ \bigcup\,\Big(\bigcup_{\substack{r\geq 0,\\ \omega\in\Omega}}(\sqcup_{\omega\in\Omega}\lf Z\rf_\omega Y)^{r}\sqcup_{\omega\in\Omega}\lf Z\rf_\omega\Big)\,\bigcup\,\{1\}
\end{align*}
where $1$ is the identity in $\fram$. Obviously, $\Lambda(Y,Z)\subseteq \fram$. Then define recursively
\begin{equation*}
\frak{X}_0:=M(X)\,\text{and}\,\frak{X}_n:=\Lambda\bigl(M(X),\frak{X}_{n-1}\bigr), n\geq 1.
\end{equation*}
Thus $\frak{X}_0\subseteq\cdots\subseteq\frak{X}_n\subseteq\cdots$ and we may define
$$\frak{X}_\infty:=\underset{\longrightarrow}\lim\frak{X}_n=\cup_{n\geq 0}\frak{X}_n.$$
Elements in $\frak{X}_\infty$ are called {\bf Rota-Baxter family bracketed words}. Denote by $S$ the following vector subspace of $\bfk\frakM(\Omega,\,X)$:
\begin{equation*}
S :=\left\{\lf x\rf_\alpha\lf y\rf_\beta-\lf\lf x\rf_\alpha y\rf_{\alpha\beta}-\lf x\lf y\rf_\beta\rf_{\alpha\beta}-\lambda\lf xy\rf_{\alpha\beta}\mid \alpha, \beta\in \Omega, x, y\in \frakM(\Omega,\,X)\right\}.
\end{equation*}

\begin{lemma}\mlabel{lem:GS}\cite{ZG}
Let $X$ be a set and let $\Omega$ be a semigroup.
\begin{enumerate}

\item $S$ is a Gr\"{o}bner-Shirshov basis in $\bfk\frak{M}(\Omega,\, X)$ with respect to a monomial order.\mlabel{it:gs}

\item The set $\frak{X}_\infty$ is a \bfk-basis of the free Rota-Baxter family algebra $(\bfk\frak{X}_\infty,\dw,(\lf\,\rf_\omega)_{\omega\in\Omega})=\bfk\mathfrak{M}(\Omega,\,X)/\Id(S)$ of weight $\lambda$,
where  $\Id(S)$ is the operated ideal generated by $S$ in $\bfk\mathfrak{M}(\Omega,\,X)$.\mlabel{it:ideal}
\end{enumerate}
\end{lemma}

We are going to erect a one-one correspondence between $\bfk\calt$ and $\bfk\frak{X}_\infty$.
Let us first expose some examples in the following table:
\begin{table}[!htb]
\centering
\begin{tabular}{|c|c|}
  \hline
  \text{Typed angularly decorated planar rooted trees $\calt$} & $\frak{X}_\infty$ \\
  \hline
  $\treeoo{\cda[1.5] o\zhd{o/a}
\node[right] at ($(o)!0.7!(oa)$) {};
}$ & $1$ \\
 \hline
\,$\treeoo{\cda[1.0] o\zhd{o/a}\cdb[1.0]o\zhd{o/b}
\node[right] at ($(o)!0.7!(oa)$) {};
\node[right] at ($(o)!0.2!(ob)$) {$\omega$};
}$ & $\lf 1\rf_\omega$\\
\hline
   $\stree x$ & $x$ \\
   \hline
 $\treeoo{\cdb o\zhd{o/b}\cdlr o
\node[above=1pt] at (o) {$x$};
\node[left] at ($(o)!0.3!(ob)$) {$\omega$};
}$ & $\lf x\rf_\omega$ \\
\hline
  $\treeoo{\cdb o\cda o
\cdx{o}{ol}{150}\cdx{o}{or}{30}
\node[left] at ($(o)!0.5!(oa)$) {$x$};
\node[right] at ($(o)!0.5!(oa)$) {$y$};
}$ & $xy$ \\
\hline
$\treeoo{\cdb o\zhd{o/b}\cda o\cdb o
\cdx{o}{ol}{150}\cdx{o}{or}{30}
\node[left] at ($(o)!0.5!(oa)$) {$x$};
\node[right] at ($(o)!0.5!(oa)$) {$y$};
\node[left] at ($(o)!0.3!(ob)$) {$\omega$};
}$ & $\lf xy\rf_\omega$\\
\hline
\end{tabular}
\label{1}
\end{table}

\noindent Define a linear map
$$\phi:\bfk\calt\ra \bfk\frak{X}_\infty, \quad T\mapsto \phi(T)$$
by induction on $\dep(T)\geq 1$.
If $\dep(T)=1$, then $T=\treeoo{\cdb o\cdx[1.5]{o}{a1}{160}
\cdx[1.5]{o}{a2}{120}
\cdx[1.5]{o}{a3}{60}
\cdx[1.5]{o}{a4}{20}
\node at (140:1.2*\xch) {$x_1$};
\node at (90:1.2*\xch) {$\cdots$};
\node at (40:1.2*\xch) {$x_n$};
}$ with $n\geq 0$ and define
\begin{equation}
\phi(T):=
\phi(\treeoo{\cdb o\cdx[1.5]{o}{a1}{160}
\cdx[1.5]{o}{a2}{120}
\cdx[1.5]{o}{a3}{60}
\cdx[1.5]{o}{a4}{20}
\node at (140:1.2*\xch) {$x_1$};
\node at (90:1.2*\xch) {$\cdots$};
\node at (40:1.2*\xch) {$x_n$};
}):=x_1\cdots x_n,
\mlabel{eq:df0}
\end{equation}
with $x_1\cdots x_n:=1$ in the case $n=0$. For the induction step  $\dep(T)\geq 2$, $T$ is the form $$\treeoo{\cdb o\ocdx[2]{o}{a1}{160}{T_1}{left}
\ocdx[2]{o}{a2}{120}{T_2}{above}
\ocdx[2]{o}{a3}{60}{T_n}{above}
\ocdx[2]{o}{a4}{20}{T_{n+1}}{right}
\node at (140:\xch) {$x_1$};
\node at (90:\xch) {$\cdots$};
\node at (40:\xch) {$x_n$};
\node[below] at ($(o)!0.7!(a1)$) {$\alpha_1$};
\node[below] at ($(o)!0.7!(a4)$) {$\alpha_{n+1}$};
\node[right] at ($(o)!0.85!(a2)$) {$\alpha_2$};
\node[left] at ($(o)!0.85!(a3)$) {$\alpha_n\!$};
}\,\text{ with some }\, T_i\neq |.$$
Then we define $\phi(T)$ by the induction on $\bra(T)\geq 1$. For the initial step  $\bra(T)=1$, we have $T= \treeoo{\zhd{o/a}
\node[right] at ($(o)!0.8!(oa)$) {$\alpha_1$};\ocdx[1.5]{o}{a1}{90}{T_1}{above}
}$ and define
\begin{equation}
\phi\Big(\treeoo{\zhd{o/a}
\node[right] at ($(o)!0.8!(oa)$) {$\alpha_1$};\ocdx[1.5]{o}{a1}{90}{T_1}{above}
}\Big):=\phi(B^{+}_{\alpha_1}(T_1)):=\lf \phi(T_1)\rf_{\alpha_1}.
\mlabel{eq:df1}
\end{equation}
For the induction step $\bra(T)\geq 2$, there are two cases to consider.

\noindent {\bf Case 1:} $T_1=|$. Define
\begin{equation}
\phi(T):=x_1\phi\Big(\treeoo{\cdb o\ocdx[2]{o}{a1}{160}{T_2}{left}
\ocdx[2]{o}{a2}{120}{T_3}{above}
\ocdx[2]{o}{a3}{60}{T_n}{above}
\ocdx[2]{o}{a4}{20}{T_{n+1}}{right}
\node at (140:\xch) {$x_2$};
\node at (90:\xch) {$\cdots$};
\node at (40:\xch) {$x_n$};
\node[below] at ($(o)!0.7!(a1)$) {$\alpha_2$};
\node[below] at ($(o)!0.7!(a4)$) {$\alpha_{n+1}$};
\node[right] at ($(o)!0.85!(a2)$) {$\alpha_3$};
\node[left] at ($(o)!0.85!(a3)$) {$\alpha_n\!$};
}\Big).
\mlabel{eq:df2}
\end{equation}
\noindent {\bf Case 2:} $T_1\neq |$. Define
 \begin{equation}
 \phi(T):=\lf \phi(T_1)\rf_{\alpha_1}x_1 \phi\Big(\treeoo{\cdb o\ocdx[2]{o}{a1}{160}{T_2}{left}
\ocdx[2]{o}{a2}{120}{T_3}{above}
\ocdx[2]{o}{a3}{60}{T_n}{above}
\ocdx[2]{o}{a4}{20}{T_{n+1}}{right}
\node at (140:\xch) {$x_2$};
\node at (90:\xch) {$\cdots$};
\node at (40:\xch) {$x_n$};
\node[below] at ($(o)!0.7!(a1)$) {$\alpha_2$};
\node[below] at ($(o)!0.7!(a4)$) {$\alpha_{n+1}$};
\node[right] at ($(o)!0.85!(a2)$) {$\alpha_3$};
\node[left] at ($(o)!0.85!(a3)$) {$\alpha_n\!$};
}\Big).
\mlabel{eq:df3}
\end{equation}

Conversely, we are going to define a linear map
$$\psi:\bfk\frak{X}_\infty\ra \bfk\calt, \quad w\mapsto \psi(w)$$
by induction on $\dep(w)\geq 1$. If $\dep(w) = 1$, then $w=x_1\cdots x_n\in M(X)$ with $n\geq 0$ and define
\begin{equation}
 \psi(x_1\cdots x_n):=\treeoo{\cdb o\cdx[1.5]{o}{a1}{160}
\cdx[1.5]{o}{a2}{120}
\cdx[1.5]{o}{a3}{60}
\cdx[1.5]{o}{a4}{20}
\node at (140:1.2*\xch) {$x_1$};
\node at (90:1.2*\xch) {$\cdots$};
\node at (40:1.2*\xch) {$x_n$};
}\,\text{ with the convention that }\, \psi(1):=\treeoo{\cda[1.5] o\zhd{o/a}
\node[right] at ($(o)!0.7!(oa)$) {};
}\,\text{ when }\, n=0.
\mlabel{eq:df4}
\end{equation}
If $\dep(w)\geq 1$, we apply induction on $\bre(w)\geq 1$.
Write $w=w_1\cdots w_n$ with $\bre(w) = n\geq 1$.

\noindent If $\bre(w)=1$, then $w=\lf\lbar{w}\rf_{\alpha}$ for $\lbar{w}\in\frak{X}_\infty$ and $\alpha\in\Omega$ by $\dep(w)\geq 1$, and define
\begin{equation}
\psi(w):=\psi(\lf\lbar{w}\rf_{\alpha}):=B^+_{\alpha}(\psi(\lbar{w})).
\mlabel{eq:df}
\end{equation}
If $\bre{(w)}\geq 2,$ then define
\begin{equation}
\psi(w):=\psi(w_1)\diamond \psi(w_2\cdots w_n).
\mlabel{eq:dff}
\end{equation}
This completes the definitions of $\phi$ and $\psi$.

\begin{prop}\mlabel{prop:iso}
We have $\psi\circ \phi = \id$ and $\phi\circ \psi =\id$.
\end{prop}

\begin{proof}
We first prove $\psi\circ \phi (T) = T$ for $T\in \calt$ by induction on $\dep(T)\geq 1$.
If $\dep(T)=1$, then $T=\treeoo{\cdb o\cdx[1.5]{o}{a1}{160}
\cdx[1.5]{o}{a2}{120}
\cdx[1.5]{o}{a3}{60}
\cdx[1.5]{o}{a4}{20}
\node at (140:1.2*\xch) {$x_1$};
\node at (90:1.2*\xch) {$\cdots$};
\node at (40:1.2*\xch) {$x_n$};
}$ with $n\geq 1$ and
$$\psi\circ\phi(T)=\psi(x_1\cdots x_n)=\treeoo{\cdb o\cdx[1.5]{o}{a1}{160}
\cdx[1.5]{o}{a2}{120}
\cdx[1.5]{o}{a3}{60}
\cdx[1.5]{o}{a4}{20}
\node at (140:1.2*\xch) {$x_1$};
\node at (90:1.2*\xch) {$\cdots$};
\node at (40:1.2*\xch) {$x_n$};
}=T\quad(\text{by Eq.~(\mref{eq:df4})}).$$
For the induction step $\dep(T)\geq 2,$
write
$$T = \treeoo{\cdb o\ocdx[2]{o}{a1}{160}{T_1}{left}\ocdx[2]{o}{a2}{120}{T_2}{above}
\ocdx[2]{o}{a3}{60}{T_n}{above}
\ocdx[2]{o}{a4}{20}{T_{n+1}}{right}
\node at (140:\xch) {$x_1$};
\node at (90:\xch) {$\cdots$};
\node at (40:\xch) {$x_n$};
\node[below] at ($(o)!0.7!(a1)$) {$\alpha_1$};
\node[below] at ($(o)!0.7!(a4)$) {$\alpha_{n+1}$};
\node[right] at ($(o)!0.85!(a2)$) {$\alpha_2$};
\node[left] at ($(o)!0.85!(a3)$) {$\alpha_n\!$};
}\,\text{ with some}\, T_i\neq |.
$$
We reduce to induction on $\bra(T)\geq 1$. For the initial step $\bra(T)=1$, we have $T=\treeoo{\zhd{o/a}
\node[right] at ($(o)!0.8!(oa)$) {$\alpha_1$};\ocdx[1.5]{o}{a1}{90}{T_1}{above}
}$ and so
\begin{align*}
\psi\circ\phi(T)&=\psi\circ\phi\Big(\treeoo{\zhd{o/a}
\node[right] at ($(o)!0.8!(oa)$) {$\alpha_1$};\ocdx[1.5]{o}{a1}{90}{T_1}{above}
}\Big)=\psi\circ\phi(B^{+}_{\alpha_1}(T_1))=\psi(\lf\phi(T_1)\rf_{\alpha_1})\quad(\text{by Eq.~(\mref{eq:df1})}
)\\
&=B^{+}_{\alpha_1}(\psi\circ\phi(T_1))\quad(\text{by Eq.~(\mref{eq:df})})\\
&= B^{+}_{\alpha_1}(T_1)\quad(\text{by the induction on $\dep(T)$})\\
&=\treeoo{\zhd{o/a}
\node[right] at ($(o)!0.8!(oa)$) {$\alpha_1$};\ocdx[1.5]{o}{a1}{90}{T_1}{above}
}=T.
\end{align*}

\noindent For the induction step $\bra(T)\geq 2$, there are two cases to consider.

\noindent {\bf Case 1:} $T_1=|$. Then
\begin{align*}
\psi\circ\phi(T)=& \ \psi\circ\phi\Big(\treeoo{\cdb o\cdx[2]{o}{a1}{160}
\ocdx[2]{o}{a2}{120}{T_2}{above}
\ocdx[2]{o}{a3}{60}{T_n}{above}
\ocdx[2]{o}{a4}{20}{T_{n+1}}{right}
\node at (140:\xch) {$x_1$};
\node at (90:\xch) {$\cdots$};
\node at (40:\xch) {$x_n$};
\node[below] at ($(o)!0.7!(a4)$) {$\alpha_{n+1}$};
\node[right] at ($(o)!0.85!(a2)$) {$\alpha_2$};
\node[left] at ($(o)!0.85!(a3)$) {$\alpha_n\!$};
}\Big)
= \psi\Big(x_1\phi\Big(\treeoo{\cdb o\ocdx[2]{o}{a1}{160}{T_2}{left}\ocdx[2]{o}{a2}{120}
{T_3}{above}
\ocdx[2]{o}{a3}{60}{T_n}{above}
\ocdx[2]{o}{a4}{20}{T_{n+1}}{right}
\node at (140:\xch) {$x_2$};
\node at (90:\xch) {$\cdots$};
\node at (40:\xch) {$x_n$};
\node[below] at ($(o)!0.7!(a1)$) {$\alpha_2$};
\node[below] at ($(o)!0.7!(a4)$) {$\alpha_{n+1}$};
\node[right] at ($(o)!0.85!(a2)$) {$\alpha_3$};
\node[left] at ($(o)!0.85!(a3)$) {$\alpha_n\!$};
}\Big)\Big)\quad(\text{by Eq.~(\mref{eq:df2})})\\
=& \ \psi(x_1)\diamond \psi\circ\phi\Big(\treeoo{\cdb o\ocdx[2]{o}{a1}{160}{T_2}{left}\ocdx[2]{o}{a2}{120}
{T_3}{above}
\ocdx[2]{o}{a3}{60}{T_n}{above}
\ocdx[2]{o}{a4}{20}{T_{n+1}}{right}
\node at (140:\xch) {$x_2$};
\node at (90:\xch) {$\cdots$};
\node at (40:\xch) {$x_n$};
\node[below] at ($(o)!0.7!(a1)$) {$\alpha_2$};
\node[below] at ($(o)!0.7!(a4)$) {$\alpha_{n+1}$};
\node[right] at ($(o)!0.85!(a2)$) {$\alpha_3$};
\node[left] at ($(o)!0.85!(a3)$) {$\alpha_n\!$};
}\Big)\quad(\text{by Eq.~(\mref{eq:dff})})\\
=& \ \stree {x_1}\diamond \psi\circ\phi\Big(\treeoo{\cdb o\ocdx[2]{o}{a1}{160}{T_2}{left}\ocdx[2]{o}{a2}{120}
{T_3}{above}
\ocdx[2]{o}{a3}{60}{T_n}{above}
\ocdx[2]{o}{a4}{20}{T_{n+1}}{right}
\node at (140:\xch) {$x_2$};
\node at (90:\xch) {$\cdots$};
\node at (40:\xch) {$x_n$};
\node[below] at ($(o)!0.7!(a1)$) {$\alpha_2$};
\node[below] at ($(o)!0.7!(a4)$) {$\alpha_{n+1}$};
\node[right] at ($(o)!0.85!(a2)$) {$\alpha_3$};
\node[left] at ($(o)!0.85!(a3)$) {$\alpha_n\!$};
}\Big)\quad(\text{by Eq.~(\mref{eq:df4})})\\
=& \ \stree {x_1}\diamond \treeoo{\cdb o\ocdx[2]{o}{a1}{160}{T_2}{left}\ocdx[2]{o}{a2}{120}
{T_3}{above}
\ocdx[2]{o}{a3}{60}{T_n}{above}
\ocdx[2]{o}{a4}{20}{T_{n+1}}{right}
\node at (140:\xch) {$x_2$};
\node at (90:\xch) {$\cdots$};
\node at (40:\xch) {$x_n$};
\node[below] at ($(o)!0.7!(a1)$) {$\alpha_2$};
\node[below] at ($(o)!0.7!(a4)$) {$\alpha_{n+1}$};
\node[right] at ($(o)!0.85!(a2)$) {$\alpha_3$};
\node[left] at ($(o)!0.85!(a3)$) {$\alpha_n\!$};
}\quad(\text{by the induction on $\bra(T)$})\\
=& \ \treeoo{\cdb o\cdx[2]{o}{a1}{160}
\ocdx[2]{o}{a2}{120}{T_2}{above}
\ocdx[2]{o}{a3}{60}{T_n}{above}
\ocdx[2]{o}{a4}{20}{T_{n+1}}{right}
\node at (140:\xch) {$x_1$};
\node at (90:\xch) {$\cdots$};
\node at (40:\xch) {$x_n$};
\node[below] at ($(o)!0.7!(a4)$) {$\alpha_{n+1}$};
\node[right] at ($(o)!0.85!(a2)$) {$\alpha_2$};
\node[left] at ($(o)!0.85!(a3)$) {$\alpha_n\!$};
}\quad(\text{by Eq.~(\mref{eq:def5})})\\
=& \ T.
\end{align*}
\noindent{\bf Case 2:} $T_1\neq |.$ Then
\begin{align*}
\psi\circ\phi(T)=& \ \psi\Biggl(\lf \phi(T_1)\rf_{\alpha_1} x_1\phi\Big(\treeoo{\cdb o\ocdx[2]{o}{a1}{160}{T_2}{left}\ocdx[2]{o}{a2}{120}{T_3}{above}
\ocdx[2]{o}{a3}{60}{T_n}{above}
\ocdx[2]{o}{a4}{20}{T_{n+1}}{right}
\node at (140:\xch) {$x_2$};
\node at (90:\xch) {$\cdots$};
\node at (40:\xch) {$x_n$};
\node[below] at ($(o)!0.7!(a1)$) {$\alpha_2$};
\node[below] at ($(o)!0.7!(a4)$) {$\alpha_{n+1}$};
\node[right] at ($(o)!0.85!(a2)$) {$\alpha_3$};
\node[left] at ($(o)!0.85!(a3)$) {$\alpha_n\!$};
}\Big)\Biggl)\quad(\text{by Eq.~(\mref{eq:df3})})\\
=& \ \psi(\lf\phi(T_1)\rf_{\alpha_1})\diamond \psi\Biggl(x_1\phi\Big(\treeoo{\cdb o\ocdx[2]{o}{a1}{160}{T_2}{left}\ocdx[2]{o}{a2}{120}{T_3}{above}
\ocdx[2]{o}{a3}{60}{T_n}{above}
\ocdx[2]{o}{a4}{20}{T_{n+1}}{right}
\node at (140:\xch) {$x_2$};
\node at (90:\xch) {$\cdots$};
\node at (40:\xch) {$x_n$};
\node[below] at ($(o)!0.7!(a1)$) {$\alpha_2$};
\node[below] at ($(o)!0.7!(a4)$) {$\alpha_{n+1}$};
\node[right] at ($(o)!0.85!(a2)$) {$\alpha_3$};
\node[left] at ($(o)!0.85!(a3)$) {$\alpha_n\!$};
}\Big)\Biggl)\quad(\text{by Eq.~(\mref{eq:dff})})\\
=& \ B^{+}_{\alpha_1}\Big(\psi\circ\phi(T_1)\Big)\diamond \psi\Biggl(x_1\phi\Big(\treeoo{\cdb o\ocdx[2]{o}{a1}{160}{T_2}{left}\ocdx[2]{o}{a2}{120}{T_3}{above}
\ocdx[2]{o}{a3}{60}{T_n}{above}
\ocdx[2]{o}{a4}{20}{T_{n+1}}{right}
\node at (140:\xch) {$x_2$};
\node at (90:\xch) {$\cdots$};
\node at (40:\xch) {$x_n$};
\node[below] at ($(o)!0.7!(a1)$) {$\alpha_2$};
\node[below] at ($(o)!0.7!(a4)$) {$\alpha_{n+1}$};
\node[right] at ($(o)!0.85!(a2)$) {$\alpha_3$};
\node[left] at ($(o)!0.85!(a3)$) {$\alpha_n\!$};
}\Big)\Biggl)\quad(\text{by Eq.~(\mref{eq:df})})\\
=& \ B^{+}_{\alpha_1}(T_1)\diamond \Biggl(\psi(x_1)\diamond\psi\circ\phi\Big(\treeoo{\cdb o\ocdx[2]{o}{a1}{160}{T_2}{left}\ocdx[2]{o}{a2}{120}{T_3}{above}
\ocdx[2]{o}{a3}{60}{T_n}{above}
\ocdx[2]{o}{a4}{20}{T_{n+1}}{right}
\node at (140:\xch) {$x_2$};
\node at (90:\xch) {$\cdots$};
\node at (40:\xch) {$x_n$};
\node[below] at ($(o)!0.7!(a1)$) {$\alpha_2$};
\node[below] at ($(o)!0.7!(a4)$) {$\alpha_{n+1}$};
\node[right] at ($(o)!0.85!(a2)$) {$\alpha_3$};
\node[left] at ($(o)!0.85!(a3)$) {$\alpha_n\!$};
}\Big)\Biggl)\\
&\quad(\text{by the induction on $\dep(T)$ and Eq.~(\mref{eq:dff})})\\
=& \ B^{+}_{\alpha_1}(T_1)\diamond \Big(\stree {x_1}\diamond \treeoo{\cdb o\ocdx[2]{o}{a1}{160}{T_2}{left}\ocdx[2]{o}{a2}{120}{T_3}{above}
\ocdx[2]{o}{a3}{60}{T_n}{above}
\ocdx[2]{o}{a4}{20}{T_{n+1}}{right}
\node at (140:\xch) {$x_2$};
\node at (90:\xch) {$\cdots$};
\node at (40:\xch) {$x_n$};
\node[below] at ($(o)!0.7!(a1)$) {$\alpha_2$};
\node[below] at ($(o)!0.7!(a4)$) {$\alpha_{n+1}$};
\node[right] at ($(o)!0.85!(a2)$) {$\alpha_3$};
\node[left] at ($(o)!0.85!(a3)$) {$\alpha_n\!$};
}\Big)\quad(\text{by Eq.~(\mref{eq:df4}) and induction on $\bra(T)$})\\
=& \ B^{+}_{\alpha_1}(T_1)\diamond \treeoo{\cdb o\cdx[2]{o}{a1}{160}
\ocdx[2]{o}{a2}{120}{T_2}{above}
\ocdx[2]{o}{a3}{60}{T_n}{above}
\ocdx[2]{o}{a4}{20}{T_{n+1}}{right}
\node at (140:\xch) {$x_1$};
\node at (90:\xch) {$\cdots$};
\node at (40:\xch) {$x_n$};
\node[below] at ($(o)!0.7!(a4)$) {$\alpha_{n+1}$};
\node[right] at ($(o)!0.85!(a2)$) {$\alpha_2$};
\node[left] at ($(o)!0.85!(a3)$) {$\alpha_n\!$};
}\quad(\text{by Eq.~(\mref{eq:def5})})\\
=& \ \treeoo{\cdb o\ocdx[2]{o}{a1}{160}{T_1}{left}\ocdx[2]{o}{a2}{120}{T_2}{above}
\ocdx[2]{o}{a3}{60}{T_n}{above}
\ocdx[2]{o}{a4}{20}{T_{n+1}}{right}
\node at (140:\xch) {$x_1$};
\node at (90:\xch) {$\cdots$};
\node at (40:\xch) {$x_n$};
\node[below] at ($(o)!0.7!(a1)$) {$\alpha_1$};
\node[below] at ($(o)!0.7!(a4)$) {$\alpha_{n+1}$};
\node[right] at ($(o)!0.85!(a2)$) {$\alpha_2$};
\node[left] at ($(o)!0.85!(a3)$) {$\alpha_n\!$};
}\quad(\text{by Eq.~(\mref{eq:def4})})\\
=& \ T.
\end{align*}
So we have $\psi\circ\phi=\id.$ Similarly, we can prove $\phi\circ\psi=\id.$ This completes the proof.
\end{proof}

The following result shows that the linear map $\phi$ is compatible with respect to the multiplications $\diamond$ and $\dw$, where $\dw$ is the associative multiplication on the free Rota-Baxter family algebra given in Lemma~\mref{lem:GS}.

\begin{lemma}\mlabel{lem:compa}
For $T$ and $T'$ in $\calt$, we have
\begin{equation}
\phi(T\diamond T') = \phi(T) \diamondw \phi(T').
\mlabel{eq:compa}
\end{equation}
\end{lemma}

\begin{proof}
If $T=\treeoo{\cda[1.5] o\zhd{o/a}
\node[right] at ($(o)!0.7!(oa)$) {};
}$ or $T'=\treeoo{\cda[1.5] o\zhd{o/a}
\node[right] at ($(o)!0.7!(oa)$) {};
}$, without loss of generality, letting $T'=\treeoo{\cda[1.5] o\zhd{o/a}
\node[right] at ($(o)!0.7!(oa)$) {};
}$, then $\phi(T')=\phi(\treeoo{\cda[1.5] o\zhd{o/a}
\node[right] at ($(o)!0.7!(oa)$) {};
})=1$ and
$$\phi(T\diamond T')=\phi(T)=\phi(T)\dw 1=\phi(T)\dw \phi(\treeoo{\cda[1.5] o\zhd{o/a}
\node[right] at ($(o)!0.7!(oa)$) {};
})=\phi(T)\dw\phi(T').$$
Suppose $T\neq \treeoo{\cda[1.5] o\zhd{o/a}
\node[right] at ($(o)!0.7!(oa)$) {};
}$ and $T'\neq \treeoo{\cda[1.5] o\zhd{o/a}
\node[right] at ($(o)!0.7!(oa)$) {};
}$. We prove Eq.~(\mref{eq:compa}) by induction on $\dep(T)+\dep(T')\geq 2$.
For the initial step $\dep(T)+\dep(T')=2,$ we have $\dep(T)=\dep(T')=1$ and
$$T=\treeoo{\cdb o\cdx[1.5]{o}{a1}{160}
\cdx[1.5]{o}{a2}{120}
\cdx[1.5]{o}{a3}{60}
\cdx[1.5]{o}{a4}{20}
\node at (140:1.2*\xch) {$x_1$};
\node at (90:1.2*\xch) {$\cdots$};
\node at (40:1.2*\xch) {$x_m$};
}\,\text{ and }\,T'=\treeoo{\cdb o\cdx[1.5]{o}{a1}{160}
\cdx[1.5]{o}{a2}{120}
\cdx[1.5]{o}{a3}{60}
\cdx[1.5]{o}{a4}{20}
\node at (140:1.2*\xch) {$y_1$};
\node at (90:1.2*\xch) {$\cdots$};
\node at (40:1.2*\xch) {$y_n$};
}\,\text{ with }\,m,n\geq 0.$$
We get
\begin{align*}
\phi(T\diamond T')=& \ \phi\Big(\treeoo{\cdb o
\cdx[2.2]{o}{a1}{170}
\cdx[2.2]{o}{a2}{145}
\cdx[2.2]{o}{a3}{115}
\cdx[2.2]{o}{a4}{90}
\cdx[2.2]{o}{a5}{65}
\cdx[2.2]{o}{a6}{35}
\cdx[2.2]{o}{a7}{10}
\node at ($(a1)!0.5!(a2)$) {$x_1$};
\node[rotate=40] at ($(a2)!0.5!(a3)$) {$\cdots$};
\node at ($(a3)!0.5!(a4)$) {$x_m$};
\node at ($(a4)!0.5!(a5)$) {$y_1$};
\node[rotate=-40] at ($(a5)!0.5!(a6)$) {$\cdots$};
\node at ($(a6)!0.5!(a7)$) {$y_n$};
}\Big)\quad(\text{by Eq.~(\mref{eq:def1})})\\
=& \ x_1\cdots x_my_1\cdots y_n\quad(\text{by Eq.~(\mref{eq:df0})})\\
=& \ (x_1\cdots x_m)\dw(y_1\cdots y_n)\quad(\text{by the definition of $\dw$)})\\
=& \ \phi(T)\dw\phi(T')\quad(\text{by Eq.~(\mref{eq:df0})}).
\end{align*}
For the induction step $\dep(T)+\dep(T')\geq 3,$ let $$T=\treeoo{\cdb o\ocdx[2]{o}{a1}{160}{T_1}{left}
\ocdx[2]{o}{a2}{120}{T_2}{above}
\ocdx[2]{o}{a3}{60}{T_m}{above}
\ocdx[2]{o}{a4}{20}{T_{m+1}}{right}
\node at (140:\xch) {$x_1$};
\node at (90:\xch) {$\cdots$};
\node at (40:\xch) {$x_m$};
\node[below] at ($(o)!0.7!(a1)$) {$\alpha_1$};
\node[below] at ($(o)!0.7!(a4)$) {$\alpha_{m+1}$};
\node[right] at ($(o)!0.85!(a2)$) {$\alpha_2$};
\node[left] at ($(o)!0.85!(a3)$) {$\alpha_m\!$};
}\,\text{and}\,T'=\treeoo{\cdb o\ocdx[2]{o}{a1}{160}{T'_1}{left}
\ocdx[2]{o}{a2}{120}{T'_2}{above}
\ocdx[2]{o}{a3}{60}{T'_n}{above}
\ocdx[2]{o}{a4}{20}{T'_{n+1}}{right}
\node at (140:\xch) {$y_1$};
\node at (90:\xch) {$\cdots$};
\node at (40:\xch) {$y_n$};
\node[below] at ($(o)!0.7!(a1)$) {$\beta_1$};
\node[below] at ($(o)!0.7!(a4)$) {$\beta_{n+1}$};
\node[right] at ($(o)!0.85!(a2)$) {$\beta_2$};
\node[left] at ($(o)!0.85!(a3)$) {$\beta_n\!$};
}\,\text{ with some } T_i\neq |\,\text{ or some }\, T'_i\neq |.$$
We reduce to induction on $\bra(T)+\bra(T')\geq 2$. For the initial step  $\bra(T)+\bra(T')=2$, we have $T=\treeoo{\zhd{o/a}
\node[right] at ($(o)!0.8!(oa)$) {$\alpha_1$};\ocdx[1.5]{o}{a1}{90}{T_1}{above}
}$ and $T'=\treeoo{\zhd{o/a}
\node[right] at ($(o)!0.8!(oa)$) {$\beta_1$};\ocdx[1.5]{o}{a1}{90}{T'_1}{above}
}$. Then
\begin{align*}
\phi(T\diamond T')&=\phi\Big(B^+_{\alpha_1}(T_1)\diamond B^+_{\beta_1}(T'_1)\Big)\\
&=\phi\Big(B^+_{\alpha_1\beta_1}\Big(B^+_{\alpha_1}(T_1)\diamond T'_1+T_1\diamond B^+_{\beta_1}(T'_1)+\lambda T_1\diamond T'_1\Big)\Big)\quad(\text{by Eq.~(\mref{eq:def6})})\\
&=\lf \lf\phi(T_1)\rf_{\alpha_1}\dw\phi(T'_1)
+\phi(T_1)\dw\lf\phi(T'_1)\rf_{\beta_1}+\lambda\phi(T_1)\dw\phi(T'_1)
\rf_{\alpha_1\beta_1}\\
&\hspace{1cm}(\text{by Eq.~(\mref{eq:df1}) and the induction on $\dep(T)+\dep(T')$})\\
&=\lf\phi(T_1)\rf_{\alpha_1}\dw\lf\phi(T'_1)\rf_{\beta_1}\quad(\text{by Lemma~\mref{lem:GS}~\mref{it:ideal}})\\
&=\phi(T)\dw\phi(T')\quad(\text{by Eq.~(\mref{eq:df1})}).
\end{align*}
For the induction step $\bra(T)+\bra(T')\geq 3,$ there are two cases to consider.

\noindent {\bf Case 1:} $T_1=|.$ Then
\begin{align*}
\phi(T\diamond T')&=x_1\phi\Big(\treeoo{\cdb o\ocdx[2]{o}{a1}{160}{T_2}{left}
\ocdx[2]{o}{a2}{120}{T_3}{above}
\ocdx[2]{o}{a3}{60}{T_m}{above}
\ocdx[2]{o}{a4}{20}{T_{m+1}}{right}
\node at (140:\xch) {$x_2$};
\node at (90:\xch) {$\cdots$};
\node at (40:\xch) {$x_m$};
\node[below] at ($(o)!0.7!(a1)$) {$\alpha_2$};
\node[below] at ($(o)!0.7!(a4)$) {$\alpha_{m+1}$};
\node[right] at ($(o)!0.85!(a2)$) {$\alpha_3$};
\node[left] at ($(o)!0.85!(a3)$) {$\alpha_m\!$};
}\diamond \treeoo{\cdb o\ocdx[2]{o}{a1}{160}{T'_1}{left}
\ocdx[2]{o}{a2}{120}{T'_2}{above}
\ocdx[2]{o}{a3}{60}{T'_n}{above}
\ocdx[2]{o}{a4}{20}{T'_{n+1}}{right}
\node at (140:\xch) {$y_1$};
\node at (90:\xch) {$\cdots$};
\node at (40:\xch) {$y_n$};
\node[below] at ($(o)!0.7!(a1)$) {$\beta_1$};
\node[below] at ($(o)!0.7!(a4)$) {$\beta_{n+1}$};
\node[right] at ($(o)!0.85!(a2)$) {$\beta_2$};
\node[left] at ($(o)!0.85!(a3)$) {$\beta_n\!$};
}\Big)\quad(\text{by Eq.~(\mref{eq:df2})})\\
&=x_1\phi\Big(\treeoo{\cdb o\ocdx[2]{o}{a1}{160}{T_2}{left}
\ocdx[2]{o}{a2}{120}{T_3}{above}
\ocdx[2]{o}{a3}{60}{T_m}{above}
\ocdx[2]{o}{a4}{20}{T_{m+1}}{right}
\node at (140:\xch) {$x_2$};
\node at (90:\xch) {$\cdots$};
\node at (40:\xch) {$x_m$};
\node[below] at ($(o)!0.7!(a1)$) {$\alpha_2$};
\node[below] at ($(o)!0.7!(a4)$) {$\alpha_{m+1}$};
\node[right] at ($(o)!0.85!(a2)$) {$\alpha_3$};
\node[left] at ($(o)!0.85!(a3)$) {$\alpha_m\!$};
}\Big)\dw\phi\Big(\treeoo{\cdb o\ocdx[2]{o}{a1}{160}{T'_1}{left}
\ocdx[2]{o}{a2}{120}{T'_2}{above}
\ocdx[2]{o}{a3}{60}{T'_n}{above}
\ocdx[2]{o}{a4}{20}{T'_{n+1}}{right}
\node at (140:\xch) {$y_1$};
\node at (90:\xch) {$\cdots$};
\node at (40:\xch) {$y_n$};
\node[below] at ($(o)!0.7!(a1)$) {$\beta_1$};
\node[below] at ($(o)!0.7!(a4)$) {$\beta_{n+1}$};
\node[right] at ($(o)!0.85!(a2)$) {$\beta_2$};
\node[left] at ($(o)!0.85!(a3)$) {$\beta_n\!$};
}\Big)\\
&\hspace{2cm}(\text{by the induction on $\bra(T)+\bra(T')$})\\
&=\phi(T)\dw\phi(T')\quad(\text{by Eq.~(\mref{eq:df2})}).
\end{align*}
\noindent {\bf Case 2:} $T_1\neq |.$ Then
\begin{align*}
\phi(T\diamond T')&=\lf\phi(T_1)\rf_{\alpha_1}x_1\phi\Big(\treeoo{\cdb o\ocdx[2]{o}{a1}{160}{T_2}{left}
\ocdx[2]{o}{a2}{120}{T_3}{above}
\ocdx[2]{o}{a3}{60}{T_m}{above}
\ocdx[2]{o}{a4}{20}{T_{m+1}}{right}
\node at (140:\xch) {$x_2$};
\node at (90:\xch) {$\cdots$};
\node at (40:\xch) {$x_m$};
\node[below] at ($(o)!0.7!(a1)$) {$\alpha_2$};
\node[below] at ($(o)!0.7!(a4)$) {$\alpha_{m+1}$};
\node[right] at ($(o)!0.85!(a2)$) {$\alpha_3$};
\node[left] at ($(o)!0.85!(a3)$) {$\alpha_m\!$};
}\diamond \treeoo{\cdb o\ocdx[2]{o}{a1}{160}{T'_1}{left}
\ocdx[2]{o}{a2}{120}{T'_2}{above}
\ocdx[2]{o}{a3}{60}{T'_n}{above}
\ocdx[2]{o}{a4}{20}{T'_{n+1}}{right}
\node at (140:\xch) {$y_1$};
\node at (90:\xch) {$\cdots$};
\node at (40:\xch) {$y_n$};
\node[below] at ($(o)!0.7!(a1)$) {$\beta_1$};
\node[below] at ($(o)!0.7!(a4)$) {$\beta_{n+1}$};
\node[right] at ($(o)!0.85!(a2)$) {$\beta_2$};
\node[left] at ($(o)!0.85!(a3)$) {$\beta_n\!$};
}\Big)\quad(\text{by Eq.~(\mref{eq:df3})})\\
&=\lf\phi(T_1)\rf_{\alpha_1}x_1\phi\Big(\treeoo{\cdb o\ocdx[2]{o}{a1}{160}{T_2}{left}
\ocdx[2]{o}{a2}{120}{T_3}{above}
\ocdx[2]{o}{a3}{60}{T_m}{above}
\ocdx[2]{o}{a4}{20}{T_{m+1}}{right}
\node at (140:\xch) {$x_2$};
\node at (90:\xch) {$\cdots$};
\node at (40:\xch) {$x_m$};
\node[below] at ($(o)!0.7!(a1)$) {$\alpha_2$};
\node[below] at ($(o)!0.7!(a4)$) {$\alpha_{m+1}$};
\node[right] at ($(o)!0.85!(a2)$) {$\alpha_3$};
\node[left] at ($(o)!0.85!(a3)$) {$\alpha_m\!$};
}\Big)\dw\phi\Big(\treeoo{\cdb o\ocdx[2]{o}{a1}{160}{T'_1}{left}
\ocdx[2]{o}{a2}{120}{T'_2}{above}
\ocdx[2]{o}{a3}{60}{T'_n}{above}
\ocdx[2]{o}{a4}{20}{T'_{n+1}}{right}
\node at (140:\xch) {$y_1$};
\node at (90:\xch) {$\cdots$};
\node at (40:\xch) {$y_n$};
\node[below] at ($(o)!0.7!(a1)$) {$\beta_1$};
\node[below] at ($(o)!0.7!(a4)$) {$\beta_{n+1}$};
\node[right] at ($(o)!0.85!(a2)$) {$\beta_2$};
\node[left] at ($(o)!0.85!(a3)$) {$\beta_n\!$};
}\Big)\\
&\hspace{2cm}(\text{by the induction on $\bra(T)+\bra(T')$})\\
&=\phi(T)\dw\phi(T')\quad(\text{by Eq.~(\mref{eq:df3})}).
\end{align*}
This completes the proof.
\end{proof}
\noindent Now we are ready for our main result in this section. Let $j_X$ be the embedding given by
$$j_X:X\ra\bfk\calt,\, x\mapsto \stree x.$$
\begin{theorem}\mlabel{thm:free}
Let $X$ be a set and let $\Omega$ be a semigroup. The triple $\bigl(\bfk\calt,\diamond,(B^{+}_\omega)_{\omega\in\Omega}\bigr)$, together with the $j_X$, is the free Rota-Baxter family algebra of weight $\lambda$ on $X$.
\end{theorem}

\begin{proof}
We divide the proof into two steps.

{\bf Step 1:} We prove that $(\bfk\calt,\diamond,(B^+_\omega)_{\omega\in\Omega})$ is a Rota-Baxter family algebra. First, for $T_1, T_2,T_3\in\cxo T;$, we have
\begin{align*}
\phi\Big((T_1\diamond T_2)\diamond T_3\Big)&=\Big(\phi(T_1)\dw\phi(T_2)\Big)\dw\phi(T_3)\Big)\quad(\text{by Lemma~\mref{lem:compa}})\\
&=\phi(T_1)\dw\Big(\phi(T_2)\dw\phi(T_3)\Big)\quad(\text{by Lemma~\mref{lem:GS}~\mref{it:ideal}})\\
&=\phi\Big(T_1\diamond(T_2\diamond T_3)\Big)\quad(\text{by Lemma~\mref{lem:compa}}).
\end{align*}
Since $\psi\circ\phi=\id$ by Proposition~\mref{prop:iso}, we have
$$(T_1\diamond T_2)\diamond T_3=T_1\diamond(T_2\diamond T_3),$$
which proves the associativity. Second, for $T, T'\in\cxo T;$, we have
\begin{align*}
\phi\Big(B^{+}_\alpha(T)\diamond B^+_\beta(T')\Big)&=\phi\Big(B^+_\alpha(T)\Big)\dw\phi\Big(B^+_\beta(T')\Big)\quad(\text{by Lemma~\mref{lem:compa}})\\
&=\lf\phi(T)\rf_\alpha\dw\lf\phi(T')\rf_\beta\quad(\text{by Eq.~(\mref{eq:df1})})\\
&=\lf\phi(T)\dw\lf\phi(T')\rf_\beta+\lf\phi(T)\rf_\alpha\dw\phi(T')
+\lambda\phi(T)\dw\phi(T')\rf_{\alpha\beta}\\
&\hspace{2cm}(\text{by Lemma~\mref{lem:GS}~\mref{it:ideal}})\\
&=\phi\Big(B^+_{\alpha\beta}\Big(T\diamond B^+_\beta(T')+B^+_\alpha(T)\diamond T'+\lambda T\diamond T'\Big)\Big)\quad(\text{by Eq.~(\mref{eq:df1}) and Lemma~\mref{lem:compa}}).
\end{align*}
Again by $\psi\circ\phi=\id$, we have
$$B^{+}_\alpha(T)\diamond B^+_\beta(T')=B^+_{\alpha\beta}\Big(T\diamond B^+_\beta(T')+B^+_\alpha(T)\diamond T'+\lambda T\diamond T'\Big),$$
whence $(B^+_\omega)_{\omega\in\Omega}$ is a Rota-Baxter family on $(\bfk\calt, \diamond)$.

{\bf Step 2:} We prove the freeness of $(\bfk\calt, \diamond, (B^+_\omega)_{\omega\in\Omega})$. By Proposition~\mref{prop:iso}, $\phi:\bfk\calt\ra \bfk\frak{X}_\infty$ is a linear isomorphism. Further, by Eq.~(\mref{eq:df1}) and Lemma~\mref{lem:compa}, $\phi$ is a Rota-Baxter family isomorphism. Finally, since
$(\bfk\frak{X}_\infty,\dw,(\lf\,\rf_\omega)_{\omega\in\Omega})$ is a free Rota-Baxter family algebra on $X$ by Lemma~\mref{lem:GS}{~\mref{it:ideal}}, the result holds.
\end{proof}

If $\Omega$ is a trivial semigroup, all typed angularly decorated planar rooted trees
in $\calt$ have the same edge decorations, and so $\calt$ reduces to angularly decorated planar rooted trees
$\mathcal{T}(X)$ without edge decorations. Meanwhile, the unique $B^+_\omega$ with $\omega\in \Omega$ reduces to
the grafting operation $B^+$ applied in the Connes-Kreimer Hopf algebra~\mcite{CK98}.
As a consequence, we obtain a new construction of free Rota-Baxter algebra in terms of
$\mathcal{T}(X)$.

\begin{coro}
Let $X$ be a set and let $\Omega$ be a trivial semigroup.
Then the triple $(\bfk\mathcal{T}(X),\diamond, B^{+})$, together with the $j_X$, is the free
Rota-Baxter algebra of weight $\lambda$ on $X$.
\end{coro}

\begin{proof}
If $\Omega$ is a trivial semigroup, Eq.~(\mref{eq:RBF}) is precisely the Rota-Baxter equation.
So the result follows directly from Theorem~\mref{thm:free}.
\end{proof}

\section{Embedding free DFAs (resp.~TFAs) into free RBFAs}~\mlabel{sec:emb}
In this section, we show that the free dendriform (resp.~tridendriform) family algebra on a set $X$ is a canonical dendriform (resp.~tridendriform) family subalgebra of the free Rota-Baxter family algebra $\bfk\calt$ of weight zero (resp.~one).

\subsection{Embedding free DFAs into free RBFAs}
The concept of a dendriform family algebra was proposed in~\mcite{ZG}, as a generalization of a dendriform algebra invented by Loday~\mcite{Lod93} in the study of algebraic $K$-theory.
\begin{defn}\mlabel{defn:dend}\cite{ZG}
Let $\Omega$ be a semigroup. A {\bf dendriform family algebra} is a $\bfk$-module $D$ with a family of binary operations $(\prec_\omega, \succ_\omega)_{\omega\in\Omega}$ such that for $ x, y, z\in D$ and $\alpha,\beta\in \Omega$,
\begin{align*}
(x\prec_{\alpha} y) \prec_{\beta} z=\ & x \prec_{\bim{\alpha \beta}} (y\prec_{\beta} z+y \succ_{\alpha} z),\\ 
(x\succ_{\alpha} y)\prec_{\beta} z=\ & x\succ_{\alpha} (y\prec_{\beta} z),\\
(x\prec_{\beta} y+x\succ_{\alpha} y) \succ_{\bim{\alpha \beta}}z  =\ & x\succ_{\alpha}(y \succ_{\beta} z). 
\end{align*}
\end{defn}

\noindent Let $X$ be a set and let $\Omega$ be a semigroup. For $n\geq 1,$
let $$Y_{n}:=Y_{n,\,X,\,\Omega}:=\caltn\cap\,\{\text{planar binary trees}\},$$ that is, $Y_n$ is the set of $X$-angularly decorated $\Omega^{1}$-typed planar binary trees with $n+1$ leaves,
such that leaves are decorated by the identity 1 in $\Omega^{1}$ and internal edges are decorated by elements of $\Omega.$
Here are some examples for better understanding.
For convenience, we omit the decoration 1 in the sequel.
\begin{align*}
Y_1&=\left\{\stree x\Bigm|x\in X
\right\},\ \
Y_2=\left\{
\XX{\xxr{-5}5
\node at (-0.4,0) {$\alpha$};
\xxhu00x \xxhu{-5}5y
}, \,
\XX{\xxl55
\node at (0.4,0) {$\alpha$};
\xxhu00x \xxhu55y
}\Bigm|x,y\in X,\alpha\in\Omega
\right\},\\
Y_3&=\left\{
\XX[scale=1.6]{\xxr{-4}4\xxr{-7.5}{7.5}
\node at (-0.35,0.1) {$\alpha$};
\node at (-0.75,0.4) {$\beta$};
\xxhu00{x} \xxhu[0.1]{-4}4{y} \xxhu[0.1]{-7.5}{7.5}{z}
}, \,
\XX[scale=1.6]{\xxl44\xxl{7.5}{7.5}
\node at (0.35,0.1) {$\alpha$};
\node at (0.7,0.4) {$\beta$};
\xxhu00{x} \xxhu[0.1]44{y} \xxhu[0.1]{7.5}{7.5}{z}
}, \,
\XX[scale=1.6]{\xxr{-6}6\xxl66
\node at (-0.5,0.15) {$\beta$};
\node at (0.45,0.15) {$\alpha$};
\xxhu00{x} \xxhu66{y} \xxhu{-6}6{z}
}, \,
\XX[scale=1.6]{\xxr{-5}5\xxl{-2}8
\node at (-0.4,0.1) {$\alpha$};
\node at (-0.2,0.55) {\tiny $\beta$};
\xxhu00x
\xxhu[0.1]{-5}5{y} \xxhu[0.1]{-2}8{z}
}, \,
\XX[scale=1.6]{\xxl55\xxr28
\node at (0.45,0.15) {$\alpha$};
\node at (0.22,0.5) {\tiny $\beta$};
\xxhu[0.1]00{x\,}
\xxhu[0.1]55{\,y} \xxhu[0.1]28{z}
},\ldots \Bigg|\,x,y,z\in X,\alpha,\beta\in\Omega\right\}.
\end{align*}
For $T\in Y_{m}, U\in Y_{n}$, $x\in X$ and $\alpha,\beta\in\Omega^{1},$ the grafting $\Ve x,\alpha,\beta;$ of $T$ and $U$ over $x$ and $(\alpha,\beta)$ is defined to be $T\Ve x,\alpha,\beta; U\in Y_{m+n+1}$, obtained by adding a new vertex decorated by $x$ and joining the roots of $T$ and $U$ with the new vertex via two new edges decorated by $\alpha$ and $\beta$ respectively.

Conversely, for $T\in Y_n$ not equal to $|$, there is a unique decomposition
$$T=T^{l}\Ve x,\alpha,\beta; T^{r}\,\text{ for some }\, x\in X\,\text{ and }\, \alpha,\beta\in \Omega^{1}.$$

\noindent Denote by
$$\DDF:=\underset{n\geq 1}\bigoplus\,\bfk Y_{n}.$$

\begin{defn}\mlabel{defn:recur}
Let $X$ be a set and let $\Omega$ be a semigroup. Slightly extending \cite[Paragraph 3.1]{ZGM}, we define binary operations
$$\prec_\omega, \succ_\omega:\Big(\DDF\ot\DDF\Big)\oplus\Big(\bfk |\ot\DDF\Big)\oplus \Big(\DDF\ot\bfk |\Big)\ra \DDF\,\text{ for }\, \omega\in\Omega$$
recursively on $\dep(T)+\dep(U)$ by
\begin{enumerate}
\item $|\succ_\omega T:=T\prec_\omega |:=T\,\text{ and }\, |\prec_\omega T:=T\succ_\omega |:=0$ for $\omega\in\Omega$ and $T\in Y_{n}$ with $n\geq 1$.\mlabel{it:intial1}

\item For $T=T^{l}\Ve x,\alpha_1,\alpha_2; T^{r}$ and $U=U^{l}\Ve y,\beta_1,\beta_2; U^{r},$ define
\begin{align}
T\prec_\omega U:=& \ T^{l}\Ve x,\alpha_1,\alpha_2\omega; (T^{r}\prec_\omega U+T^{r}\succ_{\alpha_2} U),\mlabel{eq:recur1}\\
T\succ_\omega U:=& \ (T\prec_{\beta_1} U^{l}+T\succ_\omega U^{l})\Ve y,\omega\beta_1,\beta_2; U^{r},\,\text{ where }\, \omega \in \Omega.\mlabel{eq:recur}
\end{align}
\end{enumerate}
\end{defn}
\noindent Note that $|\prec_\omega |$ and $|\succ_\omega |$ are not defined for $\omega\in\Omega^1.$
Here we apply the convention
\begin{equation}
|\succ_1 T:=T\prec_1 |:=T\,\text{ and }\, |\prec_1 T:=T\succ_1 |:=0.
\mlabel{eq:conv}
\end{equation}
\noindent Let $j: X \ra \DDF$ be the map defined by $j(x)=\stree x$ for $x\in X$.
\begin{theorem}\mlabel{thm:free1}\cite{ZGM}
Let $X$ be a set and let $\Omega$ be a semigroup. Then $(\DDF,(\prec_\omega, \succ_\omega)_{\omega\in\Omega})$, together with the map $j$, is the free dendriform family algebra on $X$.
\end{theorem}

It is well-known that a Rota-Baxter algebra of weight $0$ induces a dendriform algebra~\cite{Agu99, EbFa}. Similarly, we have the following fact.

\begin{lemma}\mlabel{lem:dendt}\cite{ZG}
Let $X$ be a set and let $\Omega$ be a semigroup. The Rota-Baxter family algebra $\bigl(\bfk\calt,\diamond,\,(B^{+}_{\omega})_{\omega \in \Omega}\bigr)$ of weight $0$ induces a dendriform family algebra
$\bigl(\bfk\calt, \, (\prec'_\omega, \succ'_\omega)_{\omega\in\Omega}\bigr)$, where
\begin{align}
T\prec'_{{\omega}}U : = & \ T\diamond B^{+}_{\omega}(U)\,\text{ and }\, T\succ'_{{\omega}}U := B^{+}_{\omega}(T) \diamond U,\text{ for }\, T, U \in \calt.
\mlabel{eq:den}
\end{align}
\end{lemma}

Now we arrive at our main result.
\begin{theorem}\mlabel{thm:sub1}
Let $X$ be a set and let $\Omega$ be a semigroup. The free dendriform family algebra $\bigl(\DDF,(\prec_\omega, \succ_\omega)_{\omega\in\Omega}\bigr)$ on $X$ is a dendriform family subalgebra of the free Rota-Baxter family algebra $\bigl(\bfk\calt,\diamond,(B^{+}_\omega)_{\omega\in\Omega}\bigr)$ of weight $0$.
\end{theorem}

\begin{proof}
By Lemma~\mref{lem:dendt}, a Rota-Baxter family algebra $\bigl(\bfk\calt,\diamond,\,$ $(B^{+}_{\omega})_{\omega \in \Omega}\bigr)$ of weight $0$ induces a dendriform family algebra
$\bigl(\bfk\calt, \, (\prec'_\omega, \succ'_\omega)_{\omega\in\Omega}\bigr)$. Now we prove that $\bigl(\DDF, (\prec'_\omega, \succ'_\omega)_{\omega\in\Omega}\bigr)$ is a dendriform family subalgebra of $\bigl(\bfk\calt, \, (\prec'_\omega, \succ'_\omega)_{\omega\in\Omega}\bigr)$.

For $T=T^l\Ve x,\alpha_1,\alpha_2;T^r$ and $U=U^l\Ve y,\beta_1,\beta_2;U^r\in\DDF$
, we use induction on $\dep(T)+\dep(U)\geq 2$. For the initial step $\dep(T)+\dep(U)=2$, we have $T=\stree x$ and $U=\stree y.$ Then
\begin{align*}
T\prec'_\omega U&=T\diamond B^+_\omega(U)=\stree x\diamond B^+_\omega\Big(\stree y\Big)=\stree x\diamond \treeoo{\cdb[1.5] o\zhd{o/b}\cdlr o
\node[above=1pt] at (o) {$y$};
\node[left] at ($(o)!0.3!(ob)$) {$\omega$};
}=\XX{\xxl55
\node at (0.4,0) {$\omega$};
\xxhu00x \xxhu55y
}\in\DDF,\\
T\succ'_\omega U&=B^+_\omega(T)\diamond U=B^+_\omega\Big(\stree x\Big)\diamond\stree y
=\treeoo{\cdb[1.5] o\zhd{o/b}\cdlr o
\node[above=1pt] at (o) {$x$};
\node[left] at ($(o)!0.3!(ob)$) {$\omega$};
}\diamond \stree y=\XX{\xxr{-5}5
\node at (-0.45,0) {$\omega$};
\xxhu00y \xxhu{-5}5{x}
}\in\DDF.
\end{align*}
For the induction step $\dep(T)+\dep(U)\geq 3$, we only prove $T\prec'_\omega U\in\DDF,$ as $T\succ'_\omega U\in\DDF$ can be obtained similarly.

\noindent If $T^r=|$, then by Eq.~(\mref{eq:def5}),
\begin{align*}
T\prec'_\omega U&=T\diamond B^+_\omega(U)=\treeoo{\cdb o\ocdl[1.5]o{T^{l}}\cdr[1.5]o
\node[left=2pt] at ($(o)!0.25!(ol)$) {$\alpha_1$};
\node[above=2pt] at (o) {$x$};
}\diamond B^{+}_\omega(U)
=\treeoo{\cdb o\ocdl[1.5]o{T^{l}}\ocdr[1.5]o{U}
\node[left=2pt] at ($(o)!0.25!(ol)$) {$\alpha_1$};
\node[right=2pt] at ($(o)!0.25!(or)$) {$\omega$};
\node[above=2pt] at (o) {$x$};
}\in\DDF.
\end{align*}

\noindent If $T^r\neq |$, then
\begin{align*}
T\prec'_\omega U&=T\diamond B^+_\omega(U)=\treeoo{\cdb o\ocdl[1.5]o{T^{l}}\ocdr[1.5]o{T^{r}}
\node[left=2pt] at ($(o)!0.25!(ol)$) {$\alpha_1$};
\node[right=2pt] at ($(o)!0.25!(or)$) {$\alpha_2$};
\node[above=2pt] at (o) {$x$};
}\diamond B^{+}_\omega(U)\\
&=\Big(\treeoo{\cdb o\ocdl[1.5]o{T^{l}}\cdr[1.5]o
\node[left=2pt] at ($(o)!0.25!(ol)$) {$\alpha_1$};
\node[above=2pt] at (o) {$x$};
}\diamond B^{+}_{\alpha_2}(T^{r})\Big)\diamond B^{+}_\omega(U)\quad\text{(by Eq.~(\mref{eq:def5}))}\\
&=\treeoo{\cdb o\ocdl[1.5]o{T^{l}}\cdr[1.5]o
\node[left=2pt] at ($(o)!0.25!(ol)$) {$\alpha_1$};
\node[above=2pt] at (o) {$x$};
}\diamond \Big(B^{+}_{\alpha_2}(T^{r})\diamond B^{+}_\omega(U)\Big)\\
&=\treeoo{\cdb o\ocdl[1.5]o{T^{l}}\cdr[1.5]o
\node[left=2pt] at ($(o)!0.25!(ol)$) {$\alpha_1$};
\node[above=2pt] at (o) {$x$};
}\diamond B^{+}_{\alpha_2\omega}\Big(T^{r}\diamond B^{+}_\omega(U)+B^{+}_{\alpha_2}(T^{r})\diamond U\Big)\quad(\text{by Theorem ~\mref{thm:free} of $\lambda=0$)}\\
&=\treeoo{\cdb o\ocdl[1.5]o{T^{l}}\ocdr[1.5]o{T^{r}\diamond B^{+}_\omega(U)}
\node[left=2pt] at ($(o)!0.25!(ol)$) {$\alpha_1$};
\node[right=2pt] at ($(o)!0.25!(or)$) {$\alpha_2\omega$};
\node[above=2pt] at (o) {$x$};
}+
\treeoo{\cdb o\ocdl[1.5]o{T^{l}}\ocdr[1.5]o{B^{+}_{\alpha_2}(T^{r})\diamond U}
\node[left=2pt] at ($(o)!0.25!(ol)$) {$\alpha_1$};
\node[right=2pt] at ($(o)!0.25!(or)$) {$\alpha_2\omega$};
\node[above=2pt] at (o) {$x$};
}\quad\text{(by Eq.~(\mref{eq:def5}))}\\
&=\treeoo{\cdb o\ocdl[1.5]o{T^{l}}\ocdr[1.5]o{T^{r}\prec'_\omega U}
\node[left=2pt] at ($(o)!0.25!(ol)$) {$\alpha_1$};
\node[right=2pt] at ($(o)!0.25!(or)$) {$\alpha_2\omega$};
\node[above=2pt] at (o) {$x$};
}+
\treeoo{\cdb o\ocdl[1.5]o{T^{l}}\ocdr[1.5]o{T^{r}\succ'_{\alpha_2} U}
\node[left=2pt] at ($(o)!0.25!(ol)$) {$\alpha_1$};
\node[right=2pt] at ($(o)!0.25!(or)$) {$\alpha_2\omega$};
\node[above=2pt] at (o) {$x$};
}\quad(\text{by Eq.~(\mref{eq:den})}).
\end{align*}
By the induction hypothesis, $T^{r}\prec'_\omega U,T^{r}\succ'_{\alpha_2}U\in\DDF$ and so $T\prec'_\omega U\in\DDF.$
Hence $\bigl(\DDF,(\prec'_\omega,\succ'_\omega)_{\omega\in\Omega}\bigr)$ is a dendriform family subalgebra of $\bigl(\bfk\calt,(\prec'_\omega,\succ'_\omega)_{\omega\in\Omega}\bigr)$. We are left to show
\begin{equation}
T\prec_\omega U=T\prec'_\omega U\,\text{ and }\, T\succ_\omega U=T\succ'_\omega U,
\mlabel{eq:prsu}
\end{equation}
for $T=T^l\Ve x,\alpha_1,\alpha_2;T^r$ and $U=U^l\Ve y,\beta_1,\beta_2;U^r.$
We use induction on $\dep(T)+\dep(U)\geq 2$.
For the initial step  $\dep(T)+\dep(U)\geq 2$, we have $T=\stree x$ and $U=\stree y.$ Then
\begin{align*}
T\prec_\omega U&=(|\Ve x,1,1;|)\prec_\omega \stree y\\
&=|\Ve x,1,\omega;\biggl(|\prec_\omega \stree y+|\succ_1\stree y\biggr)\quad(\text{by Eq.~(\mref{eq:recur1})})\\
&=|\Ve x,1,\omega;\stree y\quad(\text{by Item~\mref{it:intial1} of Definition~\mref{defn:recur} and Eq.~(\mref{eq:conv})})\\
&=\XX{\xxl55
\node at (0.4,0) {$\omega$};
\xxhu00x \xxhu55y
}=\stree x\diamond \treeoo{\cdb[1.5] o\zhd{o/b}\cdlr o
\node[above=1pt] at (o) {$y$};
\node[left] at ($(o)!0.3!(ob)$) {$\omega$};
}\quad\text{(by Eq.~(\mref{eq:def5}))}\\
&=\stree x\diamond B^{+}_\omega\Big(\stree y\Big)
=T\diamond B^{+}_\omega(U)=T\prec'_\omega U\quad(\text{by Eq.~(\mref{eq:den})}).
\end{align*}
Similarly,
\begin{align*}
T\succ_\omega U&=\stree x\succ_\omega (|\Ve y,1,1;|)\\
&=\biggl(\stree x\prec_1 |+\stree x\succ_\omega |\biggr)\Ve y,\omega,1;|\quad(\text{by Eq.~(\mref{eq:recur})})\\
&=\stree x\Ve y,\omega,1;|\quad(\text{by Item~\mref{it:intial1} of Definition~\mref{defn:recur} and Eq.~(\mref{eq:conv})})\\
&=\XX{\xxr{-5}5
\node at (-0.45,0) {$\omega$};
\xxhu00y \xxhu{-5}5{x}
}
=\treeoo{\cdb[1.5] o\zhd{o/b}\cdlr o
\node[above=1pt] at (o) {$x$};
\node[left] at ($(o)!0.3!(ob)$) {$\omega$};
}\diamond \stree y\quad\text{(by Eq.~(\mref{eq:def4}))}\\
&=B^{+}_\omega\Big(\stree x\Big)\diamond \stree y
=B^{+}_\omega(T)\diamond U
=T\succ'_\omega U\quad(\text{by Eq.~(\mref{eq:den})}).
\end{align*}
Consider the induction step of $\dep(T)+\dep(U)\geq 3.$ We first prove the first equation of Eq.~(\mref{eq:prsu}).
\noindent If $T^{r}=|$, then
\begin{align*}
T\prec_\omega U&=(T^{l}\Ve x,\alpha_1,1;|)\prec_\omega U\\
&=T^{l}\Ve x,\alpha_1,\omega;(|\prec_\omega U+|\succ_1U)\quad\text{(by Eq.~(\mref{eq:recur1}))}\\
&=T^{l}\Ve x,\alpha_1,\omega;U\quad(\text{by Item~\mref{it:intial1} of Definition~\mref{defn:recur} and Eq.~(\mref{eq:conv})})\\
&=\treeoo{\cdb o\ocdl[1.5]o{T^{l}}\ocdr[1.5]o{U}
\node[left=2pt] at ($(o)!0.25!(ol)$) {$\alpha_1$};
\node[right=2pt] at ($(o)!0.25!(or)$) {$\omega$};
\node[above=2pt] at (o) {$x$};
}
=\treeoo{\cdb o\ocdl[1.5]o{T^{l}}\cdr[1.5]o
\node[left=2pt] at ($(o)!0.25!(ol)$) {$\alpha_1$};
\node[above=2pt] at (o) {$x$};
}\diamond B^{+}_\omega(U)\quad\text{(by Eq.~(\mref{eq:def5}))}\\
&=T\diamond B^{+}_\omega(U)=T\prec'_\omega U\quad\text{(by Eq.~(\mref{eq:den}))}.
\end{align*}
\noindent If $T^{r}\neq |$, then
\begin{align*}
T\prec_\omega U&=(T^{l}\Ve x,\alpha_1,\alpha_2;T^{r})\prec_\omega U\\
&=T^{l}\Ve x,\alpha_1,\alpha_2\omega;(T^{r}\prec_\omega U+T^{r}\succ_{\alpha_2}U)\quad\text{(by Eq.~(\mref{eq:recur1}))}\\
&=T^{l}\Ve x,\alpha_1,\alpha_2\omega;(T^{r}\prec'_\omega U+T^{r}\succ'_{\alpha_2}U)\quad\text{(by the induction hypothesis)}\\
&=T^{l}\Ve x,\alpha_1,\alpha_2\omega;\Big(T^{r}\diamond B^{+}_\omega(U)+B^{+}_{\alpha_2}(T^{r})\diamond U\Big)\quad\text{(by Eq.~(\mref{eq:den}))}\\
&=\treeoo{\cdb o\ocdl[1.5]o{T^{l}}\ocdr[1.5]o{T^{r}\diamond B^{+}_\omega(U)}
\node[left=2pt] at ($(o)!0.25!(ol)$) {$\alpha_1$};
\node[right=2pt] at ($(o)!0.25!(or)$) {$\alpha_2\omega$};
\node[above=2pt] at (o) {$x$};
}+
\treeoo{\cdb o\ocdl[1.5]o{T^{l}}\ocdr[1.5]o{B^{+}_{\alpha_2}(T^{r})\diamond U}
\node[left=2pt] at ($(o)!0.25!(ol)$) {$\alpha_1$};
\node[right=2pt] at ($(o)!0.25!(or)$) {$\alpha_2\omega$};
\node[above=2pt] at (o) {$x$};
}\\
&=\treeoo{\cdb o\ocdl[1.5]o{T^{l}}\cdr[1.5]o
\node[left=2pt] at ($(o)!0.25!(ol)$) {$\alpha_1$};
\node[above=2pt] at (o) {$x$};
}\diamond B^{+}_{\alpha_2\omega}\Big(T^{r}\diamond B^{+}_\omega(U)+B^{+}_{\alpha_2}(T^{r})\diamond U\Big)\quad\text{(by Eq.~(\mref{eq:def5}))}\\
&=\treeoo{\cdb o\ocdl[1.5]o{T^{l}}\cdr[1.5]o
\node[left=2pt] at ($(o)!0.25!(ol)$) {$\alpha_1$};
\node[above=2pt] at (o) {$x$};
}\diamond \Big(B^{+}_{\alpha_2}(T^{r})\diamond B^{+}_\omega(U)\Big)\quad\text{(by Theorem ~\mref{thm:free} of $\lambda=0$)}\\
&=\Big(\treeoo{\cdb o\ocdl[1.5]o{T^{l}}\cdr[1.5]o
\node[left=2pt] at ($(o)!0.25!(ol)$) {$\alpha_1$};
\node[above=2pt] at (o) {$x$};
}\diamond B^{+}_{\alpha_2}(T^{r})\Big)\diamond B^{+}_\omega(U)\\
&=\treeoo{\cdb o\ocdl[1.5]o{T^{l}}\ocdr[1.5]o{T^{r}}
\node[left=2pt] at ($(o)!0.25!(ol)$) {$\alpha_1$};
\node[right=2pt] at ($(o)!0.25!(or)$) {$\alpha_2$};
\node[above=2pt] at (o) {$x$};
}\diamond B^{+}_\omega(U)\quad\text{(by Eq.~(\mref{eq:def5}))}\\
&=T\diamond B^{+}_\omega(U)=T\prec'_\omega U\quad\text{(by Eq.~(\mref{eq:den}))}.
\end{align*}
We next prove the second statement of Eq.~(\mref{eq:prsu}). If $U^{l}=|,$ then
\begin{align*}
T\succ_\omega U&=T\succ_\omega(|\Ve y,1,\beta_2;U^{r})\\
&=(T\succ_\omega |+T\prec_1|)\Ve y,\omega,\beta_2;U^{r}\quad\text{(by Eq.~(\mref{eq:recur}))}\\
&=T\Ve y,\omega,\beta_2;U^{r}\quad(\text{by Item~\mref{it:intial1} of Definition~\mref{defn:recur} and Eq.~(\mref{eq:conv})})\\
&=\treeoo{\cdb o\ocdl[1.5]o{T}\ocdr[1.5]o{U^{r}}
\node[left=2pt] at ($(o)!0.25!(ol)$) {$\omega$};
\node[right=2pt] at ($(o)!0.25!(or)$) {$\beta_2$};
\node[above=2pt] at (o) {$y$};
}
=B^{+}_\omega(T)\diamond \treeoo{\cdb o\cdl[1.5]o\ocdr[1.5]o{U^{r}}
\node[right=2pt] at ($(o)!0.25!(or)$) {$\beta_2$};
\node[above=2pt] at (o) {$y$};
}\quad\text{(by Eq.~(\mref{eq:def4}))}\\
&=B^{+}_\omega(T)\diamond U=T\succ'_\omega U\quad\text{(by Eq.~(\mref{eq:den}))}.
\end{align*}

\noindent If $U^{l}\neq |$, then
\begin{align*}
T\succ_\omega U&=T\succ_\omega(U^{l}\Ve y,\beta_1,\beta_2;U^{r})\\
&=(T\succ_\omega U^{l}+T\prec_{\beta_1}U^{l})\Ve y,\omega\beta_1,\beta_2;U^{r}\quad\text{(by Eq.~(\mref{eq:recur}))}\\
&=(T\succ'_\omega U^{l}+T\prec'_{\beta_1}U^{l})\Ve y,\omega\beta_1,\beta_2;U^{r}\quad\text{(by the induction hypothesis)}\\
&=\Big(B^{+}_\omega(T)\diamond U^{l}+T\diamond B^{+}_{\beta_1}(U^{l})\Big)\Ve y,\omega\beta_1,\beta_2;U^{r}\quad\text{(by Eq.~\mref{eq:den})}\\
&=\treeoo{\cdb o\ocdl[1.5]o{B^{+}_\omega(T)\diamond U^{l}}\ocdr[1.5]o{U^{r}}
\node[left=2pt] at ($(o)!0.25!(ol)$) {$\omega\beta_1$};
\node[right=2pt] at ($(o)!0.25!(or)$) {$\beta_2$};
\node[above=2pt] at (o) {$y$};
}+\treeoo{\cdb o\ocdl[1.5]o{T\diamond B^{+}_{\beta_1}(U^{l})}\ocdr[1.5]o{U^{r}}
\node[left=2pt] at ($(o)!0.25!(ol)$) {$\omega\beta_1$};
\node[right=2pt] at ($(o)!0.25!(or)$) {$\beta_2$};
\node[above=2pt] at (o) {$y$};
}\\
&=B^{+}_{\omega\beta_1}\Big(B^{+}_\omega(T)\diamond U^{l}+T\diamond B^{+}_{\beta_1}(U^{l})\Big)\diamond \treeoo{\cdb o\cdl[1.5]o\ocdr[1.5]o{U^{r}}
\node[right=2pt] at ($(o)!0.25!(or)$) {$\beta_2$};
\node[above=2pt] at (o) {$y$};
}\quad\text{(by Eq.~(\mref{eq:def4}))}\\
&=\Big(B^{+}_\omega(T)\diamond B^{+}_{\beta_1}(U^{l})\Big)\diamond \treeoo{\cdb o\cdl[1.5]o\ocdr[1.5]o{U^{r}}
\node[right=2pt] at ($(o)!0.25!(or)$) {$\beta_2$};
\node[above=2pt] at (o) {$y$};
}\quad\text{(by Theorem ~\mref{thm:free} of weight $0$)}\\
&=B^{+}_\omega(T)\diamond \Big(B^{+}_{\beta_1}(U^{l})\diamond \treeoo{\cdb o\cdl[1.5]o\ocdr[1.5]o{U^{r}}
\node[right=2pt] at ($(o)!0.25!(or)$) {$\beta_2$};
\node[above=2pt] at (o) {$y$};
}\Big)\\
&=B^{+}_\omega(T)\diamond \treeoo{\cdb o\ocdl[1.5]o{U^{l}}\ocdr[1.5]o{U^{r}}
\node[left=2pt] at ($(o)!0.25!(ol)$) {$\beta_1$};
\node[right=2pt] at ($(o)!0.25!(or)$) {$\beta_2$};
\node[above=2pt] at (o) {$y$};
}\quad\text{(by Eq.~(\mref{eq:def4}))}\\
&=B^{+}_\omega(T)\diamond U=T\succ'_\omega U\quad\text{(by Eq.~(\mref{eq:den}))}.
\end{align*}

This completes the proof.
\end{proof}

\subsection{Embedding free TFAs into free RBFAs}
The concept of a tridendriform family algebra was introduced in~\cite{ZG}, which is a generalization of a tridendriform algebra invented by Loday and Ronco~~\mcite{LoRo04} in the study of polytopes and Koszul duality.

\begin{defn}\cite{ZG}\label{defn:tridend}
Let $\Omega$ be a semigroup. A {\bf tridendriform family algebra} is a $\bfk$-module $T$ equipped with
a family of binary operations $(\prec_{\omega}, \succ_\omega)_{\omega \in \Omega}$
and a binary operation $\cdot$ such that for $ x, y, z\in T$ and $\alpha,\beta\in \Omega$,
\begin{align*}
(x\prec_{\alpha} y)\prec_{\beta} z=\ & x\prec_{\bim{\alpha \beta}} (y\prec_{\beta} z+y \succ_{\alpha} z + y \cdot z),\\ 
(x\succ_{\alpha} y)\prec_{\beta} z=\ & x\succ_{\alpha}(y\prec_{\beta} z), \\ 
(x\prec_{\beta} y + x\succ_{\alpha} y + x \cdot y) \succ_{\bim{\alpha \beta}}z = \ & x\succ_{\alpha} (y\succ_{\beta} z),\\ 
(x\succ_{\alpha} y)\cdot  z=\  & x\succ_{\alpha}(y \cdot  z),\\ 
(x\prec_{\alpha} y)\cdot  z= \ & x \cdot (y\succ_{\alpha} z),\\ 
(x\cdot y)\prec_{\alpha} z= \ & x\cdot (y\prec_{\alpha} z),\\ 
(x\cdot y)\cdot  z= \ & x\cdot  (y\cdot  z).
\end{align*}
\end{defn}

Let $X$ be a set and let $\Omega$ be a semigroup.
For $n\geq 1,$ let
$$T_{n}:=T_{n,\,X,\,\Omega}:=\caltn\cap\,\{\text{Schr\"oder trees}\},$$
 be the set of $X$-angularly
decorated $\Omega^{1}$-typed Schr\"oder trees with $n+1$ leaves, such that
the leaves are decorated by the identity 1 in $\Omega^{1}$ and internal edges are decorated by elements of $\Omega$.
Here are some examples. Note that the edge decoration 1 is omitted for convenience.
\begin{align*}
T_1=& \ \left\{\stree x\Bigm|x\in X\right\},
\qquad T_2=\left\{
\XX{\xxr{-5}5
\node at (-0.4,0) {$\alpha$};
\xxhu00x \xxhu{-5}5y
},
\XX{\xxl55
\node at (0.4,0) {$\alpha$};
\xxhu00x \xxhu55y
},
\XX{\xx002
\xxh0023{x\ }{0.5} \xxh0012{\ \,y}{0.4}
}\Bigm|x,y\in X,\alpha\in\Omega
\right\},\\
T_3=& \ \left\{
\XX[scale=1.6]{\xxr{-4}4\xxr{-7.5}{7.5}
\node at (-0.3,0) {$\alpha$};
\node at (-0.72,0.4) {$\beta$};
\xxhu00x \xxhu[0.1]{-4}4{\,y} \xxhu[0.1]{-7.5}{7.5}{z}
},
\XX[scale=1.6]{\xxr{-5}5\xxl{-2}8
\node at (-0.4,0.1) {$\alpha$};
\node at (-0.2,0.55) {\tiny $\beta$};
\xxhu00x
\xxhu[0.1]{-5}5{y} \xxhu[0.1]{-2}8{z}
},
\XX[scale=1.6]{\xxr{-4}4\xx{-4}42
\node at (-0.3,0) {$\alpha$};
\xxhu00x \xxh{-4}423{y\ \,}{0.3} \xxh{-4}412{\ \, z}{0.3}
},
\XX[scale=1.6]{\xx00{1.6}\xx00{2.4}
\xxh001{1.6}{\ \ \,z}{0.6}
\xxh00{1.6}{2.4}{y}{0.5}
\xxh00{2.4}3{x\ \ }{0.6}
},
\XX[scale=1.6]{\xxlr0{7.5} \draw(0,0)--(0,0.75);
\xxh0023{x\ \,}{0.3} \xxhu[0.12]0{7.5}{z}
\node at (0.16,0.35) {\tiny $y$};
\node at (0.17,0.62) {\tiny $\alpha$};
},
\ldots\Bigg|\,x,y,z\in X,\alpha,\beta\in\Omega
\right\}.
\end{align*}
Denote by $$\DTF:=\bigoplus_{n\geq 1}\bfk T_{n}.$$

For typed angularly decorated Schr\"oder trees $T^{(i)}\in T_{n_i}, 0\leq i\leq k,$ and $x_1,\ldots,x_k\in X, \alpha_0,\ldots,\alpha_k\in\Omega^{1},$ the grafting $\bigvee$ of $T^{(i)}$ over $(x_1,\ldots,x_k)$ and $(\alpha_0,\ldots,\alpha_k)$ is
\begin{equation}
T=\bigv x_1,\dots,x_k;k+1;\alpha_0,\dots,\alpha_k;
(T^{(0)},\dots,T^{(k)})
\mlabel{eq:texpr}
\end{equation}
obtained by joining the $k+1$ roots of $T^{(i)}$ to a new vertex angularly decorated by $(x_1,\ldots,x_k)$ and
decorating the new edges from left to right by $\alpha_0,\ldots,\alpha_k$ respectively.

Conversely, any typed angularly decorated Schr\"oder tree $T\in \DTF$ can be uniquely expressed as such a grafting in Eq.~(\mref{eq:texpr}) of lower depth typed angularly decorated planar rooted trees.

\begin{defn}\mlabel{defn:trirec}
Let $X$ be a set and let $\Omega$ be a semigroup. Slightly extending \cite[Definition 3.7]{ZGM}, we
define binary operations 
$$ \prec_\omega, \succ_\omega,\cdot:\Big(\DTF\ot\DTF\Big)\oplus\Big(\bfk |\ot\DTF\Big)\oplus \Big(\DTF\ot\bfk |\Big)\ra \DTF\,\text{ for }\, \omega\in\Omega$$
 recursively on $\dep(T)+\dep(U)$ by
\begin{enumerate}
\item $|\succ_\omega T:=T\prec_\omega |:=T,\, |\prec_\omega T:=T\succ_\omega |:=0 \text{ and }\, |\cdot T:=T\cdot |:=0$ for $\omega\in \Omega$ and $T\in T_{n}$ with $n\geq 1.$\mlabel{it:trida}

\item Let
\begin{align*}
T=& \ \bigv x_1,\ldots,x_m;m+1;\alpha_0,\ldots,\alpha_m;(T^{(0)},\ldots, T^{(m)})\in T_{m},\\
U=& \ \bigv y_1,\ldots,y_n;n+1;\beta_0,\ldots,\beta_n;(U^{(0)},\ldots, U^{(n)})\in T_{n},
\end{align*}
define
\begin{align}
 T\prec_\omega U:=& \ \bigv x_1,\ldots,x_{m-1},x_m;m+1;\alpha_0,\ldots,\alpha_{m-1},\alpha_m\omega; (T^{(0)},\ldots,T^{(m-1)},T^{(m)}\succ_{\alpha_m} U+T^{(m)}\prec_\omega U+T^{(m)}\cdot U),\mlabel{eq:tdpre}\\
T\succ_\omega U:=& \ \bigv y_1,y_2,\ldots,y_n;n+1;\omega\beta_0,\beta_1,\ldots,\beta_n;  (T\succ_\omega U^{(0)}+T\prec_{\beta_0} U^{(0)}+T\cdot U^{(0)},U^{(1)},\ldots,U^{(n)}),\mlabel{eq:tdsuc}\\
 T\cdot U:=& \ \bigv x_1,\ldots,x_{m-1},x_m,y_1,\ldots,y_n;m+n+1;\alpha_0,\ldots,\alpha_{m-1},\alpha_m\beta_0,\beta_1,
 \ldots,\beta_n;
 (T^{(0)},\ldots,T^{(m-1)},T^{(m)}\succ_{\alpha_m} U^{(0)}+T^{(m)}\prec_{\beta_0} U^{(0)}+T^{(m)}\cdot U^{(0)},\mlabel{eq:tdcdot}\\
 & \hspace{5cm} U^{(1)},\ldots,U^{(n)}).\nonumber
\end{align}
\end{enumerate}
\end{defn}
\noindent Note that $|\prec_\omega |, |\succ_\omega |$ and $|\cdot |$ are not defined. We employ the convention that
\begin{equation}
|\prec_1 |+|\succ_1 |+|\cdot | :=|,
\mlabel{eq:1star1}
\end{equation}
and
\begin{equation}
|\succ_1 T:=T\prec_1 |:=T\,\text{ and }\, |\prec_1 T:=T\succ_1 |:=0.
\mlabel{eq:con}
\end{equation}

Let $j: X \ra \DTF$ be the map defined by $j(x)=\stree x$ for $x\in X$.
\begin{lemma}\mlabel{lem:free2}\cite{ZGM}
Let $X$ be a set and let $\Omega$ be a semigroup. Then $(\DTF,(\prec_{\omega}, \succ_\omega)_{\omega \in \Omega},\cdot)$, together with the map $j$, is the free tridendriform family algebra on $X$.
\end{lemma}

Let us recall the following fact.

\begin{lemma}\mlabel{lem:tendt}\cite{ZG}
Let $X$ be a set and let $\Omega$ be a semigroup. The Rota-Baxter family algebra $(\bfk\calt, \diamond,\,(B^{+}_{\omega})_{\omega \in \Omega})$ of weight $1$ induces a tridendriform family algebra
$(\bfk\calt, \, (\prec'_{\omega}, \succ'_\omega)_{\omega \in \Omega}, \cdot'\,)$, where
\begin{align}
\begin{split}
T\prec'_{{\omega}}U : = & \ T \diamond B^{+}_{\omega}(U),\,\, T\succ'_{{\omega}}U := B^{+}_{\omega}(T) \diamond U\,\text{ and }\\
 T\cdot'  U := & \  T\diamond U,\,\text{ for }\, T, U \in \calt.
 \end{split}
 \mlabel{eq:trid}
\end{align}
\end{lemma}

Now we are ready for our main result.
\begin{theorem}\label{thm:main}
Let $X$ be a set and let $\Omega$ be a semigroup. The free tridendriform family algebra $\bigl(\DTF,(\prec_{\omega}, \succ_\omega)_{\omega \in \Omega},\cdot\bigr)$ on $X$ is a tridendriform family subalgebra of the free Rota-Baxter family algebra $\bigl(\bfk\calt,\diamond,(B^{+}_{\omega})_{\omega \in \Omega}\bigr)$ of weight $1$.
\end{theorem}

\begin{proof}
By Lemma~\mref{lem:tendt}, a Rota-Baxter family algebra $\bigl(\bfk\calt,\diamond,\,$ $(B^{+}_{\omega})_{\omega \in \Omega}\bigr)$ of weight $1$ induces a tridendriform family algebra
$\bigl(\bfk\calt, \, (\prec'_{\omega}, \succ'_\omega)_{\omega \in \Omega},\cdot'\bigr).$
Now we prove that $\bigl(\DTF,(\prec'_\omega,\succ'_\omega)_{\omega\in\Omega},\cdot'\bigr)$ is a tridendriform family subalgebra of $\bigl(\bfk\calt, \, (\prec'_{\omega}, \succ'_\omega)_{\omega \in \Omega},\cdot'\bigr).$

For
\begin{align*}
T=& \ \bigv x_1,\ldots,x_m;m+1;\alpha_0,\ldots,\alpha_m;(T^{(0)},\ldots, T^{(m)})\in T_{m},\\
U=& \ \bigv y_1,\ldots,y_n;n+1;\beta_0,\ldots,\beta_n;(U^{(0)},\ldots, U^{(n)})\in T_{n}\,\text{ with }\,m,n\geq 1,
\end{align*}
we use induction on $\dep(T)+\dep(U)\geq 2.$
For the initial step $\dep(T)+\dep(U)=2$, we have
$T=\treeoo{\cdb o\cdx[1.5]{o}{a1}{160}
\cdx[1.5]{o}{a2}{120}
\cdx[1.5]{o}{a3}{60}
\cdx[1.5]{o}{a4}{20}
\node at (140:1.2*\xch) {$x_1$};
\node at (90:1.2*\xch) {$\cdots$};
\node at (40:1.2*\xch) {$x_m$};
}$ and $U=\treeoo{\cdb o\cdx[1.5]{o}{a1}{160}
\cdx[1.5]{o}{a2}{120}
\cdx[1.5]{o}{a3}{60}
\cdx[1.5]{o}{a4}{20}
\node at (140:1.2*\xch) {$y_1$};
\node at (90:1.2*\xch) {$\cdots$};
\node at (40:1.2*\xch) {$y_n$};
}$. Then
\begin{align*}
T\prec'_\omega U&=T\diamond B^+_\omega(U)=\treeoo{\cdb o\cdx[1.5]{o}{a1}{160}
\cdx[1.5]{o}{a2}{120}
\cdx[1.5]{o}{a3}{60}
\cdx[1.5]{o}{a4}{20}
\node at (140:1.2*\xch) {$x_1$};
\node at (90:1.2*\xch) {$\cdots$};
\node at (40:1.2*\xch) {$x_m$};
}\diamond B^{+}_\omega(U)
=\treeoo{\cdb o\cdx[1.5]{o}{a1}{160}
\cdx[1.5]{o}{a2}{120}
\cdx[1.5]{o}{a3}{60}
\ocdx[1.5]{o}{a4}{20}{U}{20}
\node[below] at ($(o)!0.7!(a4)$) {$\omega$};
\node at (140:1.2*\xch) {$x_1$};
\node at (90:1.2*\xch) {$\cdots$};
\node at (40:1.2*\xch) {$x_m$};
}\in\DTF,\\
T\succ'_\omega U&=B^+_\omega(T)\diamond U
=B^{+}_\omega(T)\diamond \treeoo{\cdb o\cdx[1.5]{o}{a1}{160}
\cdx[1.5]{o}{a2}{120}
\cdx[1.5]{o}{a3}{60}
\cdx[1.5]{o}{a4}{20}
\node at (140:1.2*\xch) {$y_1$};
\node at (90:1.2*\xch) {$\cdots$};
\node at (40:1.2*\xch) {$y_n$};
}
=\treeoo{\cdb o\ocdx[1.5]{o}{a1}{160}{T}{160}
\cdx[1.5]{o}{a2}{120}
\cdx[1.5]{o}{a3}{60}
\cdx[1.5]{o}{a4}{20}
\node[below] at ($(o)!0.7!(a1)$) {$\omega$};
\node at (140:1.2*\xch) {$y_1$};
\node at (90:1.2*\xch) {$\cdots$};
\node at (40:1.2*\xch) {$y_n$};
}\in\DTF\,\text{ and }\,\\
T\cdot'U&=T\diamond U=
\treeoo{\cdb o\cdx[1.5]{o}{a1}{160}
\cdx[1.5]{o}{a2}{120}
\cdx[1.5]{o}{a3}{60}
\cdx[1.5]{o}{a4}{20}
\node at (140:1.2*\xch) {$x_1$};
\node at (90:1.2*\xch) {$\cdots$};
\node at (40:1.2*\xch) {$x_m$};
}\diamond \treeoo{\cdb o\cdx[1.5]{o}{a1}{160}
\cdx[1.5]{o}{a2}{120}
\cdx[1.5]{o}{a3}{60}
\cdx[1.5]{o}{a4}{20}
\node at (140:1.2*\xch) {$y_1$};
\node at (90:1.2*\xch) {$\cdots$};
\node at (40:1.2*\xch) {$y_n$};
}
=\treeoo{\cdb o
\cdx[2.2]{o}{a1}{170}
\cdx[2.2]{o}{a2}{145}
\cdx[2.2]{o}{a3}{115}
\cdx[2.2]{o}{a4}{90}
\cdx[2.2]{o}{a5}{65}
\cdx[2.2]{o}{a6}{35}
\cdx[2.2]{o}{a7}{10}
\node at ($(a1)!0.5!(a2)$) {$x_1$};
\node[rotate=40] at ($(a2)!0.5!(a3)$) {$\cdots$};
\node at ($(a3)!0.5!(a4)$) {$x_m$};
\node at ($(a4)!0.5!(a5)$) {$y_1$};
\node[rotate=-40] at ($(a5)!0.5!(a6)$) {$\cdots$};
\node at ($(a6)!0.5!(a7)$) {$y_n$};
}\in\DTF.
\end{align*}
For the induction step $\dep(T)+\dep(U)\geq 3$, we first prove $T\prec'_\omega U\in\DTF,$ there are two cases to consider.

\noindent {\bf Case 1:} $T^{(m)}=|$. Then
\begin{align*}
T\prec'_\omega U=T\diamond B^+_\omega(U)=\treeoo{\cdb o\ocdx[2]{o}{a1}{160}{T^{(0)}}{left}
\ocdx[2]{o}{a2}{120}{T^{(1)}}{above}
\ocdx[2]{o}{a3}{60}{T^{(m-1)}}{above}
\cdx[2]{o}{a4}{20}
\node at (140:\xch) {$x_1$};
\node at (90:\xch) {$\cdots$};
\node at (40:\xch) {$x_m$};
\node[below] at ($(o)!0.7!(a1)$) {$\alpha_0$};
\node[right] at ($(o)!0.75!(a2)$) {$\alpha_1$};
\node[right] at ($(o)!0.75!(a3)$) {$\alpha_{m-1}\!$};
}\diamond B^+_\omega(U)
=\treeoo{\cdb o\ocdx[2]{o}{a1}{160}{T^{(0)}}{left}
\ocdx[2]{o}{a2}{120}{T^{(1)}}{above}
\ocdx[2]{o}{a3}{60}{T^{(m-1)}}{above}
\ocdx[2]{o}{a4}{20}{U}{right}
\node at (140:\xch) {$x_1$};
\node at (90:\xch) {$\cdots$};
\node at (40:\xch) {$x_m$};
\node[below] at ($(o)!0.7!(a1)$) {$\alpha_0$};
\node[below] at ($(o)!0.7!(a4)$) {$\omega$};
\node[right] at ($(o)!0.75!(a2)$) {$\alpha_1$};
\node[right] at ($(o)!0.75!(a3)$) {$\alpha_{m-1}\!$};
}\in\DTF.
\end{align*}
\noindent {\bf Case 2:} $T^{(m)}\neq |$. Then
\begin{align*}
T\prec'_\omega U&=T\diamond B^+_\omega(U)
=\treeoo{\cdb o\ocdx[2]{o}{a1}{160}{T^{(0)}}{left}
\ocdx[2]{o}{a2}{120}{T^{(1)}}{above}
\ocdx[2]{o}{a3}{60}{T^{(m-1)}}{above}
\ocdx[2]{o}{a4}{20}{T^{(m)}}{right}
\node at (140:\xch) {$x_1$};
\node at (90:\xch) {$\cdots$};
\node at (40:\xch) {$x_m$};
\node[below] at ($(o)!0.7!(a1)$) {$\alpha_0$};
\node[below] at ($(o)!0.7!(a4)$) {$\alpha_{m}$};
\node[right] at ($(o)!0.75!(a2)$) {$\alpha_1$};
\node[right] at ($(o)!0.75!(a3)$) {$\alpha_{m-1}\!$};
}\diamond B^{+}_\omega(U)\\
&=\Big(\treeoo{\cdb o\ocdx[2]{o}{a1}{160}{T^{(0)}}{left}
\ocdx[2]{o}{a2}{120}{T^{(1)}}{above}
\ocdx[2]{o}{a3}{60}{T^{(m-1)}}{above}
\cdx[2]{o}{a4}{20}
\node at (140:\xch) {$x_1$};
\node at (90:\xch) {$\cdots$};
\node at (40:\xch) {$x_m$};
\node[below] at ($(o)!0.7!(a1)$) {$\alpha_0$};
\node[right] at ($(o)!0.75!(a2)$) {$\alpha_1$};
\node[right] at ($(o)!0.75!(a3)$) {$\alpha_{m-1}\!$};
}\diamond B^{+}_{\alpha_m}(T^{(m)})\Big)\diamond B^{+}_{\omega}(U)
\quad(\text{by Eq.~(\mref{eq:def5})})\\
&=\treeoo{\cdb o\ocdx[2]{o}{a1}{160}{T^{(0)}}{left}
\ocdx[2]{o}{a2}{120}{T^{(1)}}{above}
\ocdx[2]{o}{a3}{60}{T^{(m-1)}}{above}
\cdx[2]{o}{a4}{20}
\node at (140:\xch) {$x_1$};
\node at (90:\xch) {$\cdots$};
\node at (40:\xch) {$x_m$};
\node[below] at ($(o)!0.7!(a1)$) {$\alpha_0$};
\node[right] at ($(o)!0.75!(a2)$) {$\alpha_1$};
\node[right] at ($(o)!0.75!(a3)$) {$\alpha_{m-1}\!$};
}\diamond \Big(B^{+}_{\alpha_m}(T^{(m)})\diamond B^{+}_{\omega}(U)\Big)\\
&=\treeoo{\cdb o\ocdx[2]{o}{a1}{160}{T^{(0)}}{left}
\ocdx[2]{o}{a2}{120}{T^{(1)}}{above}
\ocdx[2]{o}{a3}{60}{T^{(m-1)}}{above}
\cdx[2]{o}{a4}{20}
\node at (140:\xch) {$x_1$};
\node at (90:\xch) {$\cdots$};
\node at (40:\xch) {$x_m$};
\node[below] at ($(o)!0.7!(a1)$) {$\alpha_0$};
\node[right] at ($(o)!0.75!(a2)$) {$\alpha_1$};
\node[right] at ($(o)!0.75!(a3)$) {$\alpha_{m-1}\!$};
}\diamond B^{+}_{\alpha_m\omega}\Big(B^{+}_{\alpha_m}(T^{(m)})\diamond U+T^{(m)}\diamond B^{+}_\omega(U)+T^{(m)}\diamond U\Big)\\
&\hspace{3cm}(\text{by Theorem~(\mref{thm:free}) of weight $1$})\\
&=\treeoo{\cdb o\ocdx[2]{o}{a1}{160}{T^{(0)}}{left}
\ocdx[2]{o}{a2}{120}{T^{(1)}}{above}
\ocdx[2]{o}{a3}{60}{T^{(m-1)}}{above}
\ocdx[2]{o}{a4}{20}{(B^{+}_{\alpha_m}(T^{(m)})\diamond U)}{right}
\node at (140:\xch) {$x_1$};
\node at (90:\xch) {$\cdots$};
\node at (40:\xch) {$x_m$};
\node[below] at ($(o)!0.7!(a1)$) {$\alpha_0$};
\node[below] at ($(o)!0.7!(a4)$) {$\alpha_{m}\omega$};
\node[right] at ($(o)!0.75!(a2)$) {$\alpha_1$};
\node[right] at ($(o)!0.75!(a3)$) {$\alpha_{m-1}\!$};
}+\treeoo{\cdb o\ocdx[2]{o}{a1}{160}{T^{(0)}}{left}
\ocdx[2]{o}{a2}{120}{T^{(1)}}{above}
\ocdx[2]{o}{a3}{60}{T^{(m-1)}}{above}
\ocdx[2]{o}{a4}{20}{(T^{(m)}\diamond B^{+}_\omega(U))}{right}
\node at (140:\xch) {$x_1$};
\node at (90:\xch) {$\cdots$};
\node at (40:\xch) {$x_m$};
\node[below] at ($(o)!0.7!(a1)$) {$\alpha_0$};
\node[below] at ($(o)!0.7!(a4)$) {$\alpha_{m}\omega$};
\node[right] at ($(o)!0.75!(a2)$) {$\alpha_1$};
\node[right] at ($(o)!0.75!(a3)$) {$\alpha_{m-1}\!$};
}+\treeoo{\cdb o\ocdx[2]{o}{a1}{160}{T^{(0)}}{left}
\ocdx[2]{o}{a2}{120}{T^{(1)}}{above}
\ocdx[2]{o}{a3}{60}{T^{(m-1)}}{above}
\ocdx[2]{o}{a4}{20}{(T^{(m)}\diamond U)}{right}
\node at (140:\xch) {$x_1$};
\node at (90:\xch) {$\cdots$};
\node at (40:\xch) {$x_m$};
\node[below] at ($(o)!0.7!(a1)$) {$\alpha_0$};
\node[below] at ($(o)!0.7!(a4)$) {$\alpha_{m}\omega$};
\node[right] at ($(o)!0.75!(a2)$) {$\alpha_1$};
\node[right] at ($(o)!0.75!(a3)$) {$\alpha_{m-1}\!$};
}.\\
&\hspace{5cm}(\text{by Eq.~(\mref{eq:def5})})
\end{align*}
By the induction hypothesis, $
B^{+}_{\alpha_m}(T^{(m)})\diamond U, T^{(m)}\diamond B^{+}_\omega(U)$ and $T^{(m)}\diamond U\in\DTF$ and so $T\prec'_\omega U\in\DTF.$

\noindent We next prove $T\succ'_\omega U\in\DTF,$  there are also two cases to consider.

\noindent{\bf Case 3:} $U^{(0)}=|$. Then
\begin{align*}
T\succ'_\omega U&=B^+_\omega(T)\diamond U=B^{+}_\omega(T)\diamond \treeoo{\cdb o\cdx[2]{o}{a1}{160}
\ocdx[2]{o}{a2}{120}{U^{(1)}}{above}
\ocdx[2]{o}{a3}{60}{U^{(n-1)}}{above}
\ocdx[2]{o}{a4}{20}{U^{(n)}}{right}
\node at (140:\xch) {$y_1$};
\node at (90:\xch) {$\cdots$};
\node at (40:\xch) {$y_n$};
\node[below] at ($(o)!0.7!(a4)$) {$\beta_{n}$};
\node[right] at ($(o)!0.75!(a2)$) {$\beta_1$};
\node[right] at ($(o)!0.75!(a3)$) {$\beta_{n-1}\!$};
}
=\treeoo{\cdb o\ocdx[2]{o}{a1}{160}{T}{left}
\ocdx[2]{o}{a2}{120}{U^{(1)}}{above}
\ocdx[2]{o}{a3}{60}{U^{(n-1)}}{above}
\ocdx[2]{o}{a4}{20}{U^{(n)}}{right}
\node at (140:\xch) {$y_1$};
\node at (90:\xch) {$\cdots$};
\node at (40:\xch) {$y_n$};
\node[below] at ($(o)!0.7!(a1)$) {$\omega$};
\node[below] at ($(o)!0.7!(a4)$) {$\beta_{n}$};
\node[right] at ($(o)!0.75!(a2)$) {$\beta_1$};
\node[right] at ($(o)!0.75!(a3)$) {$\beta_{n-1}\!$};
}\in\DTF.
\end{align*}
\noindent {\bf Case 4:} $U^{(0)}\neq |$. Then
\begin{align*}
T\succ'_\omega U&=B^+_\omega(T)\diamond U
=B^{+}_\omega(T)\diamond \treeoo{\cdb o\ocdx[2]{o}{a1}{160}{U^{(0)}}{left}
\ocdx[2]{o}{a2}{120}{U^{(1)}}{above}
\ocdx[2]{o}{a3}{60}{U^{(n-1)}}{above}
\ocdx[2]{o}{a4}{20}{U^{(n)}}{right}
\node at (140:\xch) {$y_1$};
\node at (90:\xch) {$\cdots$};
\node at (40:\xch) {$y_n$};
\node[below] at ($(o)!0.7!(a1)$) {$\beta_0$};
\node[below] at ($(o)!0.7!(a4)$) {$\beta_{n}$};
\node[right] at ($(o)!0.75!(a2)$) {$\beta_1$};
\node[right] at ($(o)!0.75!(a3)$) {$\beta_{n-1}\!$};
}\\
&=B^{+}_\omega(T)\diamond \Big(B^{+}_{\beta_0}(U^{(0)})\diamond \treeoo{\cdb o\cdx[2]{o}{a1}{160}
\ocdx[2]{o}{a2}{120}{U^{(1)}}{above}
\ocdx[2]{o}{a3}{60}{U^{(n-1)}}{above}
\ocdx[2]{o}{a4}{20}{U^{(n)}}{right}
\node at (140:\xch) {$y_1$};
\node at (90:\xch) {$\cdots$};
\node at (40:\xch) {$y_n$};
\node[below] at ($(o)!0.7!(a4)$) {$\beta_{n}$};
\node[right] at ($(o)!0.75!(a2)$) {$\beta_1$};
\node[right] at ($(o)!0.75!(a3)$) {$\beta_{n-1}\!$};
}\Big)\quad(\text{by Eq.~(\mref{eq:def4})})\\
&=\Big(B^{+}_\omega(T)\diamond B^{+}_{\beta_0}(U^{(0)})\Big)\diamond
\treeoo{\cdb o\cdx[2]{o}{a1}{160}
\ocdx[2]{o}{a2}{120}{U^{(1)}}{above}
\ocdx[2]{o}{a3}{60}{U^{(n-1)}}{above}
\ocdx[2]{o}{a4}{20}{U^{(n)}}{right}
\node at (140:\xch) {$y_1$};
\node at (90:\xch) {$\cdots$};
\node at (40:\xch) {$y_n$};
\node[below] at ($(o)!0.7!(a4)$) {$\beta_{n}$};
\node[right] at ($(o)!0.75!(a2)$) {$\beta_1$};
\node[right] at ($(o)!0.75!(a3)$) {$\beta_{n-1}\!$};
}\\
&=B^{+}_{\omega\beta_0}\Big(B^{+}_\omega(T)\diamond U^{(0)}+T\diamond B^{+}_{\beta_0}(U^{(0)}
)+T\diamond U^{(0)}\Big)\diamond \treeoo{\cdb o\cdx[2]{o}{a1}{160}
\ocdx[2]{o}{a2}{120}{U^{(1)}}{above}
\ocdx[2]{o}{a3}{60}{U^{(n-1)}}{above}
\ocdx[2]{o}{a4}{20}{U^{(n)}}{right}
\node at (140:\xch) {$y_1$};
\node at (90:\xch) {$\cdots$};
\node at (40:\xch) {$y_n$};
\node[below] at ($(o)!0.7!(a4)$) {$\beta_{n}$};
\node[right] at ($(o)!0.75!(a2)$) {$\beta_1$};
\node[right] at ($(o)!0.75!(a3)$) {$\beta_{n-1}\!$};
}\\
&\hspace{3cm}(\text{by Theorem~(\mref{thm:free}) of weight $1$})\\
&=\treeoo{\cdb o\ocdx[2]{o}{a1}{160}{\Big(B^{+}_\omega(T)\diamond U^{(0)}+T\diamond B^{+}_{\beta_0}(U^{(0)})+T\diamond U^{(0)}\Big)}{left}
\ocdx[2]{o}{a2}{120}{U^{(1)}}{above}
\ocdx[2]{o}{a3}{60}{U^{(n-1)}}{above}
\ocdx[2]{o}{a4}{20}{U^{(n)}}{right}
\node at (140:\xch) {$y_1$};
\node at (90:\xch) {$\cdots$};
\node at (40:\xch) {$y_n$};
\node[below] at ($(o)!0.7!(a1)$) {$\omega\beta_0$};
\node[below] at ($(o)!0.7!(a4)$) {$\beta_{n}$};
\node[right] at ($(o)!0.75!(a2)$) {$\beta_1$};
\node[right] at ($(o)!0.75!(a3)$) {$\beta_{n-1}\!$};
}\quad(\text{by Eq.~(\mref{eq:def4})}).
\end{align*}
By the induction hypothesis, $B^{+}_\omega(T)\diamond U^{(0)}+T\diamond B^{+}_{\beta_0}(U^{(0)})+T\diamond U^{(0)}\in\DTF$ and so $T\succ'_\omega U\in\DTF.$

\noindent We finally prove $T\cdot' U\in\DTF$, there are four cases to consider.

\noindent{\bf Case 5:} $T^{(m)}=|=U^{(0)}$. Then
\begin{align*}
T\cdot' U=T\diamond U
=\treeoo{\cdb o\ocdx[2]{o}{a1}{160}{T^{(0)}}{left}
\ocdx[2]{o}{a2}{120}{T^{(1)}}{above}
\ocdx[2]{o}{a3}{60}{T^{(m-1)}}{above}
\cdx[2]{o}{a4}{20}
\node at (140:\xch) {$x_1$};
\node at (90:\xch) {$\cdots$};
\node at (40:\xch) {$x_m$};
\node[below] at ($(o)!0.7!(a1)$) {$\alpha_0$};
\node[right] at ($(o)!0.75!(a2)$) {$\alpha_1$};
\node[right] at ($(o)!0.75!(a3)$) {$\alpha_{m-1}\!$};
}\diamond \treeoo{\cdb o\cdx[2]{o}{a1}{160}
\ocdx[2]{o}{a2}{120}{U^{(1)}}{above}
\ocdx[2]{o}{a3}{60}{U^{(n-1)}}{above}
\ocdx[2]{o}{a4}{20}{U^{(n)}}{right}
\node at (140:\xch) {$y_1$};
\node at (90:\xch) {$\cdots$};
\node at (40:\xch) {$y_n$};
\node[below] at ($(o)!0.7!(a4)$) {$\beta_{n}$};
\node[right] at ($(o)!0.75!(a2)$) {$\beta_1$};
\node[right] at ($(o)!0.75!(a3)$) {$\beta_{n-1}\!$};
}
=\treeoo{\cdb o
\ocdx[2.5]{o}{a1}{170}{T^{(0)}}{170}
\ocdx[2.5]{o}{a2}{145}{T^{(1)}}{145}
\ocdx[2.5]{o}{a3}{115}{T^{(m-1)}}{115}
\cdx[2.5]{o}{a4}{90}{-90}
\ocdx[2.5]{o}{a5}{65}{U^{(1)}}{65}
\ocdx[2.5]{o}{a6}{35}{U^{(n-1)}}{35}
\ocdx[2.5]{o}{a7}{10}{U^{(n)}}{10}
\node at (157:1.3*\xch) {$x_1$};
\node[rotate=40] at (129:2.5*\xch) {$\cdots$};
\node at (102:1.5*\xch) {$x_m$};
\node at (75:1.5*\xch) {$y_1$};
\node[rotate=-40] at (49:2.5*\xch) {$\cdots$};
\node at (20:1.5*\xch) {$y_n$};
\node[below] at ($(o)!0.7!(a1)$) {$\alpha_0$};
\node[below] at ($(o)!0.8!(a7)$) {\tiny$\beta_{n}$};
\node[below] at ($(o)!0.9!(a2)$) {\tiny$\alpha_1$};
\node[left=-1pt] at ($(o)!0.65!(a3)$) {\tiny$\alpha_{m-1}$};
\node[right=1pt] at ($(o)!0.65!(a5)$) {\tiny$\beta_1$};
\node[right=1pt] at ($(o)!0.65!(a6)$) {\tiny$\beta_{n-1}$};
}\in\DTF.
\end{align*}
\noindent{\bf Case 6:} $T^{(m)}\neq|=U^{(0)}$. Then
\begin{align*}
T\cdot' U=T\diamond U=\treeoo{\cdb o\ocdx[2]{o}{a1}{160}{T^{(0)}}{left}
\ocdx[2]{o}{a2}{120}{T^{(1)}}{above}
\ocdx[2]{o}{a3}{60}{T^{(m-1)}}{above}
\ocdx[2]{o}{a4}{20}{T^{(m)}}{right}
\node at (140:\xch) {$x_1$};
\node at (90:\xch) {$\cdots$};
\node at (40:\xch) {$x_m$};
\node[below] at ($(o)!0.7!(a1)$) {$\alpha_0$};
\node[below] at ($(o)!0.7!(a4)$) {$\alpha_{m}$};
\node[right] at ($(o)!0.75!(a2)$) {$\alpha_1$};
\node[right] at ($(o)!0.75!(a3)$) {$\alpha_{m-1}\!$};
}\diamond \treeoo{\cdb o\cdx[2]{o}{a1}{160}
\ocdx[2]{o}{a2}{120}{U^{(1)}}{above}
\ocdx[2]{o}{a3}{60}{U^{(n-1)}}{above}
\ocdx[2]{o}{a4}{20}{U^{(n)}}{right}
\node at (140:\xch) {$y_1$};
\node at (90:\xch) {$\cdots$};
\node at (40:\xch) {$y_n$};
\node[below] at ($(o)!0.7!(a4)$) {$\beta_{n}$};
\node[right] at ($(o)!0.75!(a2)$) {$\beta_1$};
\node[right] at ($(o)!0.75!(a3)$) {$\beta_{n-1}\!$};
}
=\treeoo{\cdb o
\ocdx[2.5]{o}{a1}{170}{T^{(0)}}{170}
\ocdx[2.5]{o}{a2}{145}{T^{(1)}}{145}
\ocdx[2.5]{o}{a3}{115}{T^{(m-1)}}{115}
\ocdx[2.5]{o}{a4}{90}{T^{(m)}}{90}
\ocdx[2.5]{o}{a5}{65}{U^{(0)}}{65}
\ocdx[2.5]{o}{a6}{35}{U^{(n-1)}}{35}
\ocdx[2.5]{o}{a7}{10}{U^{(n)}}{10}
\node at (157:1.3*\xch) {$x_1$};
\node[rotate=40] at (129:2.5*\xch) {$\cdots$};
\node at (102:1.5*\xch) {$x_m$};
\node at (75:1.5*\xch) {$y_1$};
\node[rotate=-40] at (49:2.5*\xch) {$\cdots$};
\node at (20:1.5*\xch) {$y_n$};
\node[below] at ($(o)!0.7!(a1)$) {$\alpha_0$};
\node[below] at ($(o)!0.8!(a7)$) {\tiny$\beta_{n}$};
\node[below] at ($(o)!0.9!(a2)$) {\tiny$\alpha_1$};
\node[left=-1pt] at ($(o)!0.65!(a3)$) {\tiny$\alpha_{n-1}$};
\node[right=1pt] at ($(o)!0.65!(a5)$) {\tiny$\beta_1$};
\node[right=1pt] at ($(o)!0.65!(a6)$) {\tiny$\beta_{m-1}$};
\node[right] at ($(o)!0.8!(a4)$) {\tiny$\alpha_{m}$};
}\in\DTF.
\end{align*}
\noindent {\bf Case 7:} $T^{(m)}= |\neq U^{(0)}$. This case is similar to Case 6.

\noindent {\bf Case 8:} $T^{(m)}\neq |\neq U^{(0)}$. Then
{\small{\begin{align*}
T\cdot' U&=T\diamond U
=\treeoo{\cdb o\ocdx[2]{o}{a1}{160}{T^{(0)}}{left}
\ocdx[2]{o}{a2}{120}{T^{(1)}}{above}
\ocdx[2]{o}{a3}{60}{T^{(m-1)}}{above}
\ocdx[2]{o}{a4}{20}{T^{(m)}}{right}
\node at (140:\xch) {$x_1$};
\node at (90:\xch) {$\cdots$};
\node at (40:\xch) {$x_m$};
\node[below] at ($(o)!0.7!(a1)$) {$\alpha_0$};
\node[below] at ($(o)!0.7!(a4)$) {$\alpha_{m}$};
\node[right] at ($(o)!0.75!(a2)$) {$\alpha_1$};
\node[right] at ($(o)!0.75!(a3)$) {$\alpha_{m-1}\!$};
}
\diamond \treeoo{\cdb o\ocdx[2]{o}{a1}{160}{U^{(0)}}{left}
\ocdx[2]{o}{a2}{120}{U^{(1)}}{above}
\ocdx[2]{o}{a3}{60}{U^{(n-1)}}{above}
\ocdx[2]{o}{a4}{20}{U^{(n)}}{right}
\node at (140:\xch) {$y_1$};
\node at (90:\xch) {$\cdots$};
\node at (40:\xch) {$y_n$};
\node[below] at ($(o)!0.7!(a1)$) {$\beta_0$};
\node[below] at ($(o)!0.7!(a4)$) {$\beta_{n}$};
\node[right] at ($(o)!0.75!(a2)$) {$\beta_1$};
\node[right] at ($(o)!0.75!(a3)$) {$\beta_{n-1}\!$};
}\\
&=\Biggl(\treeoo{\cdb o\ocdx[2]{o}{a1}{160}{T^{(0)}}{left}
\ocdx[2]{o}{a2}{120}{T^{(1)}}{above}
\ocdx[2]{o}{a3}{60}{T^{(m-1)}}{above}
\cdx[2]{o}{a4}{20}
\node at (140:\xch) {$x_1$};
\node at (90:\xch) {$\cdots$};
\node at (40:\xch) {$x_m$};
\node[below] at ($(o)!0.7!(a1)$) {$\alpha_0$};
\node[right] at ($(o)!0.75!(a2)$) {$\alpha_1$};
\node[right] at ($(o)!0.75!(a3)$) {$\alpha_{m-1}\!$};
}
\diamond \Big(B^{+}_{\alpha_m}(T^{(m)})\diamond B^{+}_{\beta_0}(U^{(0)})\Big)\Biggl)\diamond \treeoo{\cdb o\cdx[2]{o}{a1}{160}
\ocdx[2]{o}{a2}{120}{U^{(1)}}{above}
\ocdx[2]{o}{a3}{60}{U^{(n-1)}}{above}
\ocdx[2]{o}{a4}{20}{U^{(n)}}{right}
\node at (140:\xch) {$y_1$};
\node at (90:\xch) {$\cdots$};
\node at (40:\xch) {$y_n$};
\node[below] at ($(o)!0.7!(a4)$) {$\beta_{n}$};
\node[right] at ($(o)!0.75!(a2)$) {$\beta_1$};
\node[right] at ($(o)!0.75!(a3)$) {$\beta_{n-1}\!$};
}\quad(\text{by Eq.~(\mref{eq:def6})})\\
&=\Biggl(\treeoo{\cdb o\ocdx[2]{o}{a1}{160}{T^{(0)}}{left}
\ocdx[2]{o}{a2}{120}{T^{(1)}}{above}
\ocdx[2]{o}{a3}{60}{T^{(m-1)}}{above}
\cdx[2]{o}{a4}{20}
\node at (140:\xch) {$x_1$};
\node at (90:\xch) {$\cdots$};
\node at (40:\xch) {$x_m$};
\node[below] at ($(o)!0.7!(a1)$) {$\alpha_0$};
\node[right] at ($(o)!0.75!(a2)$) {$\alpha_1$};
\node[right] at ($(o)!0.75!(a3)$) {$\alpha_{m-1}\!$};
}\diamond B^{+}_{\alpha_m\beta_0}\Big(B^{+}_{\alpha_m}(T^{(m)})\diamond U^{(0)}+T^{(m)}\diamond B^{+}_{\beta_0}(U^{(0)})+T^{(m)}\diamond U^{(0)}\Big)\Biggl)\diamond
\treeoo{\cdb o\cdx[2]{o}{a1}{160}
\ocdx[2]{o}{a2}{120}{U^{(1)}}{above}
\ocdx[2]{o}{a3}{60}{U^{(n-1)}}{above}
\ocdx[2]{o}{a4}{20}{U^{(n)}}{right}
\node at (140:\xch) {$y_1$};
\node at (90:\xch) {$\cdots$};
\node at (40:\xch) {$y_n$};
\node[below] at ($(o)!0.7!(a4)$) {$\beta_{n}$};
\node[right] at ($(o)!0.75!(a2)$) {$\beta_1$};
\node[right] at ($(o)!0.75!(a3)$) {$\beta_{n-1}\!$};
}\\
&\hspace{3cm}\text{(by Theorem ~(\mref{thm:free}) of weight $1$)}\\
&=\treeoo{\cdb o
\ocdx[2.5]{o}{a1}{170}{T^{(0)}}{170}
\ocdx[2.5]{o}{a2}{145}{T^{(1)}}{145}
\ocdx[2.5]{o}{a3}{115}{T^{(m-1)}}{115}
\ocdx[4.0]{o}{a4}{90}{B^{+}_{\alpha_m}(T^{(m)})\diamond U^{(0)}}{90}
\ocdx[2.5]{o}{a5}{65}{U^{(1)}}{65}
\ocdx[2.5]{o}{a6}{35}{U^{(n-1)}}{35}
\ocdx[2.5]{o}{a7}{10}{U^{(n)}}{10}
\node at (157:1.3*\xch) {$x_1$};
\node[rotate=40] at (129:2.5*\xch) {$\cdots$};
\node at (102:1.5*\xch) {$x_m$};
\node at (75:1.5*\xch) {$y_1$};
\node[rotate=-40] at (49:2.5*\xch) {$\cdots$};
\node at (20:1.5*\xch) {$y_n$};
\node[below] at ($(o)!0.7!(a1)$) {$\alpha_0$};
\node[below] at ($(o)!0.8!(a7)$) {\tiny$\beta_{n}$};
\node[below] at ($(o)!0.9!(a2)$) {\tiny$\alpha_1$};
\node[left=-1pt] at ($(o)!0.65!(a3)$) {\tiny$\alpha_{m-1}$};
\node[right=1pt] at ($(o)!0.65!(a5)$) {\tiny$\beta_1$};
\node[right=1pt] at ($(o)!0.65!(a6)$) {\tiny$\beta_{m-1}$};
\node[right] at ($(o)!0.9!(a4)$) {\tiny$\alpha_{m}\beta_0$};
}+\treeoo{\cdb o
\ocdx[2.5]{o}{a1}{170}{T^{(0)}}{170}
\ocdx[2.5]{o}{a2}{145}{T^{(1)}}{145}
\ocdx[2.5]{o}{a3}{115}{T^{(m-1)}}{115}
\ocdx[4.0]{o}{a4}{90}{T^{(m)}\diamond B^{+}_{\beta_0}(U^{(0)})}{90}
\ocdx[2.5]{o}{a5}{65}{U^{(1)}}{65}
\ocdx[2.5]{o}{a6}{35}{U^{(n-1)}}{35}
\ocdx[2.5]{o}{a7}{10}{U^{(n)}}{10}
\node at (157:1.3*\xch) {$x_1$};
\node[rotate=40] at (129:2.5*\xch) {$\cdots$};
\node at (102:1.5*\xch) {$x_m$};
\node at (75:1.5*\xch) {$y_1$};
\node[rotate=-40] at (49:2.5*\xch) {$\cdots$};
\node at (20:1.5*\xch) {$y_n$};
\node[below] at ($(o)!0.7!(a1)$) {$\alpha_0$};
\node[below] at ($(o)!0.8!(a7)$) {\tiny$\beta_{n}$};
\node[below] at ($(o)!0.9!(a2)$) {\tiny$\alpha_1$};
\node[left=-1pt] at ($(o)!0.65!(a3)$) {\tiny$\alpha_{m-1}$};
\node[right=1pt] at ($(o)!0.65!(a5)$) {\tiny$\beta_1$};
\node[right=1pt] at ($(o)!0.65!(a6)$) {\tiny$\beta_{m-1}$};
\node[right] at ($(o)!0.9!(a4)$) {\tiny$\alpha_{m}\beta_0$};
}+\treeoo{\cdb o
\ocdx[2.5]{o}{a1}{170}{T^{(0)}}{170}
\ocdx[2.5]{o}{a2}{145}{T^{(1)}}{145}
\ocdx[2.5]{o}{a3}{115}{T^{(m-1)}}{115}
\ocdx[4.0]{o}{a4}{90}{T^{(m)}\diamond U^{(0)}}{90}
\ocdx[2.5]{o}{a5}{65}{U^{(1)}}{65}
\ocdx[2.5]{o}{a6}{35}{U^{(n-1)}}{35}
\ocdx[2.5]{o}{a7}{10}{U^{(n)}}{10}
\node at (157:1.3*\xch) {$x_1$};
\node[rotate=40] at (129:2.5*\xch) {$\cdots$};
\node at (102:1.5*\xch) {$x_m$};
\node at (75:1.5*\xch) {$y_1$};
\node[rotate=-40] at (49:2.5*\xch) {$\cdots$};
\node at (20:1.5*\xch) {$y_n$};
\node[below] at ($(o)!0.7!(a1)$) {$\alpha_0$};
\node[below] at ($(o)!0.8!(a7)$) {\tiny$\beta_{n}$};
\node[below] at ($(o)!0.9!(a2)$) {\tiny$\alpha_1$};
\node[left=-1pt] at ($(o)!0.65!(a3)$) {\tiny$\alpha_{m-1}$};
\node[right=1pt] at ($(o)!0.65!(a5)$) {\tiny$\beta_1$};
\node[right=1pt] at ($(o)!0.65!(a6)$) {\tiny$\beta_{m-1}$};
\node[right] at ($(o)!0.9!(a4)$) {\tiny$\alpha_{m}\beta_0$};
}\\
&\hspace{3cm}\quad(\text{by Eqs.~(\mref{eq:def4}) and ~(\mref{eq:def5})}).
\end{align*}}}
By the induction hypothesis, $B^{+}_{\alpha_m}(T^{(m)})\diamond U^{(0)}, T^{(m)}\diamond B^{+}_{\beta_0}(U^{(0)})$ and $T^{(m)}\diamond U^{(0)}\in \DTF$ and so $T\cdot' U\in\DTF.$
We are left to show
\begin{equation}
 T\prec_\omega U=T\prec'_\omega U,\,T\succ_\omega U=T\succ'_\omega U\text{ and }\,T\cdot U=T\cdot' U,
\mlabel{eq:cont}
\end{equation}
for
\begin{align*}
T=& \ \bigv x_1,\ldots,x_m;m+1;\alpha_0,\ldots,\alpha_m;(T^{(0)},\ldots, T^{(m)})\in T_{m},\\
U=& \ \bigv y_1,\ldots,y_n;n+1;\beta_0,\ldots,\beta_n;(U^{(0)},\ldots, U^{(n)})\in T_{n}\,\text{ with }\,m,n\geq 1.
\end{align*}
We apply induction on $\dep(T)+\dep(U)\geq 2.$
For the initial step  $\dep(T)+\dep(U)=2$, we have
$T=\treeoo{\cdb o\cdx[1.5]{o}{a1}{160}
\cdx[1.5]{o}{a2}{120}
\cdx[1.5]{o}{a3}{60}
\cdx[1.5]{o}{a4}{20}
\node at (140:1.2*\xch) {$x_1$};
\node at (90:1.2*\xch) {$\cdots$};
\node at (40:1.2*\xch) {$x_m$};
}$ and $U=\treeoo{\cdb o\cdx[1.5]{o}{a1}{160}
\cdx[1.5]{o}{a2}{120}
\cdx[1.5]{o}{a3}{60}
\cdx[1.5]{o}{a4}{20}
\node at (140:1.2*\xch) {$y_1$};
\node at (90:1.2*\xch) {$\cdots$};
\node at (40:1.2*\xch) {$y_n$};
}$. Then
\begin{align*}
T\prec_\omega U&=\bigv x_1,\ldots,x_{m-1},x_m;m+1;1,\ldots,1,\omega; (|,\ldots,|,|\succ_1 U+|\prec_\omega U+|\cdot U)\quad(\text{by Eq.~(\mref{eq:tdpre})})\\
&=\bigv x_1,\ldots,x_{m-1},x_m;m+1;1,\ldots,1,\omega; (|,\ldots,|,U)\quad(\text{by Item~\mref{it:trida} in Definition~\mref{defn:trirec} and Eq.~(\mref{eq:con})})\\
&=\treeoo{\cdb o\cdx[1.5]{o}{a1}{160}
\cdx[1.5]{o}{a2}{120}
\cdx[1.5]{o}{a3}{60}
\ocdx[1.5]{o}{a4}{20}{U}{20}
\node[below] at ($(o)!0.7!(a4)$) {$\omega$};
\node at (140:1.2*\xch) {$x_1$};
\node at (90:1.2*\xch) {$\cdots$};
\node at (40:1.2*\xch) {$x_m$};
}
=\treeoo{\cdb o\cdx[1.5]{o}{a1}{160}
\cdx[1.5]{o}{a2}{120}
\cdx[1.5]{o}{a3}{60}
\cdx[1.5]{o}{a4}{20}
\node at (140:1.2*\xch) {$x_1$};
\node at (90:1.2*\xch) {$\cdots$};
\node at (40:1.2*\xch) {$x_m$};
}\diamond B^{+}_\omega(U)\quad(\text{by Eq.~(\mref{eq:def5})})\\
&=T\diamond B^{+}_\omega(U)=T\prec'_\omega U\quad(\text{by Eq.~(\mref{eq:trid})}),
\end{align*}
and
\begin{align*}
T\succ_\omega U&=\bigv y_1,y_2,\ldots,y_n;n+1;\omega,1,\ldots,1;  (T\succ_\omega |+T\prec_{1} |+T\cdot |,|,\ldots,|)
\quad(\text{by Eq.~(\mref{eq:tdsuc})})\\
&=\bigv y_1,y_2,\ldots,y_n;n+1;\omega,1,\ldots,1;  (T,|,\ldots,|)\quad(\text{by Item~\mref{it:trida} in Definition~\mref{defn:trirec} and Eq.~(\mref{eq:con})})\\
&=\treeoo{\cdb o\ocdx[1.5]{o}{a1}{160}{T}{160}
\cdx[1.5]{o}{a2}{120}
\cdx[1.5]{o}{a3}{60}
\cdx[1.5]{o}{a4}{20}
\node[below] at ($(o)!0.7!(a1)$) {$\omega$};
\node at (140:1.2*\xch) {$y_1$};
\node at (90:1.2*\xch) {$\cdots$};
\node at (40:1.2*\xch) {$y_n$};
}
=B^{+}_\omega(T)\diamond \treeoo{\cdb o\cdx[1.5]{o}{a1}{160}
\cdx[1.5]{o}{a2}{120}
\cdx[1.5]{o}{a3}{60}
\cdx[1.5]{o}{a4}{20}
\node at (140:1.2*\xch) {$y_1$};
\node at (90:1.2*\xch) {$\cdots$};
\node at (40:1.2*\xch) {$y_n$};
}\quad(\text{by Eq.~(\mref{eq:def4})})\\
&=B^{+}_\omega(T)\diamond U=T\succ'_\omega U\quad(\text{by Eq.~(\mref{eq:trid})}),
\end{align*}
and
\begin{align*}
T\cdot U&=\bigv x_1,\ldots,x_{m-1},x_m,y_1,\ldots,y_n;m+n+1;1,\ldots,1,1,1,
 \ldots,1;
 (|,\ldots,|,|\succ_1 |+|\prec_1 |+|\cdot |, |,\ldots,|)\quad(\text{by Eq.~(\mref{eq:tdcdot})})\\
 &=\bigv x_1,\ldots,x_{m-1},x_m,y_1,\ldots,y_n;m+n+1;1,\ldots,1,1,1,
 \ldots,1;
 (|,\ldots,|,|, |,\ldots,|)\quad(\text{by Eq.~(\mref{eq:1star1})})\\
 &=\treeoo{\cdb o
\cdx[2.2]{o}{a1}{170}
\cdx[2.2]{o}{a2}{145}
\cdx[2.2]{o}{a3}{115}
\cdx[2.2]{o}{a4}{90}
\cdx[2.2]{o}{a5}{65}
\cdx[2.2]{o}{a6}{35}
\cdx[2.2]{o}{a7}{10}
\node at ($(a1)!0.5!(a2)$) {$x_1$};
\node[rotate=40] at ($(a2)!0.5!(a3)$) {$\cdots$};
\node at ($(a3)!0.5!(a4)$) {$x_m$};
\node at ($(a4)!0.5!(a5)$) {$y_1$};
\node[rotate=-40] at ($(a5)!0.5!(a6)$) {$\cdots$};
\node at ($(a6)!0.5!(a7)$) {$y_n$};
}=\treeoo{\cdb o\cdx[1.5]{o}{a1}{160}
\cdx[1.5]{o}{a2}{120}
\cdx[1.5]{o}{a3}{60}
\cdx[1.5]{o}{a4}{20}
\node at (140:1.2*\xch) {$x_1$};
\node at (90:1.2*\xch) {$\cdots$};
\node at (40:1.2*\xch) {$x_m$};
}\diamond \treeoo{\cdb o\cdx[1.5]{o}{a1}{160}
\cdx[1.5]{o}{a2}{120}
\cdx[1.5]{o}{a3}{60}
\cdx[1.5]{o}{a4}{20}
\node at (140:1.2*\xch) {$y_1$};
\node at (90:1.2*\xch) {$\cdots$};
\node at (40:1.2*\xch) {$y_n$};
}\quad(\text{by Eq.~(\mref{eq:def1})})\\
&=T\diamond U=T\cdot' U\quad(\text{by Eq.~(\mref{eq:trid})}).
\end{align*}
For the induction step  $\dep(T)+\dep(U)\geq 3$, we first prove the first statement of Eq.~(\mref{eq:cont}). There are two cases to consider.

\noindent {\bf Case 1:} $T^{(m)}=|$. Then
\begin{align*}
T\prec_\omega U&=\bigv x_1,\ldots,x_{m-1},x_m;m+1;\alpha_0,\ldots,\alpha_{m-1},\omega; (T^{(0)},\ldots,T^{(m-1)},|\succ_1 U+|\prec_\omega U+|\cdot U)\quad(\text{by Eq.~(\mref{eq:tdpre})})\\
&=\bigv x_1,\ldots,x_{m-1},x_m;m+1;\alpha_0,\ldots,\alpha_{m-1},\omega; (T^{(0)},\ldots,T^{(m-1)},U)\quad(\text{by Item~\mref{it:trida} in Definition~\mref{defn:trirec} and Eq.~(\mref{eq:con})})\\
&=\treeoo{\cdb o\ocdx[2]{o}{a1}{160}{T^{(0)}}{left}
\ocdx[2]{o}{a2}{120}{T^{(1)}}{above}
\ocdx[2]{o}{a3}{60}{T^{(m-1)}}{above}
\ocdx[2]{o}{a4}{20}{U}{right}
\node at (140:\xch) {$x_1$};
\node at (90:\xch) {$\cdots$};
\node at (40:\xch) {$x_m$};
\node[below] at ($(o)!0.7!(a1)$) {$\alpha_0$};
\node[below] at ($(o)!0.7!(a4)$) {$\omega$};
\node[right] at ($(o)!0.75!(a2)$) {$\alpha_1$};
\node[right] at ($(o)!0.75!(a3)$) {$\alpha_{m-1}\!$};
}
=\treeoo{\cdb o\ocdx[2]{o}{a1}{160}{T^{(0)}}{left}
\ocdx[2]{o}{a2}{120}{T^{(1)}}{above}
\ocdx[2]{o}{a3}{60}{T^{(m-1)}}{above}
\cdx[2]{o}{a4}{20}
\node at (140:\xch) {$x_1$};
\node at (90:\xch) {$\cdots$};
\node at (40:\xch) {$x_m$};
\node[below] at ($(o)!0.7!(a1)$) {$\alpha_0$};
\node[right] at ($(o)!0.75!(a2)$) {$\alpha_1$};
\node[right] at ($(o)!0.75!(a3)$) {$\alpha_{m-1}\!$};
}\diamond B^{+}_\omega(U)\quad(\text{by Eq.~(\mref{eq:def5})})\\
&=T\diamond B^{+}_\omega(U)=T\prec'_\omega U\quad(\text{by Eq.~(\mref{eq:trid})}).
\end{align*}
\noindent {\bf Case 2:} $T^{(m)}\neq|$. Then
\begin{align*}
T\prec_\omega U&=\bigv x_1,\ldots,x_{m-1},x_m;m+1;\alpha_0,\ldots,\alpha_{m-1},\alpha_m\omega; (T^{(0)},\ldots,T^{(m-1)},T^{(m)}\succ_{\alpha_m} U+T^{(m)}\prec_\omega U+T^{(m)}\cdot U)\quad(\text{by Eq.~(\mref{eq:tdpre})})\\
&=\bigv x_1,\ldots,x_{m-1},x_m;m+1;\alpha_0,\ldots,\alpha_{m-1},\alpha_m\omega; (T^{(0)},\ldots,T^{(m-1)},T^{(m)}\succ'_{\alpha_m} U+T^{(m)}\prec'_\omega U+T^{(m)}\cdot' U)\\
&\hspace{5cm}(\text{by the induction hypothesis})\\
&=\bigv x_1,\ldots,x_{m-1},x_m;m+1;\alpha_0,\ldots,\alpha_{m-1},\alpha_m\omega; (T^{(0)},\ldots,T^{(m-1)},B^{+}_{\alpha_m}(T^{(m)})\diamond U+T^{(m)}\diamond B^{+}_\omega(U)+T^{(m)}\diamond U)\\
&\hspace{5cm}(\text{by Eq.~(\mref{eq:trid})})\\
&=\treeoo{\cdb o\ocdx[2]{o}{a1}{160}{T^{(0)}}{left}
\ocdx[2]{o}{a2}{120}{T^{(1)}}{above}
\ocdx[2]{o}{a3}{60}{T^{(m-1)}}{above}
\ocdx[2]{o}{a4}{20}{\Big(B^{+}_{\alpha_m}(T^{(m)})\diamond U+T^{(m)}\diamond B^{+}_\omega(U)+T^{(m)}\diamond U\Big)}{right}
\node at (140:\xch) {$x_1$};
\node at (90:\xch) {$\cdots$};
\node at (40:\xch) {$x_m$};
\node[below] at ($(o)!0.7!(a1)$) {$\alpha_0$};
\node[below] at ($(o)!0.7!(a4)$) {$\alpha_{m}\omega$};
\node[right] at ($(o)!0.75!(a2)$) {$\alpha_1$};
\node[right] at ($(o)!0.75!(a3)$) {$\alpha_{m-1}\!$};
}\\
&=\treeoo{\cdb o\ocdx[2]{o}{a1}{160}{T^{(0)}}{left}
\ocdx[2]{o}{a2}{120}{T^{(1)}}{above}
\ocdx[2]{o}{a3}{60}{T^{(m-1)}}{above}
\cdx[2]{o}{a4}{20}
\node at (140:\xch) {$x_1$};
\node at (90:\xch) {$\cdots$};
\node at (40:\xch) {$x_m$};
\node[below] at ($(o)!0.7!(a1)$) {$\alpha_0$};
\node[right] at ($(o)!0.75!(a2)$) {$\alpha_1$};
\node[right] at ($(o)!0.75!(a3)$) {$\alpha_{m-1}\!$};
}\diamond B^{+}_{\alpha_m\omega}\Big(B^{+}_{\alpha_m}(T^{(m)})\diamond U+T^{(m)}\diamond B^{+}_\omega(U)+T^{(m)}\diamond U\Big)\quad(\text{by Eq.~(\mref{eq:def5})})\\
&=\treeoo{\cdb o\ocdx[2]{o}{a1}{160}{T^{(0)}}{left}
\ocdx[2]{o}{a2}{120}{T^{(1)}}{above}
\ocdx[2]{o}{a3}{60}{T^{(m-1)}}{above}
\cdx[2]{o}{a4}{20}
\node at (140:\xch) {$x_1$};
\node at (90:\xch) {$\cdots$};
\node at (40:\xch) {$x_m$};
\node[below] at ($(o)!0.7!(a1)$) {$\alpha_0$};
\node[right] at ($(o)!0.75!(a2)$) {$\alpha_1$};
\node[right] at ($(o)!0.75!(a3)$) {$\alpha_{m-1}\!$};
}\diamond \Big(B^{+}_{\alpha_m}(T^{(m)})\diamond B^{+}_{\omega}(U)\Big)\quad(\text{by Theorem~\mref{thm:free} of weight $1$})\\
&=\Big(\treeoo{\cdb o\ocdx[2]{o}{a1}{160}{T^{(0)}}{left}
\ocdx[2]{o}{a2}{120}{T^{(1)}}{above}
\ocdx[2]{o}{a3}{60}{T^{(m-1)}}{above}
\cdx[2]{o}{a4}{20}
\node at (140:\xch) {$x_1$};
\node at (90:\xch) {$\cdots$};
\node at (40:\xch) {$x_m$};
\node[below] at ($(o)!0.7!(a1)$) {$\alpha_0$};
\node[right] at ($(o)!0.75!(a2)$) {$\alpha_1$};
\node[right] at ($(o)!0.75!(a3)$) {$\alpha_{m-1}\!$};
}\diamond B^{+}_{\alpha_m}(T^{(m)})\Big)\diamond B^{+}_{\omega}(U)\\
&=\treeoo{\cdb o\ocdx[2]{o}{a1}{160}{T^{(0)}}{left}
\ocdx[2]{o}{a2}{120}{T^{(1)}}{above}
\ocdx[2]{o}{a3}{60}{T^{(m-1)}}{above}
\ocdx[2]{o}{a4}{20}{T^{(m)}}{right}
\node at (140:\xch) {$x_1$};
\node at (90:\xch) {$\cdots$};
\node at (40:\xch) {$x_m$};
\node[below] at ($(o)!0.7!(a1)$) {$\alpha_0$};
\node[below] at ($(o)!0.7!(a4)$) {$\alpha_{m}$};
\node[right] at ($(o)!0.75!(a2)$) {$\alpha_1$};
\node[right] at ($(o)!0.75!(a3)$) {$\alpha_{m-1}\!$};
}\diamond B^{+}_\omega(U)\quad(\text{by Eq.~(\mref{eq:def5})})\\
&=T\diamond B^{+}_\omega(U)=T\prec'_\omega U\quad(\text{by Eq.~(\mref{eq:trid})}).
\end{align*}
We next prove the second statement of Eq.~(\mref{eq:cont}), and there are also two cases to consider.

\noindent
{\bf Case 3:} $U^{(0)}=|.$ Then
\begin{align*}
T\succ_\omega U&=\bigv y_1,y_2,\ldots,y_n;n+1;\omega,\,\beta_1,\ldots,\beta_n;  (T\succ_\omega |+T\prec_1 |+T\cdot |,U^{(1)},\ldots,U^{(n)})\quad(\text{by Eq.~(\mref{eq:tdsuc})})\\
&=\bigv y_1,y_2,\ldots,y_n;n+1;\omega,\,\beta_1,\ldots,\beta_n;  (T,U^{(1)},\ldots,U^{(n)})\quad(\text{by Item~\mref{it:trida} in Definition~\mref{defn:trirec} and Eq.~(\mref{eq:con})})\\
&=\treeoo{\cdb o\ocdx[2]{o}{a1}{160}{T}{left}
\ocdx[2]{o}{a2}{120}{U^{(1)}}{above}
\ocdx[2]{o}{a3}{60}{U^{(n-1)}}{above}
\ocdx[2]{o}{a4}{20}{U^{(n)}}{right}
\node at (140:\xch) {$y_1$};
\node at (90:\xch) {$\cdots$};
\node at (40:\xch) {$y_n$};
\node[below] at ($(o)!0.7!(a1)$) {$\omega$};
\node[below] at ($(o)!0.7!(a4)$) {$\beta_{n}$};
\node[right] at ($(o)!0.75!(a2)$) {$\beta_1$};
\node[right] at ($(o)!0.75!(a3)$) {$\beta_{n-1}\!$};
}
=B^{+}_\omega(T)\diamond \treeoo{\cdb o\cdx[2]{o}{a1}{160}
\ocdx[2]{o}{a2}{120}{U^{(1)}}{above}
\ocdx[2]{o}{a3}{60}{U^{(n-1)}}{above}
\ocdx[2]{o}{a4}{20}{U^{(n)}}{right}
\node at (140:\xch) {$y_1$};
\node at (90:\xch) {$\cdots$};
\node at (40:\xch) {$y_n$};
\node[below] at ($(o)!0.7!(a4)$) {$\beta_{n}$};
\node[right] at ($(o)!0.75!(a2)$) {$\beta_1$};
\node[right] at ($(o)!0.75!(a3)$) {$\beta_{n-1}\!$};
}\quad(\text{by Eq.~(\mref{eq:def4})})\\
&=B^{+}_\omega(T)\diamond U=T\succ'_\omega U\quad(\text{by Eq.~(\mref{eq:trid})}).
\end{align*}
\noindent {\bf Case 4:} $U^{(0)}\neq |$. Then
\begin{align*}
T\succ_\omega U&=\bigv y_1,y_2,\ldots,y_n;n+1;\omega\beta_0,\,\beta_1,\ldots,\beta_n;  (T\succ_\omega U^{(0)}+T\prec_{\beta_0} U^{(0)}+T\cdot U^{(0)},U^{(1)},\ldots,U^{(n)})\quad(\text{by Eq.~(\mref{eq:tdsuc})})\\
&=\bigv y_1,y_2,\ldots,y_n;n+1;\omega\beta_0,\,\beta_1,\ldots,\beta_n;  (T\succ'_\omega U^{(0)}+T\prec'_{\beta_0} U^{(0)}+T\cdot' U^{(0)},U^{(1)},\ldots,U^{(n)})\\
 &\hspace{5cm}(\text{by the induction hypothesis})\\
 &=\bigv y_1,y_2,\ldots,y_n;n+1;\omega\beta_0,\,\beta_1,\ldots,\beta_n;  (B^{+}_\omega(T)\diamond U^{(0)}+T\diamond B^{+}_{\beta_0}(U^{(0)})+T\diamond U^{(0)},U^{(1)},\ldots,U^{(n)})\\
 &\hspace{5cm}(\text{by Eq.~(\mref{eq:trid})})\\
&=\treeoo{\cdb o\ocdx[2]{o}{a1}{160}{\Big(B^{+}_\omega(T)\diamond U^{(0)}+T\diamond B^{+}_{\beta_0}(U^{(0)})+T\diamond U^{(0)}\Big)}{left}
\ocdx[2]{o}{a2}{120}{U^{(1)}}{above}
\ocdx[2]{o}{a3}{60}{U^{(n-1)}}{above}
\ocdx[2]{o}{a4}{20}{U^{(n)}}{right}
\node at (140:\xch) {$y_1$};
\node at (90:\xch) {$\cdots$};
\node at (40:\xch) {$y_n$};
\node[below] at ($(o)!0.7!(a1)$) {$\omega\beta_0$};
\node[below] at ($(o)!0.7!(a4)$) {$\beta_{n}$};
\node[right] at ($(o)!0.75!(a2)$) {$\beta_1$};
\node[right] at ($(o)!0.75!(a3)$) {$\beta_{n-1}\!$};
}\\
&=B^{+}_{\omega\beta_0}\Big(B^{+}_\omega(T)\diamond U^{(0)}+T\diamond B^{+}_{\beta_0}(U^{(0)}
)+T\diamond U^{(0)}\Big)\diamond \treeoo{\cdb o\cdx[2]{o}{a1}{160}
\ocdx[2]{o}{a2}{120}{U^{(1)}}{above}
\ocdx[2]{o}{a3}{60}{U^{(n-1)}}{above}
\ocdx[2]{o}{a4}{20}{U^{(n)}}{right}
\node at (140:\xch) {$y_1$};
\node at (90:\xch) {$\cdots$};
\node at (40:\xch) {$y_n$};
\node[below] at ($(o)!0.7!(a4)$) {$\beta_{n}$};
\node[right] at ($(o)!0.75!(a2)$) {$\beta_1$};
\node[right] at ($(o)!0.75!(a3)$) {$\beta_{n-1}\!$};
}\quad(\text{by Eq.~(\mref{eq:def4})})\\
&=\Big(B^{+}_\omega(T)\diamond B^{+}_{\beta_0}(U^{(0)})\Big)\diamond
\treeoo{\cdb o\cdx[2]{o}{a1}{160}
\ocdx[2]{o}{a2}{120}{U^{(1)}}{above}
\ocdx[2]{o}{a3}{60}{U^{(n-1)}}{above}
\ocdx[2]{o}{a4}{20}{U^{(n)}}{right}
\node at (140:\xch) {$y_1$};
\node at (90:\xch) {$\cdots$};
\node at (40:\xch) {$y_n$};
\node[below] at ($(o)!0.7!(a4)$) {$\beta_{n}$};
\node[right] at ($(o)!0.75!(a2)$) {$\beta_1$};
\node[right] at ($(o)!0.75!(a3)$) {$\beta_{n-1}\!$};
}\quad(\text{by Theorem~\mref{thm:free} of weight $1$})\\
&=B^{+}_\omega(T)\diamond \Big(B^{+}_{\beta_0}(U^{(0)})\diamond \treeoo{\cdb o\cdx[2]{o}{a1}{160}
\ocdx[2]{o}{a2}{120}{U^{(1)}}{above}
\ocdx[2]{o}{a3}{60}{U^{(n-1)}}{above}
\ocdx[2]{o}{a4}{20}{U^{(n)}}{right}
\node at (140:\xch) {$y_1$};
\node at (90:\xch) {$\cdots$};
\node at (40:\xch) {$y_n$};
\node[below] at ($(o)!0.7!(a4)$) {$\beta_{n}$};
\node[right] at ($(o)!0.75!(a2)$) {$\beta_1$};
\node[right] at ($(o)!0.75!(a3)$) {$\beta_{n-1}\!$};
}\Big)\\
&=B^{+}_\omega(T)\diamond \treeoo{\cdb o\ocdx[2]{o}{a1}{160}{U^{(0)}}{left}
\ocdx[2]{o}{a2}{120}{U^{(1)}}{above}
\ocdx[2]{o}{a3}{60}{U^{(n-1)}}{above}
\ocdx[2]{o}{a4}{20}{U^{(n)}}{right}
\node at (140:\xch) {$y_1$};
\node at (90:\xch) {$\cdots$};
\node at (40:\xch) {$y_n$};
\node[below] at ($(o)!0.7!(a1)$) {$\beta_0$};
\node[below] at ($(o)!0.7!(a4)$) {$\beta_{n}$};
\node[right] at ($(o)!0.75!(a2)$) {$\beta_1$};
\node[right] at ($(o)!0.75!(a3)$) {$\beta_{n-1}\!$};
}\quad(\text{by Eq.~(\mref{eq:def4})})\\
&=B^{+}_\omega(T)\diamond U=T\succ'_\omega U\quad(\text{by Eq.~(\mref{eq:trid})}).
\end{align*}
We finally prove the third statement of Eq.~(\mref{eq:cont}), and there are four cases to consider.

\noindent {\bf Case 5:} $T^{(m)}=|=U^{(0)}$. Then
\begin{align*}
T\cdot U&=\bigv x_1,\ldots,x_{m-1},x_m,y_1,\ldots,y_n;m+n+1;\alpha_0,\ldots,\alpha_{m-1},1,\,\beta_1,
 \ldots,\beta_n;
 (T^{(0)},\ldots,T^{(m-1)},|\succ_1 |+|\prec_1 |+|\cdot |, U^{(1)},\ldots,U^{(n)})\\
 &\hspace{3cm}(\text{by Eq.~(\mref{eq:tdcdot})})\\
 &=\bigv x_1,\ldots,x_{m-1},x_m,y_1,\ldots,y_n;m+n+1;\alpha_0,\ldots,\alpha_{m-1},1,\,\beta_1,
 \ldots,\beta_n;
 (T^{(0)},\ldots,T^{(m-1)},|, U^{(1)},\ldots,U^{(n)})
 \quad(\text{by Eq.~(\mref{eq:1star1})})\\
 &=\treeoo{\cdb o
\ocdx[2.5]{o}{a1}{170}{T^{(0)}}{170}
\ocdx[2.5]{o}{a2}{145}{T^{(1)}}{145}
\ocdx[2.5]{o}{a3}{115}{T^{(m-1)}}{115}
\cdx[2.5]{o}{a4}{90}{-90}
\ocdx[2.5]{o}{a5}{65}{U^{(1)}}{65}
\ocdx[2.5]{o}{a6}{35}{U^{(n-1)}}{35}
\ocdx[2.5]{o}{a7}{10}{U^{(n)}}{10}
\node at (157:1.3*\xch) {$x_1$};
\node[rotate=40] at (129:2.5*\xch) {$\cdots$};
\node at (102:1.5*\xch) {$x_m$};
\node at (75:1.5*\xch) {$y_1$};
\node[rotate=-40] at (49:2.5*\xch) {$\cdots$};
\node at (20:1.5*\xch) {$y_n$};
\node[below] at ($(o)!0.7!(a1)$) {$\alpha_0$};
\node[below] at ($(o)!0.8!(a7)$) {\tiny$\beta_{n}$};
\node[below] at ($(o)!0.9!(a2)$) {\tiny$\alpha_1$};
\node[left=-1pt] at ($(o)!0.65!(a3)$) {\tiny$\alpha_{m-1}$};
\node[right=1pt] at ($(o)!0.65!(a5)$) {\tiny$\beta_1$};
\node[right=1pt] at ($(o)!0.65!(a6)$) {\tiny$\beta_{n-1}$};
}
=\treeoo{\cdb o\ocdx[2]{o}{a1}{160}{T^{(0)}}{left}
\ocdx[2]{o}{a2}{120}{T^{(1)}}{above}
\ocdx[2]{o}{a3}{60}{T^{(m-1)}}{above}
\cdx[2]{o}{a4}{20}
\node at (140:\xch) {$x_1$};
\node at (90:\xch) {$\cdots$};
\node at (40:\xch) {$x_m$};
\node[below] at ($(o)!0.7!(a1)$) {$\alpha_0$};
\node[right] at ($(o)!0.75!(a2)$) {$\alpha_1$};
\node[right] at ($(o)!0.75!(a3)$) {$\alpha_{m-1}\!$};
}\diamond \treeoo{\cdb o\cdx[2]{o}{a1}{160}
\ocdx[2]{o}{a2}{120}{U^{(1)}}{above}
\ocdx[2]{o}{a3}{60}{U^{(n-1)}}{above}
\ocdx[2]{o}{a4}{20}{U^{(n)}}{right}
\node at (140:\xch) {$y_1$};
\node at (90:\xch) {$\cdots$};
\node at (40:\xch) {$y_n$};
\node[below] at ($(o)!0.7!(a4)$) {$\beta_{n}$};
\node[right] at ($(o)!0.75!(a2)$) {$\beta_1$};
\node[right] at ($(o)!0.75!(a3)$) {$\beta_{n-1}\!$};
}\quad(\text{by Eq.~(\mref{eq:def3})})\\
&=T\diamond U=T\cdot' U\quad(\text{by Eq.~(\mref{eq:trid})}).
\end{align*}
\noindent {\bf Case 6:} $T^{(m)}\neq |=U^{(0)}$. Then
\begin{align*}
T\cdot U=&\bigv x_1,\ldots,x_{m-1},x_m,y_1,\ldots,y_n;m+n+1;\alpha_0,\ldots,\alpha_{m-1},\alpha_m,\,
\beta_1,\ldots,\beta_n;
 (T^{(0)},\ldots,T^{(m-1)},T^{(m)}\succ_{\alpha_m} |+T^{(m)}\prec_1 |+T^{(m)}\cdot |, U^{(1)},\ldots,U^{(n)})\\
 &\hspace{3cm}(\text{by Eq.~(\mref{eq:tdcdot})})\\
 &=\bigv x_1,\ldots,x_{m-1},x_m,y_1,\ldots,y_n;m+n+1;\alpha_0,\ldots,\alpha_{m-1},\alpha_m,
 \,\beta_1,
 \ldots,\beta_n;
 (T^{(0)},\ldots,T^{(m-1)},T^{(m)}, U^{(1)},\ldots,U^{(n)})\\
 &\hspace{2cm}(\text{by Item~\mref{it:trida} in Definition~\mref{defn:trirec} and Eq.~(\mref{eq:con})})\\
 &=\treeoo{\cdb o
\ocdx[2.5]{o}{a1}{170}{T^{(0)}}{170}
\ocdx[2.5]{o}{a2}{145}{T^{(1)}}{145}
\ocdx[2.5]{o}{a3}{115}{T^{(m-1)}}{115}
\ocdx[2.5]{o}{a4}{90}{T^{(m)}}{90}
\ocdx[2.5]{o}{a5}{65}{U^{(1)}}{65}
\ocdx[2.5]{o}{a6}{35}{U^{(n-1)}}{35}
\ocdx[2.5]{o}{a7}{10}{U^{(n)}}{10}
\node at (157:1.3*\xch) {$x_1$};
\node[rotate=40] at (129:2.5*\xch) {$\cdots$};
\node at (102:1.5*\xch) {$x_m$};
\node at (75:1.5*\xch) {$y_1$};
\node[rotate=-40] at (49:2.5*\xch) {$\cdots$};
\node at (20:1.5*\xch) {$y_n$};
\node[below] at ($(o)!0.7!(a1)$) {$\alpha_0$};
\node[below] at ($(o)!0.8!(a7)$) {\tiny$\beta_{n}$};
\node[below] at ($(o)!0.9!(a2)$) {\tiny$\alpha_1$};
\node[left=-1pt] at ($(o)!0.65!(a3)$) {\tiny$\alpha_{m-1}$};
\node[right=1pt] at ($(o)!0.65!(a5)$) {\tiny$\beta_1$};
\node[right=1pt] at ($(o)!0.65!(a6)$) {\tiny$\beta_{m-1}$};
\node[right] at ($(o)!0.8!(a4)$) {\tiny$\alpha_{m}$};
}
=\treeoo{\cdb o\ocdx[2]{o}{a1}{160}{T^{(0)}}{left}
\ocdx[2]{o}{a2}{120}{T^{(1)}}{above}
\ocdx[2]{o}{a3}{60}{T^{(m-1)}}{above}
\ocdx[2]{o}{a4}{20}{T^{(m)}}{right}
\node at (140:\xch) {$x_1$};
\node at (90:\xch) {$\cdots$};
\node at (40:\xch) {$x_m$};
\node[below] at ($(o)!0.7!(a1)$) {$\alpha_0$};
\node[below] at ($(o)!0.7!(a4)$) {$\alpha_{m}$};
\node[right] at ($(o)!0.75!(a2)$) {$\alpha_1$};
\node[right] at ($(o)!0.75!(a3)$) {$\alpha_{m-1}\!$};
}\diamond \treeoo{\cdb o\cdx[2]{o}{a1}{160}
\ocdx[2]{o}{a2}{120}{U^{(1)}}{above}
\ocdx[2]{o}{a3}{60}{U^{(n-1)}}{above}
\ocdx[2]{o}{a4}{20}{U^{(n)}}{right}
\node at (140:\xch) {$y_1$};
\node at (90:\xch) {$\cdots$};
\node at (40:\xch) {$y_n$};
\node[below] at ($(o)!0.7!(a4)$) {$\beta_{n}$};
\node[right] at ($(o)!0.75!(a2)$) {$\beta_1$};
\node[right] at ($(o)!0.75!(a3)$) {$\beta_{n-1}\!$};
}\quad(\text{by Eq.~(\mref{eq:def4})})\\
&=T\diamond U=T\cdot' U\quad(\text{by Eq.~(\mref{eq:trid})}).
\end{align*}
\noindent {\bf Case 7:} $T^{(m)}= |\neq U^{(0)}$. This case is similar to Case 6.

\noindent {\bf Case 8:} $T^{(m)}\neq |\neq U^{(0)}$. Then
{\small{\begin{align*}
T\cdot U&=\bigv x_1,\ldots,x_{m-1},x_m,y_1,\ldots,y_n;m+n+1;\alpha_0,\ldots,\alpha_{m-1},
\alpha_m\beta_0,\,\beta_1,
 \ldots,\beta_n;
 (T^{(0)},\ldots,T^{(m-1)},T^{(m)}\succ_{\alpha_m} U^{(0)}+T^{(m)}\prec_{\beta_0} U^{(0)}+T^{(m)}\cdot U^{(0)},\\
 &\hspace{5cm} U^{(1)},\ldots,U^{(n)})\quad(\text{by Eq.~(\mref{eq:tdcdot})})\\
 &=\bigv x_1,\ldots,x_{m-1},x_m,y_1,\ldots,y_n;m+n+1;\alpha_0,\ldots,\alpha_{m-1},
 \alpha_m\beta_0,\,\beta_1,
 \ldots,\beta_n;
 (T^{(0)},\ldots,T^{(m-1)},T^{(m)}\succ'_{\alpha_m} U^{(0)}+T^{(m)}\prec'_{\beta_0} U^{(0)}+T^{(m)}\cdot' U^{(0)},\\
 & \hspace{5cm}U^{(1)},\ldots,U^{(n)})\quad(\text{by the induction hypothesis})\\
 &=\bigv x_1,\ldots,x_{m-1},x_m,y_1,\ldots,y_n;m+n+1;\alpha_0,\ldots,\alpha_{m-1},
 \alpha_m\beta_0,\,\beta_1,
 \ldots,\beta_n;
 (T^{(0)},\ldots,T^{(m-1)},B^{+}_{\alpha_m}(T^{(m)})\diamond U^{(0)}+T^{(m)}\diamond B^{+}_{\beta_0}(U^{(0)})+T^{(m)}\diamond U^{(0)},\\
 & \hspace{5cm}U^{(1)},\ldots,U^{(n)})\quad(\text{by Eq.~(\mref{eq:trid})})\\
 &=\treeoo{\cdb o
\ocdx[2.5]{o}{a1}{170}{T^{(0)}}{170}
\ocdx[2.5]{o}{a2}{145}{T^{(1)}}{145}
\ocdx[2.5]{o}{a3}{115}{T^{(m-1)}}{115}
\ocdx[4.0]{o}{a4}{90}{B^{+}_{\alpha_m}(T^{(m)})\diamond U^{(0)}}{90}
\ocdx[2.5]{o}{a5}{65}{U^{(1)}}{65}
\ocdx[2.5]{o}{a6}{35}{U^{(n-1)}}{35}
\ocdx[2.5]{o}{a7}{10}{U^{(n)}}{10}
\node at (157:1.3*\xch) {$x_1$};
\node[rotate=40] at (129:2.5*\xch) {$\cdots$};
\node at (102:1.5*\xch) {$x_m$};
\node at (75:1.5*\xch) {$y_1$};
\node[rotate=-40] at (49:2.5*\xch) {$\cdots$};
\node at (20:1.5*\xch) {$y_n$};
\node[below] at ($(o)!0.7!(a1)$) {$\alpha_0$};
\node[below] at ($(o)!0.8!(a7)$) {\tiny$\beta_{n}$};
\node[below] at ($(o)!0.9!(a2)$) {\tiny$\alpha_1$};
\node[left=-1pt] at ($(o)!0.65!(a3)$) {\tiny$\alpha_{m-1}$};
\node[right=1pt] at ($(o)!0.65!(a5)$) {\tiny$\beta_1$};
\node[right=1pt] at ($(o)!0.65!(a6)$) {\tiny$\beta_{m-1}$};
\node[right] at ($(o)!0.9!(a4)$) {\tiny$\alpha_{m}\beta_0$};
}+\treeoo{\cdb o
\ocdx[2.5]{o}{a1}{170}{T^{(0)}}{170}
\ocdx[2.5]{o}{a2}{145}{T^{(1)}}{145}
\ocdx[2.5]{o}{a3}{115}{T^{(m-1)}}{115}
\ocdx[4.0]{o}{a4}{90}{T^{(m)}\diamond B^{+}_{\beta_0}(U^{(0)})}{90}
\ocdx[2.5]{o}{a5}{65}{U^{(1)}}{65}
\ocdx[2.5]{o}{a6}{35}{U^{(n-1)}}{35}
\ocdx[2.5]{o}{a7}{10}{U^{(n)}}{10}
\node at (157:1.3*\xch) {$x_1$};
\node[rotate=40] at (129:2.5*\xch) {$\cdots$};
\node at (102:1.5*\xch) {$x_m$};
\node at (75:1.5*\xch) {$y_1$};
\node[rotate=-40] at (49:2.5*\xch) {$\cdots$};
\node at (20:1.5*\xch) {$y_n$};
\node[below] at ($(o)!0.7!(a1)$) {$\alpha_0$};
\node[below] at ($(o)!0.8!(a7)$) {\tiny$\beta_{n}$};
\node[below] at ($(o)!0.9!(a2)$) {\tiny$\alpha_1$};
\node[left=-1pt] at ($(o)!0.65!(a3)$) {\tiny$\alpha_{m-1}$};
\node[right=1pt] at ($(o)!0.65!(a5)$) {\tiny$\beta_1$};
\node[right=1pt] at ($(o)!0.65!(a6)$) {\tiny$\beta_{m-1}$};
\node[right] at ($(o)!0.9!(a4)$) {\tiny$\alpha_{m}\beta_0$};
}+\treeoo{\cdb o
\ocdx[2.5]{o}{a1}{170}{T^{(0)}}{170}
\ocdx[2.5]{o}{a2}{145}{T^{(1)}}{145}
\ocdx[2.5]{o}{a3}{115}{T^{(m-1)}}{115}
\ocdx[4.0]{o}{a4}{90}{T^{(m)}\diamond U^{(0)}}{90}
\ocdx[2.5]{o}{a5}{65}{U^{(1)}}{65}
\ocdx[2.5]{o}{a6}{35}{U^{(n-1)}}{35}
\ocdx[2.5]{o}{a7}{10}{U^{(n)}}{10}
\node at (157:1.3*\xch) {$x_1$};
\node[rotate=40] at (129:2.5*\xch) {$\cdots$};
\node at (102:1.5*\xch) {$x_m$};
\node at (75:1.5*\xch) {$y_1$};
\node[rotate=-40] at (49:2.5*\xch) {$\cdots$};
\node at (20:1.5*\xch) {$y_n$};
\node[below] at ($(o)!0.7!(a1)$) {$\alpha_0$};
\node[below] at ($(o)!0.8!(a7)$) {\tiny$\beta_{n}$};
\node[below] at ($(o)!0.9!(a2)$) {\tiny$\alpha_1$};
\node[left=-1pt] at ($(o)!0.65!(a3)$) {\tiny$\alpha_{m-1}$};
\node[right=1pt] at ($(o)!0.65!(a5)$) {\tiny$\beta_1$};
\node[right=1pt] at ($(o)!0.65!(a6)$) {\tiny$\beta_{m-1}$};
\node[right] at ($(o)!0.9!(a4)$) {\tiny$\alpha_{m}\beta_0$};
}\\
&=\Biggl(\treeoo{\cdb o\ocdx[2]{o}{a1}{160}{T^{(0)}}{left}
\ocdx[2]{o}{a2}{120}{T^{(1)}}{above}
\ocdx[2]{o}{a3}{60}{T^{(m-1)}}{above}
\cdx[2]{o}{a4}{20}
\node at (140:\xch) {$x_1$};
\node at (90:\xch) {$\cdots$};
\node at (40:\xch) {$x_m$};
\node[below] at ($(o)!0.7!(a1)$) {$\alpha_0$};
\node[right] at ($(o)!0.75!(a2)$) {$\alpha_1$};
\node[right] at ($(o)!0.75!(a3)$) {$\alpha_{m-1}\!$};
}\diamond B^{+}_{\alpha_m\beta_0}\Big(B^{+}_{\alpha_m}(T^{(m)})\diamond U^{(0)}+T^{(m)}\diamond B^{+}_{\beta_0}(U^{(0)})+T^{(m)}\diamond U^{(0)}\Big)\Biggl)\diamond
\treeoo{\cdb o\cdx[2]{o}{a1}{160}
\ocdx[2]{o}{a2}{120}{U^{(1)}}{above}
\ocdx[2]{o}{a3}{60}{U^{(n-1)}}{above}
\ocdx[2]{o}{a4}{20}{U^{(n)}}{right}
\node at (140:\xch) {$y_1$};
\node at (90:\xch) {$\cdots$};
\node at (40:\xch) {$y_n$};
\node[below] at ($(o)!0.7!(a4)$) {$\beta_{n}$};
\node[right] at ($(o)!0.75!(a2)$) {$\beta_1$};
\node[right] at ($(o)!0.75!(a3)$) {$\beta_{n-1}\!$};
}\\
&\hspace{3cm}\quad(\text{by Eqs.~(\mref{eq:def4}) and ~(\mref{eq:def5})})\\
&=\Biggl(\treeoo{\cdb o\ocdx[2]{o}{a1}{160}{T^{(0)}}{left}
\ocdx[2]{o}{a2}{120}{T^{(1)}}{above}
\ocdx[2]{o}{a3}{60}{T^{(m-1)}}{above}
\cdx[2]{o}{a4}{20}
\node at (140:\xch) {$x_1$};
\node at (90:\xch) {$\cdots$};
\node at (40:\xch) {$x_m$};
\node[below] at ($(o)!0.7!(a1)$) {$\alpha_0$};
\node[right] at ($(o)!0.75!(a2)$) {$\alpha_1$};
\node[right] at ($(o)!0.75!(a3)$) {$\alpha_{m-1}\!$};
}
\diamond \Big(B^{+}_{\alpha_m}(T^{(m)})\diamond B^{+}_{\beta_0}(U^{(0)})\Big)\Biggl)\diamond \treeoo{\cdb o\cdx[2]{o}{a1}{160}
\ocdx[2]{o}{a2}{120}{U^{(1)}}{above}
\ocdx[2]{o}{a3}{60}{U^{(n-1)}}{above}
\ocdx[2]{o}{a4}{20}{U^{(n)}}{right}
\node at (140:\xch) {$y_1$};
\node at (90:\xch) {$\cdots$};
\node at (40:\xch) {$y_n$};
\node[below] at ($(o)!0.7!(a4)$) {$\beta_{n}$};
\node[right] at ($(o)!0.75!(a2)$) {$\beta_1$};
\node[right] at ($(o)!0.75!(a3)$) {$\beta_{n-1}\!$};
}\\
&\hspace{3cm}\text{(by Theorem ~\mref{thm:free} of weight $1$)}\\
&=\treeoo{\cdb o\ocdx[2]{o}{a1}{160}{T^{(0)}}{left}
\ocdx[2]{o}{a2}{120}{T^{(1)}}{above}
\ocdx[2]{o}{a3}{60}{T^{(m-1)}}{above}
\ocdx[2]{o}{a4}{20}{T^{(m)}}{right}
\node at (140:\xch) {$x_1$};
\node at (90:\xch) {$\cdots$};
\node at (40:\xch) {$x_m$};
\node[below] at ($(o)!0.7!(a1)$) {$\alpha_0$};
\node[below] at ($(o)!0.7!(a4)$) {$\alpha_{m}$};
\node[right] at ($(o)!0.75!(a2)$) {$\alpha_1$};
\node[right] at ($(o)!0.75!(a3)$) {$\alpha_{m-1}\!$};
}
\diamond \treeoo{\cdb o\ocdx[2]{o}{a1}{160}{U^{(0)}}{left}
\ocdx[2]{o}{a2}{120}{U^{(1)}}{above}
\ocdx[2]{o}{a3}{60}{U^{(n-1)}}{above}
\ocdx[2]{o}{a4}{20}{U^{(n)}}{right}
\node at (140:\xch) {$y_1$};
\node at (90:\xch) {$\cdots$};
\node at (40:\xch) {$y_n$};
\node[below] at ($(o)!0.7!(a1)$) {$\beta_0$};
\node[below] at ($(o)!0.7!(a4)$) {$\beta_{n}$};
\node[right] at ($(o)!0.75!(a2)$) {$\beta_1$};
\node[right] at ($(o)!0.75!(a3)$) {$\beta_{n-1}\!$};
}\quad(\text{by Eq.~(\mref{eq:def6})})\\
&=T\diamond U=T\cdot' U\quad(\text{by Eq.~(\mref{eq:trid})}).
\end{align*}}}
This completes the proof.
\end{proof}
\begin{remark}{\em In the rest of this paper, let $\bfk$ be a field.}
Theorem \ref{thm:main} is also true for any weight $\lambda\neq 0$, at the price of replacing the canonical inclusion $j$ by $\lambda^{-1}j$. A similar (and simpler) proof yields an analogon in weight zero case, namely that the free dendriform family algebra $\bigl(\DDF,(\prec_{\omega}, \succ_\omega)_{\omega \in \Omega}\bigr)$ on $X$ is a dendriform family subalgebra of the free Rota-Baxter family algebra of weight zero. The image of $\DDF$ in $\bfk\cxo T;$ is included in the linear span $\bigl(\bfk\caltbin,\diamond,(B^{+}_{\omega})_{\omega \in \Omega}\bigr)$ of typed decorated rooted trees in which any vertex has exactly one or two incoming edges. Details are left to the reader.
\end{remark}

\subsection{Universal enveloping algebras of (tri)dendriform family algebras}
In this subsection, we mainly prove that the free Rota-Baxter family algebra $\bigl(\bfk \cxo T;, \diamond, (B^+_\omega)_{\omega\in\Omega}\bigr)$ is the universal enveloping algebra of the free dendriform (resp.~tridendriform) family algebra $\DDF$ (resp.~$\DTF$).

\begin{defn}
Let $D$ be a dendriform (resp.~tridendriform) family algebra. A {\bf universal enveloping Rota-Baxter family algebra} of weight $\lambda$ of $D$ is a Rota-Baxter family algebra $\rbf(D):=\rbf_\lambda(D)$, together with a dendriform (resp.~tridendriform) family algebra morphism $j: D\ra \rbf(D)$ such that for any weight $\lambda$ Rota-Baxter family algebra $A$ and dendriform (resp.~tridendriform) family algebra morphism $f:D\ra A$, there is a unique Rota-Baxter family algebra morphism $\bar{f}:\rbf(D)\ra A$ such that $\bar{f}\circ j=f.$
\end{defn}

\delete{
\begin{lemma}\cite{ZGM}
Let $X$ be a set and let $\Omega$ be a semigroup.
Then $(\DDF, (\prec_\omega,\succ_\omega)_{\omega\in\Omega})$ is the free dendriform family algebra, generated by $\biggl\{\stree x\Bigm| x\in X\biggr\}.$
\mlabel{lemma:dend}
\end{lemma}

\begin{lemma}\cite{ZGM}
Let $X$ be a set and let $\Omega$ be a semigroup.
Then $(\DTF,\,(\prec_\omega,\succ_\omega)_{\omega\in\Omega},\cdot)$ is the free tridendriform family algebra, generated by  $\biggl\{\stree x\Bigm|x\in X\biggr\}.$
\mlabel{lemma:trid}
\end{lemma}
}
From Theorem~\mref{thm:sub1}, we know that the free dendriform family algebra $\DDF$ is a dendriform family subalgebra of $\bfk \cxo T; $. Let $j:\DDF\ra \bfk \cxo T; $ be the natural embedding map.

\begin{theorem}
 Let $X$ be a set and let $\Omega$ be a semigroup. Then $\bfk \cxo T;$, together with the canonical inclusion map $j$, is the universal enveloping weight zero Rota-Baxter family algebra of the free dendriform family algebra $\DDF.$
\mlabel{thm:ddf}
\end{theorem}

\begin{proof}
From Lemma~\mref{lem:dendt}, a Rota-Baxter family algebra of weight $0$ induces a dendriform family algebra. Let $(A,(Q_\omega)_\omega)$ be any Rota-Baxter family algebra of weight $0$ and let $g:\DDF\ra A$ be a dendriform family algebra morphism. We need to find a Rota-Baxter family algebra morphism $\bar{g}:\bfk\cxo F;\ra A$
making the following diagram commute.

$$
\xymatrix@C=2cm{
  \DDF \ar[d]_{g} \ar[r]^-{j}
                & \bfk \cxo T;
                \ar@{.>}[dl]_{\bar{g}} \\
  A
                          }
$$
\noindent From Lemma~\mref{thm:free1} and Theorem~\mref{thm:sub1}, we know that $\DDF$ is generated by $\biggl\{\stree x\Bigm| x\in X\biggr\}$ and $\bfk \cxo T; $ is generated by $\DDF$ under the operations $(B^+_\omega)_{\omega\in\Omega}$ and $\diamond$. Hence $\bfk \cxo T; $ is generated by $\biggl\{\stree x\Bigm| x\in X\biggr\}$ under the operations $(B^+_\omega)_{\omega\in\Omega}$ and $\diamond.$ By the freeness of $\bfk \cxo T; $, there is a unique Rota-Baxter family algebra morphism $\bar{g}: \bfk \cxo T; \ra A$, so that $\bar{g}\Big(\stree x\Big):=g\Big(\stree x\Big)=a_x$. Then $g\Big(\stree x\Big)=\bar{g}\circ j\Big(\stree x\Big).$ So $g=\bar{g}\circ j$ and the diagram commutes.
\end{proof}

From Theorem~\mref{thm:main}, we know that the free tridendriform family algebra $\DTF$ is a tridendriform family subalgebra of $\bfk \cxo T; $. Let $j:\DTF\ra \bfk \cxo T; $ be the natural embedding map.

\begin{theorem}
Let $X$ be a set and let $\Omega$ be a semigroup. Let $\lambda$ be any nonzero element of the base field $\bfk$. Then $\bfk \cxo T; $, together with $\lambda^{-1}j$, is the universal enveloping weight $\lambda$ Rota-Baxter family algebra of the free tridendriform family algebra $\DTF.$
\mlabel{thm:dtf}
\end{theorem}

\begin{proof}
The proof is similar to Theorem~\mref{thm:ddf}. Details are left to the reader.
\end{proof}


\smallskip

\noindent {\bf Acknowledgments}:
This work was supported by the National Natural Science Foundation
of China (Grant No.\@ 11771191 and 11861051).

\end{document}